\documentclass[oneside,english,11pt]{amsart}
\usepackage{comment}
\usepackage[latin9]{inputenc}
\usepackage{float}
\usepackage{mathtools}
\usepackage{amstext}
\usepackage{amsthm}
\usepackage{amsmath}
\usepackage{amssymb}
\usepackage{geometry}
\usepackage{enumitem,xcolor}
\usepackage[backref=page,colorlinks,citecolor=blue,bookmarks=true]{hyperref}
\usepackage{cleveref}
\hypersetup{linkcolor=blue,  citecolor=blue, urlcolor=blue}
\usepackage[fontsize=11pt]{scrextend}
\usepackage{xspace}

\usepackage[nobysame,abbrev,alphabetic]{amsrefs}
\usepackage{algorithmic}

\usepackage{microtype}
\usepackage{tikz}
\usepackage{changepage}
\usepackage{nicefrac,xfrac}

\usepackage{mathrsfs}
\floatname{algorithm}{Algorithm}

\makeatletter

\floatstyle{ruled}
\newfloat{algorithm}{tbp}{loa}
\providecommand{\algorithmname}{Algorithm}
\floatname{algorithm}{\protect\algorithmname}

\numberwithin{equation}{section}
\numberwithin{figure}{section}
\theoremstyle{plain}
\newtheorem{thm}{\protect\theoremname}[section]
\DeclareMathOperator*{\Ex}{\mathbb{E}}
\DeclareMathOperator*{\Pro}{\mathbb{P}}

\newtheorem{conj}[thm]{Conjecture}
\newtheorem{cor}[thm]{Corollary}%{\protect\Corollaryname}
\newtheorem*{thm*}{\protect\theoremname}
\theoremstyle{definition}

\newtheorem*{problem*}{Problem}
\theoremstyle{remark}
\newtheorem*{rem*}{\protect\remarkname}
\theoremstyle{remark}
\newtheorem{rem}[thm]{\protect\remarkname}

\theoremstyle{definition}
\newtheorem{defn}[thm]{\protect\definitionname}
\theoremstyle{plain}
\newtheorem{prop}[thm]{\protect\propositionname}
\theoremstyle{plain}
\newtheorem{fact}[thm]{\protect\factname}
\theoremstyle{plain}
\newtheorem{lem}[thm]{\protect\lemmaname}
\theoremstyle{definition}
\newtheorem{example}[thm]{\protect\examplename}
\theoremstyle{plain}
\newtheorem{claim}[thm]{Claim}%{\protect\claimname}
\theoremstyle{plain}
\usepackage{amsrefs}

\usepackage{algorithmic}

\usepackage{microtype}

\usepackage{changepage}

\floatname{algorithm}{Algorithm}

\let\originalleft\left
\let\originalright\right
\renewcommand{\left}{\mathopen{}\mathclose\bgroup\originalleft}
\renewcommand{\right}{\aftergroup\egroup\originalright}

\newcommand{\class}[1]{\ensuremath{\mathsf{#1}}\xspace}

\newcommand{\MIP}{\class{MIP}} %
\newcommand{\NEXP}{\class{NEXP}}
\newcommand{\TMIP}{\class{TailoredMIP}} %
\newcommand{\RE}{\class{RE}} %

\usepackage{babel}
\makeatletter
\addto\extrasfrench{%
   \providecommand{\fg}{\ifdim\lastskip>\z@\unskip\fi~\frqq}%
}

\makeatother
\addto\captionsenglish{\renewcommand{\definitionname}{Definition}}
\addto\captionsenglish{\renewcommand{\factname}{Fact}}
\addto\captionsenglish{\renewcommand{\lemmaname}{Lemma}}
\addto\captionsenglish{\renewcommand{\problemname}{Problem}}
\addto\captionsenglish{\renewcommand{\propositionname}{Proposition}}
\addto\captionsenglish{\renewcommand{\remarkname}{Remark}}
\addto\captionsenglish{\renewcommand{\theoremname}{Theorem}}
\addto\captionsfrench{\renewcommand{\algorithmname}{Algorithm}}
\addto\captionsfrench{\renewcommand{\definitionname}{Definition}}
\addto\captionsfrench{\renewcommand{\factname}{Fait}}
\addto\captionsfrench{\renewcommand{\lemmaname}{Lemme}}
\addto\captionsfrench{\renewcommand{\problemname}{Problème}}
\addto\captionsfrench{\renewcommand{\propositionname}{Proposition}}
\addto\captionsfrench{\renewcommand{\remarkname}{Remarque}}
\addto\captionsfrench{\renewcommand{\theoremname}{Théorème}}
\providecommand{\definitionname}{Definition}
\providecommand{\factname}{Fact}
\providecommand{\lemmaname}{Lemma}
\providecommand{\problemname}{Problem}
\providecommand{\propositionname}{Proposition}
\providecommand{\remarkname}{Remark}
\providecommand{\theoremname}{Theorem}
\providecommand{\examplename}{Example}
\begin{document}
\global\long\def\eps{\varepsilon}

\newcommand{\Sym}{{\rm Sym}}
\newcommand{\Bad}{{\rm Bad}}
\newcommand{\erg}{{\rm erg}}
\newcommand{\Rob}{{\rm Rob}}
\newcommand{\Sch}{{\rm Sch}}
\newcommand{\Cay}{{\rm Cay}}
\newcommand{\IP}{{\rm IP}}
\newcommand{\IRS}{{\rm IRS}}
\newcommand{\Stab}{{\rm Stab}}
\newcommand{\Id}{{\rm Id}}
\newcommand{\Prob}{{\rm Prob}}
\newcommand{\Aug}{{\frak Aug}}

\newcommand{\val}{{\rm val}}
\newcommand{\fd}{{\rm fd}}
\newcommand{\sof}{{\rm sof}}
\newcommand{\cF}{\mathcal{F}}
\newcommand{\FF}{\mathbb{F}}
\newcommand{\RR}{\mathbb{R}}
\newcommand{\QQ}{\mathbb{Q}}
\newcommand{\NN}{\mathbb{N}}
\newcommand{\ZZ}{\mathbb{Z}}
\newcommand{\cP}{\mathcal{P}}
\newcommand{\cT}{\mathcal{T}}
\newcommand{\cQ}{\mathcal{Q}}
\newcommand{\cV}{\mathcal{V}}
\newcommand{\cC}{\mathcal{C}}
\newcommand{\cS}{\mathcal{S}}
\newcommand{\cK}{\mathcal{K}}
\newcommand{\cD}{\mathcal{D}}
\newcommand{\cG}{\mathcal{G}}
\newcommand{\Sub}{\frak{sub}}
\newcommand{\WH}{{\rm WH}}
\newcommand{\Img}{{\rm Im}}

\newcommand{\frL}{\frak{L}}
\newcommand{\frR}{\frak{R}}
\newcommand{\frS}{\frak{s}}

\newcommand{\sJ}{\mathsf{J}}
\newcommand{\sX}{\mathsf{X}}
\newcommand{\sY}{\mathsf{Y}}
\newcommand{\sZ}{\mathsf{Z}}
\newcommand{\sU}{\mathsf{U}}
\newcommand{\sP}{\mathsf{P}}
\newcommand{\sQ}{\mathsf{Q}}
\newcommand{\sA}{\mathsf{A}}
\newcommand{\sB}{\mathsf{B}}
\newcommand{\sC}{\mathsf{C}}
\newcommand{\sD}{\mathsf{D}}
\newcommand{\sE}{\mathsf{E}}
\newcommand{\sR}{\mathsf{R}}

\newcommand{\Sam}{\mathsf{Sam}}
\newcommand{\Loc}{\mathsf{Loc}}
\newcommand{\Dec}{\mathsf{Dec}}

\setcounter{tocdepth}{1}

\newcommand{\tnote}[1]{\textcolor{blue}{\footnotesize{\bf (Thomas:} {#1}{\bf ) }}}
\newcommand{\mnote}[1]{\textcolor{red}{\footnotesize{\bf (Michael:} {#1}{\bf ) }}}

\title{The Aldous--Lyons Conjecture I: Subgroup Tests}

\author[L.\ Bowen]{Lewis Bowen}
\address{Lewis Bowen\hfill\break
Department of Mathematics \hfill\break 1 University Station C1200 \hfill\break University of Texas at Austin\hfill\break Austin, TX, 78712 USA.}
\email{lpbowen@math.utexas.edu}

\author[M.\ Chapman]{Michael Chapman}
\address{Michael Chapman\hfill\break
	Courant Institute of Mathematical Sciences\hfill\break
	New York University,\hfill\break 251 Mercer St, New York, NY 10012, USA.}
\email{mc9578@nyu.edu}

\author[A.\ Lubotzky]{Alexander Lubotzky}
\address{Alexander Lubotzky\hfill\break
	Weizmann institute of Science\hfill\break
	Rehovot, Israel.}
\email{alex.lubotzky@mail.huji.ac.il}

\author[T.\ Vidick]{Thomas Vidick}
\address{Thomas Vidick\hfill\break
	Weizmann institute of Science\hfill\break
	Rehovot, Israel.}
\email{thomas.vidick@weizmann.ac.il}

\begin{abstract}
    This paper, and its companion \cite{Tailored_MIPRE}, are devoted to  a negative resolution of the Aldous--Lyons Conjecture \cites{Aldous_Lyons_Conj,Aldous--Lyons_conj_blogpost}. This conjecture, originated in probability theory, is well known (cf. \cite{Gelander_ICM2018}) to be equivalent to the statement that every invariant random subgroup of the free group is co-sofic. We disprove this last statement.
    \\
    
   In this part we introduce \emph{subgroup tests}. These tests are finite distributions over continuous functions from the space of subgroups of the free group to $\{0,1\}$. Subgroup tests provide a general framework in which one can study invariant random subgroups of the free group. Classical notions such as group soficity and  group stability arise naturally in this framework. By the correspondence between subgroups of the free group and Schreier graphs, one can view subgroup tests as a property testing model for certain edge-labeled graphs. This correspondence  also provides the connection to random networks.
    
    Subgroup tests have \emph{values}, which are their asymptotic optimal expectations when integrated against co-sofic invariant random subgroups.  Our first main result is that, if every invariant random subgroup of the free group is co-sofic, then one can  approximate the value of a subgroup test up to any positive additive constant. 

    Our second main result is an essentially value preserving correspondence between certain non-local games and subgroup tests. By composing this correspondence with a stronger variant of the reduction in $\MIP^*=\RE$ \cite{MIPRE},  proved in the companion paper \cite{Tailored_MIPRE}, we deduce that approximating the sofic value of a subgroup test  is as hard as the Halting Problem, and in particular, undecidable. The combination of our two main results proves the existence of non co-sofic invariant random subgroups of the free group.

\end{abstract}
\maketitle
\tableofcontents

\section{Introduction}
In their seminal paper \cite{Aldous_Lyons_Conj}, Aldous and Lyons ask whether every unimodular network is sofic. Conditional on a certain result which will be addressed in a companion paper \cite{Tailored_MIPRE}, we prove the answer is `no'.
This conjecture, which originated in probability theory, can also be expressed in the language of invariant random subgroups, and this is the one we will use in this paper.
We explain and motivate this conjecture before discussing our approach. The formal definitions and results will start in  Section \ref{subsec_subgroup_tests}.

\pagebreak
\subsection*{The Aldous--Lyons Conjecture}

\subsubsection*{What is soficity?}
Roughly speaking, a countable group $\Gamma$ is sofic if it admits a sequence of partial actions on finite sets which approximates the action of $\Gamma$ on itself by left-translations. If $\Gamma=\langle S | R\rangle$ is finitely presented then this is equivalent to the existence of a sequence of finite graphs which converge to the Cayley graph of $\Gamma$ in the sense of Benjamini--Schramm, which is a kind of local-on-average convergence.

To be more precise, a \emph{rooted graph} is a pair $(G,o)$ where $G=(V,E)$ is a graph and $o \in V$ is a distinguished vertex called the \emph{root}. A random rooted graph $(G,o)$ is \emph{sofic} if there exists a sequence $(G_i)_{i \in \NN}$ of finite graphs such that, if $o_i$ is a uniformly random vertex of $G_i$, then $(G_i,o_i)$ converges to $(G,o)$ in distribution (this means: for every $r>0$, the radius $r$ neighborhood of the root in $G_i$ converges in distribution to the radius $r$ neighborhood of the root in $G$). This notion of convergence is due to Benjamini and Schramm \cite{MR1873300}: They proved that if each finite graph $G_i$ is planar and there is a uniform degree bound, then the limit $(G,o)$ is almost surely recurrent. 

The notions of Benjamini--Schramm convergence and soficity naturally generalize from graphs to labeled graphs or networks \cite{Aldous_Lyons_Conj}, simplicial complexes \cite{MR2797963}, manifolds \cite{MR4520306}, measured equivalence relations \cite{MR2566316}, measured groupoids 
\cite{dykema-2014} and most generally, measured metric spaces \cite{khezeli2023unimodular}. 

Again, a finitely presented group is sofic if one of its Cayley graphs is sofic.\footnote{If one of the Cayley graphs is sofic, then all those that use finitely many generators are.}  For example, amenable groups are sofic. In fact, the Cayley graph of an amenable group admits a sequence of finite sub-graphs in which the isoperimetric ratio (number of boundary vertices to number of vertices) tends to zero, which implies this sequence Benjamini--Schramm converges to the Cayley graph.  Also, residually finite groups are sofic because they admit sequences of finite quotient groups whose Cayley graphs approximate the Cayley graph of the given group. 

Gromov implicitly introduced sofic groups in  \cite{MR1694588}. He was motivated by Gottshalk's Conjecture: If $\Gamma$ is a countable group and $k \in \NN$, then any $\Gamma$-equivariant continuous map $\phi:[k]^\Gamma \to [k]^\Gamma$ which is injective is necessarily surjective. This is obviously true when $\Gamma$ is finite. Gromov proved that it holds when $\Gamma$ is sofic. Benjy Weiss gave another proof and coined the term `sofic' from the Hebrew word for finite \cite{weiss-2000}.

%The property has been used to study , group rings \cite{}, graph polynomials \cite{}, and it useful in constructing dynamical invariants of group actions \cite{}. 

      The soficity property has been useful in obtaining positive results about $L^2$ and spectral theory invariants \cite{MR3664810}, group rings \cites{MR2089244, MR2417890}, invariant couplings of random fields \cite{MR3503036}, and in constructing dynamical invariants of groups actions \cites{MR2552252, MR2854085, MR3077882, MR3132735, MR3993930}. See \cites{Cap_Lup_Sofic_Hyperlinear_book, MR3821628} for more background on sofic groups. The basic idea is roughly the same in each case: One transfers properties of the finite approximating objects to properties of the limit.  

      It is a major open problem to determine whether all countable groups are sofic. We do not directly address this problem because the class of objects we study, described next, is more general than the class of groups.

%A finitely generated group $G=\langle S | R\rangle$ is {\bf sofic} if its Cayley graph with respect to some choice of finite generating set, is sofic. If this is the case, then it does not depend on the choice of Cayley graph. More generally, if $G$ is countable but not necessarily finitely generated then we say $G$ is sofic if it is a union of sofic subgroups.

\subsubsection*{What is unimodularity?}

The main tool for analyzing Benjamini--Schramm limits is called the \emph{Mass Transport Principle} which makes precise the intuitive notion of statistical homogeneity. To explain, a \emph{doubly-rooted graph} is an ordered triple $(G,o_1,o_2)$, where $o_1,o_2$ are vertices of $G$. A \emph{transport function} is a function $f$, satisfying some measurability condition, which takes as input a doubly rooted graph and outputs a non-negative real number. The interpretation is that $f(G,o_1,o_2)$ is the amount of mass the first root $o_1$ sends to the second root $o_2$. So if $(G,o)$ is a random rooted graph then $\Ex_{(G,o)} \left[\sum_{x \in V(G)} f(G,o,x)\right]$ is the average amount of mass sent out of the root. Symmetrically, $\Ex_{(G,o)} \left[\sum_{x \in V(G)} f(G,x,o)\right]$ is the average amount of mass sent into the root. If these two quantities are equal for every transport function $f$, then we say $(G,o)$ satisfies the \emph{Mass Transport Principle}. Equivalently, we say it is \emph{unimodular} (this is the term used by Aldous and Lyons in \cite{Aldous_Lyons_Conj}). Alternatively, unimodularity can be formulated in terms of graphings or as invariance with respect to the root-changing equivalence relation on the space of rooted graphs. The former arises from the ergodic theory of measured equivalence relations (e.g. \cite{MR1164598}). The latter point-of-view was introduced in \cite{MR1631732} and further developed in \cite{MR3504507}.

The term unimodular comes from the following special case. Suppose $G$ is a connected transitive graph; this means that its automorphism group $\textrm{Aut}(G)$ acts transitively on its vertex set. Then $(G,o)$ is unimodular if and only if the automorphism group $\textrm{Aut}(G)$ is unimodular in the sense that its left and right Haar measures agree.

The Mass Transport Principle arose in H\"{a}ggstr\"{o}m's study of percolation clusters on trees \cite{MR1457624}. In fact, it is an exercise to show that if $(G,o)$ is unimodular and $p \in [0,1]$ is given, then the Bernoulli $p$-percolation cluster containing $o$ is also unimodular. Unimodularity was further developed in \cite{blps-group-perc} (see also \cite{MR2883390} where it was used to construct a version of the Euler characteristic for unimodular planar maps).

It is straightforward to verify that (1) unimodularity is closed under weak limits and (2) if $G$ is a finite graph and $o$ is a uniformly random vertex of $G$ then $(G,o)$ is unimodular. It follows from these observations that soficity implies unimodularity. Question 10.1 of \cite{Aldous_Lyons_Conj} asks whether unimodularity implies soficity.

As above, unimodularity has been generalized to random rooted labeled graphs, simplicial complexes, manifolds and so on. Benjamini and Schramm note that many results which are known to hold in the deterministic setting of unimodular transitive graphs can be generalized to  unimodular random graphs \cite{MR1873300}. Aldous and Lyons demonstrate this with results about random walks, amenability, ends of graphs, spanning forests, percolation, and so on \cite{Aldous_Lyons_Conj}.

\subsubsection*{What is an invariant random subgroup?}\label{sec:intro_what_is_IRS}

As a first step towards our goal, let us describe the algebraic formulation of the Aldous--Lyons Conjecture in terms of invariant random subgroups, which were introduced in \cite{MR2749291}, \cite{abert2014kesten}, \cite{MR3193754} (and implicitly in \cite{stuck1994stabilizers}). The concept of invariant random subgroups has been profitably studied in the context of locally compact groups (e.g. \cites{stuck1994stabilizers, MR3664810,Gelander_ICM2018}). 
However, we will restrict our focus to the case of most relevance to this paper: finitely generated groups. 
Already in this discrete setup, invariant random subgroups play an important role in graph convergence \cite{Hatami_Lovasz_Szegedi_Graph_limits}, stability properties of groups \cite{BLT} and the analysis of dynamical systems \cites{stuck1994stabilizers, MR3193754} for example.

So, fix a finitely generated group $\Gamma$. An \emph{invariant random subgroup} (or IRS) of $\Gamma$ is a random variable taking values in the space of all subgroups of $\Gamma$, whose law is invariant under the conjugation action of $\Gamma$. For example, every normal subgroup of $\Gamma$ is an IRS. More generally, if $H \le \Gamma$ has only finitely many conjugates, then a uniformly random sample of its conjugates is an IRS. As finite index subgroups have only finitely many conjugates, they induce  IRSs which are called \emph{elementary}. To generate more examples, one can take convex combinations and weak* limits of elementary IRSs, and the IRSs that arise this way are called \emph{co-sofic}. It turns out that for some groups every IRS is co-sofic, and for others there are non co-sofic IRSs (cf. \cite{BLT}). The Aldous--Lyons Conjecture can be formulated as follows: 
\begin{center}
    Are all IRSs of a (non-commutative) free group co-sofic?
\end{center} 
For the rest of the paper, this is the formulation that we tackle, and it is known --- and we explain it in the following paragraphs --- to be equivalent to the probability theoretic formulation on Benjamini--Schramm limits of finite graphs.

The law (or distribution) of an IRS is a Borel probability measure $\mu$ on $\mathfrak{sub}(\Gamma)$, the space of subgroups of $\Gamma$. We let $\IRS(\Gamma)$ denote the space of all such measures. By abuse of language, we say that $\mu$ is an IRS if $\mu \in \IRS(\Gamma)$.
To see the relation between the two formulations of the Aldous--Lyons Conjecture, note that if $H$ is a subgroup of $\Gamma$, then the Schreier coset graph of $_H\backslash^\Gamma$ with respect to some generating set $S$ is a rooted (directed, edge-labeled) graph, with the root being the coset of the identity.  This gives a map from the space  of subgroups $\mathfrak{sub}(\Gamma)$ to the space of isomorphism classes of rooted (edge-labeled, directed) graphs. 

If $\mu \in \IRS(\Gamma)$ then $\mu$ pushes forward under this map to a unimodular measure \cite[Proposition 14]{abert2014kesten}. 
 Moreover, if $\mu$ is elementary, then its pushforward is induced by uniformly choosing a root of some finite connected graph. Therefore, if $\mu$ is co-sofic, then its pushforward is sofic.

Now, consider the special case in which $\Gamma=\cF(S)$ is the free group generated by a finite set $S$. Suppose $\mu$ is an invariant random subgroup of $\Gamma$, and that the (random) Schreier coset graph of $\mu$ is sofic (in the directed, edge-labeled rooted graph category). This means there exists a sequence of finite directed, edge-labeled graphs $G_i$ which Benjamini--Schramm converge to the (random) Schreier coset graph of $\mu$. After perturbing the given sequence if necessary, we may assume that each $G_i$ is itself a Schreier coset graph with respect to some finite-index subgroup $H_i$ of $\cF(S)$ --- the existence of such perturbations is exactly where the \emph{freeness} of the group $\cF(S)$ plays a role. Moreover, if $\widetilde{H}_i$ is a uniformly random conjugate of $H_i$, then the distribution of $\widetilde{H}_i$ converges to $\mu$. Therefore, a positive solution to the Aldous--Lyons Conjecture implies that every IRS of $\cF(S)$ is co-sofic.\footnote{There is a standard ``decoration of graphs'' technique that allows one to show that the directed, edge labeled (with finitely many labels) version of the Aldous--Lyons conjecture is \textbf{equivalent} to the non-directed, non-edge labeled version. As B\'alint Vir\'ag pointed to us, the same technique is used, e.g., to show that every group is the automorphism group of some graph. As this is standard, we do not elaborate on it anymore.}

%So a positive answer to the   question of Aldous and Lyons  implies that every IRS is co-sofic. The converse is also true.

\subsection*{Our approach}
\subsubsection*{What is a subgroup test?}
The set $\IRS(\cF)$ of all IRSs of the free group and the set $\IRS_{\sof}(\cF)$ of co-sofic IRSs of the free group are both convex. Thus, the Aldous--Lyons Conjecture is asking whether certain convex sets can be separated. A natural way to distinguish between two convex sets is to use (continuous) linear functionals. So, in order to separate the above, one should search for a \textbf{rich enough} collection of functionals. \emph{Subgroup tests} and their \emph{values} will play this exact role.  

A \emph{challenge} is a continuous map from $\Sub(\cF)$ to $\{0,1\}$. A subgroup test $\cT$ is a probability distribution on a finite set of challenges. One can integrate any subgroup test against any probability measure over $\Sub(\cF)$. If $\pi$ is in $\IRS(\cF)$, then this integral is called the \emph{value} of the \emph{strategy} $\pi$ against the test $\cT$.
%\tnote{adding this here but it's a placeholder, I'm working on it:}
 Many natural and well studied properties of groups, specifically soficity and pointwise permutation stability, can be formulated as properties of subgroup tests. 

The terminology of ``tests'' and ``values'' originates in the field of \emph{interactive proofs} in computer science. 
An interactive proof designates an interaction between two entities, a ``verifier'' and a ``prover''. The goal of the prover is to convince the verifier of the validity of a certain claim, and the goal of the verifier is to test the prover so that it  accepts the interaction only if the claim  is indeed correct. 
Examples of ``claims'' studied in computer science are the $3$-colorability of a given input graph, or that a Turing machine whose description is passed as input halts.
In our approach, the methods of the verifier and the prover are restricted. The prover is required to sample an element of $\Sub(\cF)$. 
%\tnote{edited a bit the remainder of the paragraph; earlier version is commented below. Main change is to define value}
The sampling method itself is restricted to some sub-class of IRSs  (e.g., elementary IRSs); the choice of distribution according which to sample is the only degree of freedom the prover has, and is called its \emph{strategy}. The verifier gets access to (the indicator function of) the prover's sampled subgroup $H$, and it makes a decision whether to \textbf{accept} or \textbf{reject} this subgroup. 
To decide, the verifier holds some finite set $K \subseteq \cF$ --- known beforehand to the prover --- and according to which combinations of elements of $K$ are contained in $H$, it accepts or rejects; the \emph{rules} that control which combinations of elements result in acceptance and which combinations result in rejection are also know  to the prover beforehand. The \textbf{value} of a strategy in a game is the probability, taken over the prover's sampling and the verifier's probabilistic choices, that the verifier makes the decision to accept when interacting with a prover using this strategy. 

In spite of these restrictions, it is possible to verify hard computational problems in this setup.\footnote{For the expert, we note that restrictions on the  prover and verifier can affect both the completeness and soundness properties of a proof system. Hence, a priori, they may either lower or raise the complexity of the associated class of interactive proofs compared to classical single-prover interactive proofs.} For example, it is not difficult to design a subgroup test such that there is a value-$1$ strategy for this test if and only if a given graph is $3$-colorable. Jumping ahead, the correct analogy is with the theory of \emph{multiprover} interactive proofs, and even more specifically interactive proofs with quantum provers sharing entanglement, also known as non-local games. We explain below this connection, which plays an essential role in the second part of this paper.

%The sampling method itself is restricted to some sub-class of IRSs  (e.g., elementary IRSs); the choice of sampling method is the only degree of freedom the prover has, and is called its \emph{strategy}. The verifier gets access to (the indicator function of) the prover's sampled subgroup $H$, and is restricted to \textbf{accept} or \textbf{reject} only according to which combinations of elements of some fixed finite subset of $\cF$ are contained in $H$. In spite of these restrictions, it is possible to verify hard computational problems in this setup.\footnote{For the expert, we note that restrictions on the  prover and verifier can affect both the completeness and soundness properties of a proof system. Hence, a priori, they may either lower or raise the complexity of the associated class of interactive proofs compared to classical single-prover interactive proofs.} For example, it is not difficult to design a subgroup test such that there is a value-$1$ strategy for this test if and only if a given graph is $3$-colorable. Jumping ahead, the correct analogy is with the theory of \emph{multiprover} interactive proofs, and even more specifically interactive proofs with quantum provers sharing entanglement, also known as non-local games. We explain below this connection, which plays an essential role in the second part of our results.

\subsubsection*{What is a non-local game?}
Morally, \emph{non-local games} originated in Bell's resolution \cite{bell1964einstein} of the Einstein--Podolsky--Rosen paradox \cite{einstein1935can}. Einstein was famously concerned by the ability of space-time separated particles to \emph{correlate} in ways that might  violate relativity, a possibility that seemed to be suggested by the mathematical modeling of quantum mechanics. Einstein argued for a \emph{local hidden variable} theory --- namely, that the particles \textbf{did} share some information beforehand,  which allowed them to correlate --- and that physicists should search for this hidden information.
Bell provided a thought experiment, that in modern jargon is an instance of a non-local game, which proves that the kinds of correlations that arise from isolated quantum mechanical systems are intrinsically different from those that could be generated in any hidden variable model. 
Bell's work lay the foundation for the subsequent design, and execution, of experiments which verify the existence of quantum entanglement (and thus refute Einstein's approach; see also \cite{clauser1969proposed}). The 2022 Nobel prize in physics was awarded to Aspect, Clauser and Zeilinger, partly for performing non-local games as experiments and verifying that the winning statistics in the games exceeds what local hidden variable models allow.

A \emph{correlation} is a function $p\colon A\times A\times X\times X\to [0,1]$, where $X$ and $A$ are finite sets, such that for every pair $x,y\in X$ the function $p(\cdot,\cdot|x,y)$ is a probability distribution over $A\times A$ --- the quantity $p(a,b|x,y)$ should be thought of as the answer to ``what is the probability $a$ and $b$ are provided as answers given that $x$ and $y$ were asked as questions?''; this cryptic phrase will soon be clarified.  A correlation  $p$ is said to be \emph{deterministic} if there are two functions $f_1,f_2\colon X\to A$ such that $p(a,b|x,y)=1$ only if $f_1(x)=a$ and $ f_2(y)=b$.
The convex hull   $C_c$ of  deterministic correlations is the set of \emph{local} or  \emph{classical} correlations.\footnote{An example of a \emph{non-local} correlation is $p(a,b|x,y)=1$ only when $a=y$ and $b=x$. Intuitively, this is because this correlation implies \emph{signaling} from one system to the other.} 
On the other hand, the  \emph{quantum correlations}  $C_q$ are those that can be generated by performing quantum measurements on a (finite-dimensional) bipartite physical system, where the measurement on one part depends only on $x$ and generates $a$, while  the measurement on the other part depends only on $y$ and generates $b$. For a formal description, see Remark \ref{rem:Born_rule}. The set $C_q$ is convex, but interestingly it is not closed~\cite{slofstra2019set}.%\footnote{This is because some points in the closure can only be realized as a limit of finite-dimensional measurements on finite-dimensional quantum states.} 

A \emph{non-local} game $\cG$ is specified by a probability distribution $\mu$ on the set ${X}\times {X}$, and a \emph{decision function} $D:X\times X\times A\times A\to \{0,1\}$.\footnote{As can be seen in Section \ref{sec:intro_non-local}, our definition of non-local games is slightly different. This difference is mainly cosmetic, and is driven by our motivation to relate non-local games to subgroup tests.} 
The interpretation of this combinatorial object  as a ``game'' comes from thinking of it (similarly to subgroup tests) as an interactive proof, but this time with two provers, a la \cite{ben1988multi}. 
The common dramatization goes as follows: Two provers are spatially separated. The verifier samples a pair of ``questions'' $x,y\sim \mu$, and ``asks''  one prover $x$ and   the other $y$. The provers then apply some local procedure to choose their ``answers'' $a$ and $b$ respectively. Finally, the verifier accepts if $D(x,y,a,b)=1$ and rejects otherwise.  

Since the process by which the provers generate their answers is hidden to us, we can only ``observe'' (samples from) the distribution of answers given questions, namely, the underlying correlation. So, the correlation can be thought of as the ``strategy'' used by the provers (similar to the way IRSs are seen as strategies for the prover in a subgroup test).  Bell's separation of $C_c$ from $C_q$ amounts to devising a game that is ``easy'' for players that use quantum correlations, yet ``hard'' for players that use only classical ones. 

Now, there is more than one candidate mathematical model for entanglement in quantum mechanics. A slightly generalized model, suggested by Tsirelson \cite{tsirelson1993some},  gives birth to a set of correlations known as \emph{quantum commuting} correlations $C_{qc}$. This set is closed and contains the quantum correlations.  Tsirelson famously asked whether the closure of $C_{q}$ is equal to $C_{qc}$ \cite{tsirelson2006bell}. At this point, we hope the type of problem already resonates with the reader, as this is again a \emph{separation of convex sets} type of problem.\footnote{Connes' embedding problem (CEP) is also a separation of convex sets type of problem, specifically, separating all  characters of the free group from those that are limits of finite dimensional characters. This problem was shown~\cite{fritz2012tsirelson,junge2011connes,ozawa2013connes} to be equivalent to Tsirelson's problem, and also has close connections with invariant random subgroups --- as every IRS induces a character. We omit discussions on the relations between this work and CEP from the text, as they are not particularly helpful for understanding our approach.}

\subsubsection*{Harnessing undecidability}
It was known that if  the quantum correlations are dense in the quantum commuting correlations, then the complexity class of multiprover interactive proofs with entangled provers ($\MIP^*$) contains only decidable languages.
By proving that the Halting Problem, which is undecidable, is in $\MIP^*$, the authors of~\cite{MIPRE} were able to resolve Tsirelson's problem in the negative.

In this paper a similar path is followed: Rather than directly find a functional that separates the co-sofic and general IRSs of the free group, we describe a computational problem --- approximating the sofic value of a subgroup test --- whose undecidability implies a negative resolution of the Aldous--Lyons conjecture.  Our approach can be described in two steps: First, show that a positive solution to the Aldous--Lyons Conjecture implies that approximating the sofic value of a subgroup test is decidable. Second, prove that this computational task is undecidable.

Elaborating on the first,  given a subgroup test $\cT$, we prove that its optimal value against co-sofic IRSs --- which we call the \emph{sofic value} of $\cT$ --- can be approximated from below, and its optimal value against \textbf{any} IRS --- which we call the \emph{ergodic value} of $\cT$ --- can be approximated from above.  Thus, if the Aldous--Lyons Conjecture has a positive answer, then the sofic and ergodic value of a test always agree, and using the aforementioned approximation procedures they can be calculated to any  predetermined accuracy. 

Elaborating on the second, we provide a mechanism for translating certain non-local games (see Section \ref{sec:intro_non-local}), which we call \emph{tailored non-local games}, to subgroup tests. This translation is essentially value preserving. Namely, if the non-local game had a (certain kind) of perfect quantum strategy --- which we call a $Z$-aligned permutation strategy that commutes along edges --- then the analogous subgroup test has a perfect (co-sofic) strategy. In the other direction, if the subgroup test has an almost perfect (co-sofic) strategy, then the non-local game has an almost perfect strategy. Hence, a certain strengthening of the reduction in \cite{MIPRE}, which states that small additive constant approximations to the value of \emph{tailored} non-local game are undecidable (Theorem \ref{thm:tailored_MIP*=RE}), implies that the sofic value of a game \textbf{cannot} be approximated to arbitrary precision. Thus, the Aldous--Lyons Conjecture has a negative solution.
\\

The rest of the introduction provides a deeper dive into the  definitions, results and ideas   of this paper.

\subsection{Subgroup Tests} \label{subsec_subgroup_tests}
Let $S$ be a finite set, $\cF=\cF(S)$ be the free group with basis $S$, $\Sub(\cF)$ the collection of subgroups of $\cF$ and ${\rm Prob}(\frak{sub}(\mathcal{F}))$ the set of all Borel probability measures on $\frak{sub}(\mathcal{F})$. 
A \emph{challenge}  is a pair $(K;D)$ such that $K$ is a finite subset of $\cF$, and $D$ is a function from subsets of $K$ to $\{0,1\}$. A subset $A\subseteq \cF$ \emph{passes} the challenge $(K;D)$ if $D(A\cap K)=1$, and \emph{fails} it otherwise.
A \emph{test} $\cT$ is a finite collection of challenges $\{(K_i;D_i)\}_{i\in Q}$, where $Q$ is a finite index set, together with a probability distribution $\mu$ over the set $Q$. 
A \emph{strategy} is a (Borel) probability distribution over subgroups of the free group, namely an element $\pi\in \Prob(\Sub (\cF))$. 
One can \emph{run} the test $\cT$ against $\pi$, as follows. 
%\tnote{Trying inline version:}
First the verifier samples $i\in Q$ according to $\mu$. Next, the prover samples a subgroup  $H$ according to $\pi$. Finally, the verifier makes the decision to \emph{accept} if $H$ passes the $i^{\rm th}$ challenge $ (K_i;D_i)$, and \emph{reject} if $H$ fails the $i^{\rm th}$ challenge.
%\begin{itemize}
%\item Sample $i\in Q$ according to $\mu$.
%\item Sample a subgroup  $H$ according to $\pi$.
%\item \emph{Accept} if $H$ passes the $i^{\rm th}$ challenge $ (K_i;D_i)$, and %\emph{reject} if $H$ fails the $i^{\rm th}$ challenge.
%\end{itemize}

The \emph{value} of using the strategy $\pi$ against the test $\cT$ is its acceptance probability, namely
\[
\val(\cT,\pi)=\Ex_{H\sim \pi}\Ex_{i\sim \mu}[D_i(H\cap K_i)].
\]
For a fixed $\cT$, the value is a continuous linear functional from probability distributions over subgroups of $\cF$ to $\mathbb{R}$. 
Hence, it enables us to study convex subsets of $\Prob(\Sub(\cF))$ by their optimal value against a given test. 

\begin{rem}
    As mentioned earlier, a challenge is, essentially, a continuous map from the space of subgroups of $\cF$ to $\{0,1\}$. One can extract from any such continuous map  a challenge $(K;D)$ that describes it and vice versa. Therefore, a test is a convex combination of continuous maps from $\Sub(\cF)$ to $\{0,1\}$, and the value is integration of this map against a (probability) measure. 
\end{rem}

\subsection{Invariant random subgroups}\label{Sec_IRS}

As $\frak{sub}(\mathcal{F})$ inherits the product topology
from the power set $\{0,1\}^{\mathcal{F}}$,  it is compact.
Every action of a group $\Gamma$ on a set $X$ extends naturally to an action of $\Gamma$ on the power set $\{0,1\}^X$ by 
\[
\forall \gamma \in\Gamma,A\subseteq X\ \colon\ \ \gamma.A=\{\gamma.a \mid a\in A\}\;.
\]
Hence, the conjugation action of $\mathcal{F}$ on itself is inherited by $\{0,1\}^\cF$. Also, $\Sub(\cF)\subseteq \{0,1\}^\cF$ is preserved by this action. Hence, the conjugation action extends  to its power  set $\{0,1\}^{\Sub(\cF)}$ by 
\[
\forall w\in \mathcal{F}, A\subseteq \frak{sub}(\mathcal{F}) \ \colon \ \ w. A=\{wHw^{-1} \mid H\in A\}\;.
\]
 The conjugation action of $\mathcal{F}$ further extends to ${\rm Prob}(\frak{sub}(\mathcal{F}))$:
For every Borel set $B\subseteq \Sub(\cF)$ and $\mu\in{\rm Prob}(\frak{sub}(\mathcal{F}))$, we have
\[
(w.\mu)(B)=\mu(w^{-1}.B)\;.
\]

\begin{defn}\label{defn:IRS}
The set of \emph{invariant random subgroups} ${\rm IRS}(\mathcal{F})$
consists of all probability measures in ${\rm Prob}(\frak{sub}(\mathcal{F}))$
that are invariant under the action of $\mathcal{F}$. 
\end{defn}

The \emph{ergodic value} of a test $\cT$ is the best performance of a strategy $\pi\in \IRS(\cF)$ against it, namely
\begin{equation}\label{ergodic_val_equation}
    \val_{\erg}(\cT)=\sup\{\val(\cT,\pi) \mid \pi\in \IRS(\cF)\}\;.
\end{equation}
\begin{rem}
The extremal points in the convex set  $\IRS(\Gamma)$ are commonly referred to as \emph{ergodic IRSs}.
    As we take the supremum of a linear functional over a convex compact set, the optimum is obtained on an extremal point, which explains the phrase \emph{ergodic value}. 
\end{rem}

\subsection{Finitely described invariant random subgroups} \label{Intro_Sec_fin_desc_IRS}
Let $X$ be a finite set, and let $\Sym(X)$ be the symmetric group acting on $X$. Every map $\sigma\colon S\to \Sym(X)$ extends uniquely to an action of $\cF=\mathcal{F}(S)$ on $X$. 
For a vertex $x\in X$, we can define its \emph{stabilizer} to be \begin{equation} \label{stabilizer}
{\rm Stab}(\sigma,x)=\left\{ w\in\mathcal{F}\mid \sigma(w).x=x\right\}\;,
\end{equation}
which is a subgroup of $\mathcal{F}$.
Hence, we can associate with $\sigma\colon S\to \Sym(X)$ an IRS of $\cF$ via the following sampling procedure: 
$(1)$ Choose $x\in X$ uniformly at random. $(2)$
  Output ${\rm Stab}(\sigma,x)$.
We denote by $\Phi$ the map from finite actions of $\cF$ to $\IRS (\cF)$ defined by this procedure, namely
\begin{equation}\label{correspondence_actions_IRSs}
\Phi(\sigma)=\Ex_{x\in X}\big[{\bf 1}_{{\rm Stab}(\sigma,x)}\big]\;,
\end{equation}
where ${\bf 1}_H$ is the Dirac measure concentrated on the subgroup $H$ and $\Ex_{x\in X}$ is the expectation according to the uniform measure over $X$. The set of \emph{finitely described invariant random subgroups} $\IRS_{\fd}(\cF)$ is the image of the map $\Phi$ in $\IRS(\cF)$. 
\begin{defn}
The \emph{sofic value} of a test $\cT$ is the (asymptotic) best performance of a strategy $\pi\in \IRS_{\fd}(\cF)$ against it, namely
\begin{equation}\label{sofic_val_equation}
    \val_{\sof}(\cT)=\sup\,\big\{\val(\cT,\Phi(\sigma)) \mid |X|<\infty,\ \sigma\colon S\to \Sym(X)\big\}\;.
\end{equation}
\end{defn}

\begin{rem}
The name sofic value is appropriate, as demonstrated in Section \ref{subsec_separation_test}. Furthermore, we often abuse notation and use $\val(\cT,\sigma)$ instead of $\val(\cT,\Phi(\sigma))$, even though the action itself is not an IRS.
\end{rem}

Since $\IRS_{\fd}(\cF)\subseteq \IRS(\cF)$, we have $\val_{\sof}(\cT)\leq \val_{\erg}(\cT)$ for every test $\cT$. The Aldous--Lyons conjecture \cite{Aldous_Lyons_Conj} \textbf{can} be formulated as follows:
\begin{conj} [Aldous--Lyons, cf.\ Section 6 of \cite{Gelander_ICM2018}] \label{(Aldous-Lyons)} Is  ${\rm IRS}_{\fd}(\mathcal{F})$
 weak* dense in ${\rm IRS}(\mathcal{F})$?
\end{conj}
\begin{rem}
In their paper \cite{Aldous_Lyons_Conj}, Aldous and Lyons ask this as a question. In the passing of time, the positive form of this problem, namely that $\IRS_{\fd}(\cF)$ \textbf{is} dense in $\IRS(\cF)$, received the name \emph{The Aldous--Lyons Conjecture} (cf. \cite{Aldous--Lyons_conj_blogpost}). Thus, we use  the more common ``Conjecture'' phrasing instead of Problem. 

The algebraic formulation we  provided earlier, which can also be found  in Section 6 of \cite{Gelander_ICM2018}, may seem different from the above. The question was  whether every invariant random subgroup of $\cF$ is \emph{co-sofic}, where co-sofic means an IRS which is the weak* limit of convex combinations of uniform distributions over finite index subgroups. Since the finitely described IRSs are dense in the co-sofic ones, the formulation in Conjecture \ref{(Aldous-Lyons)} is equivalent to this standard (algebraic) formulation.

%As previously discussed, in their original paper \cite{Aldous_Lyons_Conj}, Aldous and Lyons asked whether every unimodular network is sofic --- see Section \ref{sec:Aldous-Lyons_perspective} for their formulation. Their formulation  is equivalent to the above algebraic formulation, for example using \cite[Section 3]{abert2014kesten}.

Lastly, the term co-sofic is closely related to the better known term of a \emph{sofic group} (cf.\ Definition \ref{defn:sofic}).
A finitely presented group $\mathcal{F}/N$ is sofic if and only if the Dirac measure concentrated on $N$ is a co-sofic IRS (cf.\ Proposition 6.1 in \cite{Gelander_ICM2018}).
      Thus, a positive solution to the Aldous--Lyons Conjecture \ref{(Aldous-Lyons)} implies in particular that every group is sofic.\footnote{The way we formulate it, a positive solution to the Aldous--Lyons Conjecture only implies that all finitely presented groups are sofic. It is standard to deduce it for all groups from the finitely presented ones \cite{MR2089244}.  } 

\begin{comment}
      The soficity property has been useful in obtaining positive results about group rings \cite{MR2089244}, invariant couplings of random fields \cite{MR3503036}, and invariants of groups actions \cite{MR2552252, MR2854085, MR3132735, MR3993930}. See \cite{Cap_Lup_Sofic_Hyperlinear_book, MR3821628} for more background on sofic groups.
\end{comment}
    
\end{rem}

\begin{cor} \label{cor:values_are_the_same}
Since $\val(\cT,\cdot)$ is a continuous functional, a positive solution to the Aldous--Lyons Conjecture \ref{(Aldous-Lyons)} implies that  for every test $\cT$,  $\val_{\sof}(\cT)=\val_{\erg}(\cT)$.
\end{cor}

\begin{rem}
    Due to the relation between subgroups of the free group and Schreier graphs, one can view subgroup tests as a new property testing model for (edge-labeled) graphs:
Let $X$ be a finite set and $\sigma\colon S\to X$  a transitive action (the non-transitive case is a convex combination of transitive ones). For $x\in X$, the stabilizer ${\rm Stab}(\sigma,x)$ is exactly the set of labeled paths in the Schreier graph of $\cF(S)$ with respect to ${\rm Stab}(\sigma,x)$ and $S$ that begin and end at $x$, namely, closed paths originating from $x$. Thus, when the   finitely described IRS $\Phi(\sigma)$ is put against a test $\cT$, we actually test the aforementioned Schreier graph  in the following way:
First, choose a vertex $x\in X$ uniformly at random and sample $i\sim \mu$.
  List the paths from $K_i$ and for each one check whether it is closed or open.  According to this check,  decide using $D_i$ whether to accept or reject.
\end{rem}

\subsection{Decidability of approximating the sofic value} \label{section:Interactive_proofs_with_a_finitely_described_IRS}
\begin{comment}
    We briefly recall basic notions from computability and complexity theory.
Let $\Sigma$ be a finite set that will play the role of an alphabet.
Let $\Sigma^{*}$ be the free monoid generated by $\Sigma$, namely
the set of all finite words in $\Sigma$. 
A \emph{language} is a subset $L\subseteq\Sigma^{*}.$ Every language
$L$ defines a \emph{decision problem}: given $w\in\Sigma^{*}$, decide
whether $w\in L$.
A \emph{promise language} is a pair $L_{yes},L_{no}\subseteq\Sigma^{*}$
such that $L_{yes}\cap L_{no}=\emptyset.$ A promise language $\left(L_{yes},L_{no}\right)$
defines a \emph{promise decision problem}: given $w\in L_{yes}\cup L_{no}$,
decide whether $w\in L_{yes}$. A promise language $(L_{yes},L_{no})$ is \emph{decidable} if there exists a Turing Machine (TM for short) $M$ satisfying the following: If $w\in L_{yes}\cup L_{no}$ is given to $M$ as input, then $M$ halts and outputs a single bit $M(w)$ satisfying
\[
\begin{split}
    M(w)=1 &\iff w\in L_{yes},\\
    M(w)=0 &\iff w\in L_{no}.
\end{split}
\]
Let $L^1$ and $L^2$ be two (promise) languages.  We say that $L^1$ \emph{reduces} to $L^2$ if there exists a time Turing machine $M$ for which: If $\sigma\in L^1_{yes}$ (respectively,  $\sigma\in L^1_{no}$) is taken as the input of $M$, then $M$ halts and the output $M(\sigma)$ is in $L^2_{yes}$ (respectively, $M$ halts and the output  $M(\sigma)$ in $L^2_{no}$).
Thus, if $L^2$ is decidable then $L^1$ is also decidable.

\end{comment}

If the  challenge distribution $\mu$ in a subgroup test $\cT$ is rational, then $\cT$ can be encoded as a finite bit string. There are many (non-equivalent) ways of doing so.  
Assume we fixed \textbf{some} encoding for tests with rational challenge distributions such that, given the encoding, all underlying combinatorial data of the test can be calculated in finite time. 
Namely, the  set of generators $S$, the indexing set of challenges $Q$, the collection of challenges $\{(K_i;D_i)\}_{i\in Q}$ and the distribution $\mu$ can all be \textbf{read} from the encoding in finite time.
 
\begin{thm}[Main Theorem I]\label{Main_Thm}
For every fixed encoding of tests such that all combinatorial data of the test can be read from it in finite time:
\begin{enumerate}
    \item There is a Turing machine $M_1$ that takes as an input (the encoding of) a test $\cT$, and outputs an infinite sequence of non-decreasing numbers $(\alpha_t)_{t=1}^\infty \subseteq [0,1]$  such that $\lim_{t\to \infty}\alpha_t=\val_{\sof}(\cT)$.
    \item There is a Turing machine $M_2$ that takes as an input (the encoding of) a test $\cT$, and outputs an infinite sequence of non-increasing numbers $(\beta_t)_{t=1}^\infty \subseteq [0,1]$  such that $\lim_{t\to \infty}\beta_t=\val_{\erg}(\cT)$.
\end{enumerate}
\end{thm}
\begin{rem}
    Clause $(2)$ of Theorem \ref{Main_Thm} should be compared to the Navascu\'es--Pironio--Ac\'in (NPA) hierarchy in the study of non-local games \cite{navascues2008convergent}. See also the introduction of \cite{MIPRE}.
\end{rem}

\begin{cor}\label{cor:Aldous_Lyons_implies_SOFVAL_decidable}
    If the Aldous--Lyons Conjecture \ref{(Aldous-Lyons)} has a positive solution, then for every $\theta>0$, there is a Turing machine which accepts a subgroup test $\cT$ as input and outputs its sofic value $\val_{\sof}(\cT)$ up to an additive error of at most $\theta$.
\end{cor}

%\subsection{Non-local games}

\newcommand{\mX}{\mathcal{X}}
\newcommand{\mY}{\mathcal{Y}}
\newcommand{\mA}{\mathcal{A}}
\newcommand{\mB}{\mathcal{B}}

\subsection{Tailored non-local games and their associated subgroup tests}\label{sec:intro_non-local}

%Our second main result, stated as Theorem~\ref{thm:main2} below, establishes a correspondence between subgroup tests and a restricted class of non-local games, which we call \emph{tailored}. By leveraging a long line of works on the study of non-local games as interactive proof systems in quantum complexity theory, this connection will eventually allow us (in the companion paper) to disprove the conclusion of Corollary~\ref{cor:Aldous_Lyons_implies_SOFVAL_decidable} and hence the Aldous--Lyons Conjecture.
%The aforementioned correspondence relates finitely described strategies for the subgroup test to a special kind of strategy for the nonlocal game. In general, a quantum strategy is specified by a collection of \emph{observables}, which are Hermitian operators (on a finite-dimensional Hilbert space) that square to identity, satisfying some additional conditions. This is because, in quantum mechanics, a measurement is represented by an observable. A special type of observable is an involutive permutation. As described earlier a family of permutations acting on a finite set implicilty specifies an IRS. Thus one can start to glimpse a connection between the two frameworks. 

As described before, the standard definition of  a (non-local) game consists of two finite sets, a questions set $X$ and an answers set $A$, together with a probability distribution $\mu$ over $X\times X$ and a decision function $D\colon X\times X\times A\times A\to \{0,1\}$. In this paper, we use a slightly modified definition, that makes the connection with subgroup tests clearer. Before diving into our alternative definition, note first that the distribution $\mu$ defines a graph structure on $X$, by associating edges with its support. Furthermore, as the exact choice of $A$ is irrelevant, we can assume it is the set of all bit strings up to some length $\Lambda$.

A (synchronous, non-local) game consists of a graph  $G=(V,E)$, a distribution $\mu$ over $E$,
a length function $\ell\colon V\to \mathbb{N}$,  formal sets of generators $S_x=\{\sX^{x,i}\mid 1\leq i\leq \ell(x)\}$ --- the union of which is denoted by $S=\bigcup_{x\in V}S_x$ --- and decision functions $D_{xy}\colon \{0,1\}^{S_{xy}}\to \{0,1\}$ for every edge $xy\in E$, where $S_{xy}=S_x\cup S_y$.\footnote{So, the main difference in our definition, is that the decision function $D$ rejects automatically answers that are not of the appropriate length according to $\ell$.}
     A (synchronous, quantum) \emph{strategy} for $\cG$ is a map $\rho\colon \bigcup_{x\in V}S_x\to U(n)$, where the images are involutions, and $\rho({S_x})$ commute for every fixed $x\in V$.\footnote{This is again not the standard definition of a quantum strategy. Unpacking our definition, $\rho$ associates an observable (Hermitian operator on a finite-dimensional Hilbert space that squares to the identity) with each bit of the answer when asked for $x\in V$. Since the answer must be totally measured, all the observables associated with a specific vertex must commute. The fact there is a single  map $\rho$, and not one for each prover, is the \emph{synchronicity} of this strategy. }
Such a strategy defines for every $xy\in E$ a probability distribution over maps $\gamma\colon S_{xy}\to \{0,1\}$ in the following way: 
For every $\sX\in S$, let  $\sP^{\sX}_0$ be the projection on the $(+1)$-eigenspace of $\rho(\sX)$, and $\sP^\sX_{1}$  the projection on its $(-1)$-eigenspace. 
Furthermore, let $\tau$ be the dimension normalized trace on $n\times n$ matrices.  
Then, the probability that $\gamma \colon S_{xy}\to \{0,1\}$ is sampled according to $\rho$ is by definition 
\begin{equation}\label{eq:intro_sampling_according_rho}
    \tau\left(\prod_{\sX\in S_{x}}\sP^{\sX}_{\gamma(\sX)}\prod_{\sY\in S_y}\sP^{\sY}_{\gamma(\sY)}\right)\;.
\end{equation}
A map $\gamma$ sampled that way is said to be sampled \emph{according to} $\rho$, and we denote it by $\gamma\sim \rho$.\footnote{Note that this sampling depends on the chosen edge $xy\in E$. Namely, though implicit in the notation, $\gamma$ is defined from $S_{xy}$ for some edge $xy\in E$. It is important to note that $\gamma$ is \textbf{not} globally defined on all of $S$.}
    Relating to standard notations, we remark that by letting $a=\gamma|_{S_x}$ and $b=\gamma|_{S_y}$, the correlation $p(a,b|x,y)$ induced by the quantum strategy  $\rho$ (as in Remark \ref{rem:Born_rule}) agrees with   formula \eqref{eq:intro_sampling_according_rho}.
%\tnote{Earlier I set up the ``standard'' notation for nonlocal games, to make it easier for those familiar with it to parse the somewhat non-standard description in this paragraph. I tried to make the connection, but I am still a bit unhappy about this $\gamma\sim\rho$ part; I wasn't sure how to best say that this is the same as the $p(a,b|x,y)$ described above, i.e.\ make the connection with correlation sets. (Possibly, we want to remove the discussion of correlation sets altogether)}
    In a similar manner to subgroup tests, one can run the game $\cG$ against the strategy $\rho$: 
    %\tnote{again, inline version}
    First, the verifier samples an edge $xy\in E$ according to $\mu$. Then, the prover samples $\gamma\colon S_{xy}\to \{0,1\}$ according to $\rho$. Finally, the verifier {accepts} if $D_{xy}(\gamma)=1$, and {rejects} otherwise.
   % \begin{itemize}
   % \item Sample an edge $xy\in E$ according to $\mu$.
   %     \item Sample $\gamma\colon S_{xy}\to \{0,1\}$ according to $\rho$.
   %     \item {Accept} if $D_{xy}(\gamma)=1$, and {reject} otherwise.
   % \end{itemize}
 
  As for subgroup tests, the  \emph{value} of the strategy $\rho$ against the non-local game $\cG$ is its acceptance probability when ran against the game. Namely,
\[
\begin{split}
\val(\cG,\rho)=\Ex_{xy\sim \mu}\Ex_{\gamma\sim \rho}[D_{xy}(\gamma)]\;.
\end{split}
\]
 The (synchronous) \emph{quantum value} of $\cG$, $\val^*(\cG)$, is the supremum of $\val(\cG,\rho)$ over all possible quantum strategies against $\cG$.

The main result in \cite{MIPRE} (recalled in detail in Theorem \ref{thm:MIP*=RE} herein) is a reduction from the Halting Problem to approximating  the quantum value of a game.
Our main contribution in the second half of this paper (Sections \ref{sec:Synch},\ref{sec:test_associated_with_game} and \ref{sec:proof_of_value_preserving_transformation}) is a mapping from a specific sub-class of non-local games --- which we term \emph{tailored} non-local games
%\footnote{The class of tailored games is a generalizing  the well known class of \emph{linear constraint system games} (see Example \ref{example:LCSs}).}  
and describe further below --- to subgroup tests, that  (almost) preserves their values.

\begin{thm}[Main Theorem II, informal. For the formal version see Theorem \ref{thm:values_of_associated_tests}]\label{thm:main2}
    There is a transformation that takes as input a \textbf{tailored} non-local game (see Definition \ref{defn:tailored_games}) and outputs a subgroup test (which we call the associated synchronous subgroup test, see Definition  \ref{defn:associated_test}), such that:
    \begin{enumerate}
        \item If the game had a perfect (i.e., a value $1$) \textbf{$Z$-aligned, commuting along edges, permutation strategy} (see Definitions \ref{defn:permutation_strategy} and \ref{defn:Z-aligned_strategy}), then the associated subgroup test has a perfect strategy. 
        \item  Every almost perfect finitely described strategy for the associated subgroup test can be transformed into an almost perfect quantum strategy to the original tailored non-local game. 
    \end{enumerate}
\end{thm}

%As emphasized, our mapping has these nice properties only when the input is a \textbf{tailored} non-local game, and its perfect strategy, when it exists, is a  \textbf{$Z$-aligned permutation} strategy. 
%\tnote{I rewrote the below a bit}
Informally, a tailored game is one where the decision function $D_{xy}$ takes the following form. First, $D_{xy}$ examines the value taken by $\gamma$ on a \textbf{subset} of the generators $S_x\cup S_y$ (we refer to these as \emph{readable} variables). Based on this, $D_{xy}$ determines a collection of linear equations on the entire string $\gamma$. Finally, $D_{xy}$ accepts if and only if all these equations are satisfied. This form for the game can be seen as a generalization of binary linear constraint system games (Example \ref{example:LCSs}), which corresponds to the case where the set of readable variables is empty. 

A \emph{permutation strategy}, as its name suggests, is a quantum strategy which is induced by permutations. Such a strategy is said \emph{$Z$-aligned} when the observables associated with the readable variables (i.e.\ the bits of the answer $\gamma$ that control the parity checks in the tailored game) are all diagonal matrices (with respect to some natural basis). The restriction to $Z$-aligned strategies is crucial for item (1) in Theorem~\ref{thm:values_of_associated_tests} above to hold. 

In a companion paper \cite{Tailored_MIPRE},  we prove that the reduction of the Halting Problem to approximating the quantum value of a game presented in \cite{MIPRE} can be modified such that the resulting non-local game is tailored and satisfies the required stronger assumptions (Theorem \ref{thm:tailored_MIP*=RE}): If the input Turing machine halts, then the game has a perfect $Z$-aligned permutation strategy that commutes along edges, and if the input Turing machine does not halt, then the value of the game is bounded from above by $\nicefrac{1}{2}$. Therefore, together with our mapping from non-local games to subgroup tests, we deduce that approximating the sofic value of a game is undecidable. In turn, by Corollary \ref{cor:Aldous_Lyons_implies_SOFVAL_decidable}, 
  the Aldous--Lyons Conjecture \ref{(Aldous-Lyons)} has a negative solution.

\begin{rem}\label{rem:explicit_refutation}
    In \cite[Section 12.3]{MIPRE}, the authors were able to  specify a non-local game whose quantum value is bounded above by $\nicefrac{1}{2}$ but has a perfect quantum commuting strategy. Using the results of \cite{kim2018synchronous}, which we elaborate on further in the next remark,  this implies the existence of a finitely presented $*$-algebra which can be endowed with a tracial state and cannot be embedded into an ultrapower of the hyperfinite $II_1$ factor. This is a somewhat constructive counter example, but the state itself is only proven to exist and does not have a constructive form (at this point). By the various formulations of Connes' embedding problem, this proves that there is a non-hyperlinear character of the free group, which is the analogue of a non co-sofic IRS of the free group, but it is provided in a non-constructive form.

By applying the ``same'' construction as in \cite[Section 12.3]{MIPRE} with the algorithm constructed in Theorem \ref{Main_Thm} replacing the NPA hierarchy algorithm,  one can pin down an explicit subgroup test with perfect ergodic strategy but all sofic strategies bounded above by some constant (which can be approximated by following the proofs). This is the most constructive we can get in this case, namely, the specific non co-sofic IRS which is the perfect strategy for this subgroup test is proven only to exist, and is not constructed. Thus, we decided not to specify this test. 
    
    It would be very interesting to find an explicit non co-sofic IRS, namely separate between $\IRS(\cF)$ and $\overline{\IRS_{\fd}(\cF)}$ in a constructive manner.
\end{rem}

\begin{rem} \label{rem:slofstra} We comment on a natural attempt to extend the construction of~\cite{MIPRE} so as to directly deduce the existence of a non-sofic group, a result which would imply a negative answer to the Aldous--Lyons Conjecture. 

The idea (which is folklore, and explicitly mentioned in~\cite{paddock2023satisfiability}) is as follows. It is well-known that to every synchronous game $\mathcal{G}$ one can associate an algebra $\mathscr{A}(\mathcal{G})$ such that there is a tight connection between the $C^*$-representation theory of this algebra and perfect  values for the game in various correlation sets~\cite{paulsen2016estimating,kim2018synchronous,helton2019algebras}. In the special case where $\mathcal{G}$ is a linear constraint system (LCS) game, there is a group  $\Gamma$ such that $\mathscr{A}(\cG)$  is (isomorphic to the completion of) the complex group algebra $\mathbb{C}[\Gamma]$~\cite{kim2018synchronous,goldberg2021synchronous}.\footnote{Actually, it is isomorphic to the completion of  $\nicefrac{\mathbb{C}[\Gamma]}{\langle \sJ+1\rangle}$, where $\sJ$ is a special generator in $\Gamma$ and $\langle\sJ+1\rangle$ is the two-sided ideal generated by $\sJ+1$ in $\mathbb{C}[\Gamma]$.} 

Now, if the reduction given in \cite{MIPRE} returned only LCS games, then one would obtain the following result: there exists an LCS game with a perfect quantum commuting strategy, but no (asymptotically) perfect quantum strategy. By the connection mentioned above this implies that the group $\Gamma$ associated with $\mathscr{A}(\cG)$  is non-hyperlinear. As every sofic group is hyperlinear, one would thus obtain an example of a non-sofic group.

Since the reduction in \cite{MIPRE} \textbf{does not} return LCS games, a workaround could be to find an embedding from the algebra $\mathscr{A}(\mathcal{G})$, where $\mathcal{G}$ is returned by~\cite{MIPRE}, into $\mathscr{A}(\mathcal{G}')$ where $\mathcal{G}'$ is an LCS game. Obstacles to this ``black-box'' approach are discussed in~\cite{paddock2023satisfiability}. 

One can view our approach as finding a ``sweet spot'' that is in some sense half-way between synchronous games, which were used in~\cite{MIPRE}, and LCS games. This is because tailored games are a natural generalization of LCS games. In the companion paper we show that~\cite{MIPRE} can be adapted to return tailored games, and in the present paper we show that this suffices to resolve the Aldous--Lyons Conjecture. The existence of non-sofic groups (and non-hyperlinear groups) remains an  elusive open problem. 
\end{rem}

\subsection{Structure of the paper}
 In Section \ref{sec_gp_theory_prelims}, we prove  structural results about invariant random subgroups, some of which are standard and well known to experts. In Section \ref{sec:Proving_the_main_thm}, we  prove our first main theorem, 
 Theorem \ref{Main_Thm},  and draw corollaries from it. 
 In Section \ref{sec:Subgroup_Tests}, we provide  several examples of subgroup tests,  inspired  by  both group theory and complexity theory.  In Section \ref{sec:Robustness}, we discuss a rigidity property of tests that generalizes group stability and which is analogous to robustness of non-local games. 
In Section \ref{sec:Synch}, we define tailored non-local games and analyze their properties. 
In Section \ref{sec:test_associated_with_game},  we describe the (almost) value preserving correspondence between tailored non-local games and synchronous subgroup tests, informally outlined in Theorem \ref{thm:main2}. This allows us in Corollary \ref{cor:AL_is_false} to provide a negative solution to the Aldous--Lyons Conjecture \ref{(Aldous-Lyons)}. The proof of the aforementioned value preserving correspondence, Theorem \ref{thm:values_of_associated_tests}, is proved in Section \ref{sec:proof_of_value_preserving_transformation}. Theorem \ref{thm:tailored_MIP*=RE}, which is the other main ingredient in our refutation, is proved in  a companion paper \cite{Tailored_MIPRE}. There is no open problems section, but  further research directions can be found in Remarks  \ref{rem:explicit_refutation}, \ref{rem:slofstra} and \ref{rem:better_params}.

For a reader who seeks the main ideas of the paper without diving too deep into the proofs, we suggest reading, in addition to the introduction, only Sections \ref{sec:Proving_the_main_thm} and \ref{sec:test_associated_with_game}.

\subsection*{Acknowledgements}
% \mnote{After some thought, it is more appropriate to thank Mikael in the companion paper.}
%\textcolor{red}{Lewis Bowen....}
We thank Alon Dogon for reading an early draft of this paper and providing many useful improvements to the presentation. We would also like to thank Mikael de la Salle for various discussions along the way. Finally, we want to thank Peter Sarnak for organizing a seminar on ``strong convergence'' in Princeton and asking the second author to give a talk, as it organized his thoughts  and led to the current presentation of ideas.

Lewis Bowen is supported by NSF grant DMS-2154680.
Michael Chapman acknowledges with gratitude the Simons Society of Fellows and is supported by a grant from the Simons Foundation (N. 965535).
Alex Lubotzky is supported by the European Research Council (ERC)
under the European Union's Horizon 2020 (N. 882751), and by a research grant from the Center for New Scientists at the Weizmann Institute of Science.
Thomas Vidick is supported by a research grant from the Center for New Scientists at the
Weizmann Institute of Science, a Simons Investigator award, AFOSR Grant No. FA9550-22-1-0391, and
ERC Consolidator Grant VerNisQDevS (101086733).
 
\section{Group theoretic preliminaries}\label{sec_gp_theory_prelims}

\subsection{Topology (and other properties) of invariant random subgroups}\footnote{The content of this section is standard. For example, it appears in Section 3 of \cite{abert2014kesten}.}
We first go through some properties of the space $\{0,1\}^\cF$ (induced with the product topology). For every  $K,L\subseteq\mathcal{F}$, let 
\[
\mathcal{C}(K,L)=\{A\subseteq\mathcal{F} \mid K\subseteq A,\ L\cap A=\emptyset\}\subseteq \{0,1\}^\cF.
\] 
Then,
\begin{enumerate}
    \item When $K$ and $L$ are finite, the set $\cC(K,L)$ is open and closed. Moreover, the collection $$\{\cC(K,L)\mid K,L\subseteq \cF,\ |K|,|L|<\infty\}$$ is a basis for the topology of $\{0,1\}^\cF$. In addition, $\{0,1\}^\cF$ is compact.
    \item For every $w\in \mathcal{F}$ and subsets $K,L\subseteq \mathcal{F}$ we have
    \[
    w.\mathcal{C}(K,L)=\mathcal{C}(wKw^{-1},wLw^{-1}).
    \]
\end{enumerate}
The following is a standard fact. A proof is added because elements of it are later used in Lemma \ref{Lem:upper_approx}.
\begin{claim} \label{close_compact_Sub(F)}
    The set of subgroups $\Sub(\cF)$ is closed in the collection of all subsets $\{0,1\}^\cF$. 
\end{claim}

\begin{proof}
    A subset $A\subseteq \cF$ is a subgroup if it satisfies the following three conditions: 
  \[
  \begin{split}
      (1)\  \Id_\cF\in A\ ,\quad
       (2)\  w\in A\implies w^{-1}\in A\ ,\quad
       (3)\ w,w'\in A\implies w\cdot w'\in A.
  \end{split}
\]
For every $w,w'\in \cF$, define the following open sets:  The subsets not containing $\Id_\cF$, $C_{\Id}=\cC(\emptyset,\{ \Id_\cF\})$. The subsets containing $w\in\cF$ but not $w^{-1}$, $C_{w}=\cC(\{w\},\{w^{-1}\})$. The subsets containing $w,w'\in \cF$ but not $w\cdot w'$,
    $C_{w,w'}=\cC(\{w,w'\},\{w\cdot w'\})$.
     Then,
     \[
        \Big(\bigcup_{w,w'\in\cF}  C_{w,w'}\Big)\cup \Big(\bigcup_{w\in\cF} C_w \Big) \cup C_\Id=\{0,1\}^\cF \setminus \Sub(\cF),
     \] 
     which proves that $\Sub(\cF)$ is a closed set.
\end{proof}

Recall the definition of $\IRS(\cF)$, the invariant random subgroups of the free group $\cF$ (Definition \ref{defn:IRS}).

\begin{fact} \label{invariance}
Let $\mu,\nu\in {\rm Prob}(\{0,1\}^\mathcal{F})$. Then,
\begin{enumerate}
    
    \item  $\mu=\nu$ if and only if $\mu(\mathcal{C}(K,L))=\nu(\mathcal{C}(K,L))$ for every finite $K,L\subseteq \mathcal{F}$. 
    \item  $\mu$ is conjugation invariant if and only if for every $s\in S$ we have $s.\mu=\mu$. Hence, $\mu\in \IRS(\cF)$ if and only if for all $s\in S$ and finite $K,L\subseteq \mathcal{F}$ we have $\mu(\mathcal{C}(K,L))=\mu(s.\mathcal{C}(K,L))$.
\end{enumerate}
\end{fact}

A sequence of measures $\mu_n\in  {\rm Prob}(\{0,1\}^\mathcal{F})$ is said to weak* converge to $\mu$ if for all continuous $f\colon  \{0,1\}^\mathcal{F} \to \mathbb{R}$, 
\[
\Ex_{H\sim \mu_n}[f(H)] \xrightarrow{n\to\infty}\Ex_{H\sim \mu}[f(H)].
\]
We denote weak convergence by $\mu_n \xrightarrow{w^*} \mu$. 
\begin{rem}\label{R:weak}
    In our context, there is a more ``down to earth'' way of viewing weak* convergence. For every finite subset $A\subseteq \cF$, and every  probability measure $\mu\in \Prob(\{0,1\}^\cF)$, we can define a probability measure $\mu\cap A\in \Prob(\{0,1\}^A)\subseteq \RR^{2^A}$ by 
\[
\forall B\subseteq A\ \colon \ \ \mu\cap A(B)=\mu(\{H\subseteq \cF\mid H\cap A=B\}).
\]
Then,
\[
\mu_n\xrightarrow{w^*}\mu \iff \forall A\subseteq\cF, |A|<\infty\ \colon \ \ \mu_n\cap A\xrightarrow{n\to \infty}\mu\cap A,
\]
where the convergence on the right is the standard one in the Euclidean space $\RR^{2^A}$.
\end{rem}
\begin{claim} \label{Prob_is_cpct}
The space $\Prob(\{0,1\}^\mathcal{F})$ is compact in the weak* topology.
\end{claim}

\begin{proof}
Since $\{0,1\}^\mathcal{F}$ is compact, the fact can be deduced from the Banach--Alaoglu Theorem (cf.\ Theorem 3.15 in \cite{Rudin_Functional_Analysisi_BOOK}).
\end{proof}

\subsection{Pseudo subgroups}
For every $B\subseteq C\subseteq \cF$, the restriction map $R_{B\subseteq C}\colon \{0,1\}^C\to \{0,1\}^B$ is \[
\forall A\subseteq C\  \colon\ \ R_{B\subseteq C}(A)=A\cap B.\] When $C=\cF$ we remove it from the notation and write $R_B$ instead of $R_{B\subseteq \cF}$.
For sets $K,L\subseteq B\subseteq \cF$ we can define
\[
\cC_B(K,L)=\{A\subseteq B\mid K\subseteq A,\ A\cap L=\emptyset\}.
\]
Note that $\cC_\cF(K,L)=\cC(K,L)$.
\begin{claim} \label{Restriction_of_C(K,L)}
    Let $K,L\subseteq B\subseteq C\subseteq \cF$. Then $R_{B\subseteq C}(\cC_C(K,L))=\cC_B(K,L)$ and $R_{B\subseteq C}^{-1}(\cC_B(K,L))=\cC_C(K,L).$
\end{claim}
\begin{proof}
    Let $A\subseteq B$ and $A'\subseteq C$ such that $A=A'\cap B$. As $K,L\subseteq B$, the intersection of $A'$ with $K$ and $L$ depends only on $A'\cap B=A$. Hence  
    \begin{equation}\label{eq_CB_CC}
         A\in \cC_B(K,L) \iff A'\in \cC_C(K,L).
    \end{equation}
  Adding this observation to the fact $R_{B\subseteq C}$ is onto proves the desired claim. 
\end{proof}

Now, we can pushforward measures along the inverse system of restrictions. Namely, given  $B\subseteq C\subseteq \cF$ and a probability measure $\pi$ on $\{0,1\}^C$ we can define the probability measure $R_{B\subseteq C*}\pi$ on $\{0,1\}^B$ by
\[
\forall Y \subseteq_{\rm Borel} \{0,1\}^B\ \colon \ \ R_{B\subseteq C*}\pi(Y)=\pi(R_{B\subseteq C}^{-1}(Y)).
\]

\begin{cor} \label{cor:contin_of_R_B}
   Claim \ref{Restriction_of_C(K,L)} implies that $R_{B\subseteq C}$ is continuous. Therefore,  $R_{B\subseteq C*}\colon \Prob(\{0,1\}^C)\to \Prob(\{0,1\}^B)$ is continuous.
\end{cor}

\begin{comment}

\begin{proof}
\textcolor{red}{This proof is very standard. As oppose to the other basic claims we attach proofs to, this one doesn't seem to contribute to the understanding of the main theorems. So maybe it should be removed. Waiting for your thoughts on it.}
 Let $f\colon \{0,1\}^B\to \RR$ be a continuous function. Define $\widetilde f\colon \{0,1\}^C\to \RR$ be the pullback of $f$ along $R_{B\subseteq C}$, namely 
    \[
    \forall A\subseteq C\ \colon \ \ \widetilde f(A)=f\circ R_{B\subseteq C}(A)=f(A\cap B).
    \]
    The pullback is a composition of continuous functions and hence continuous. For every $\mu\in \Prob(\{0,1\}^C)$,
    \[
\begin{split}
    \Ex_{H\sim \mu}[\widetilde f(H)]=\Ex_{H\sim \mu}[f\circ R_{B\subseteq C}(H)]=\Ex_{H\sim R_{B\subseteq C*}\mu}[f(H)].
\end{split}
    \]
       So, if  $\mu_n \xrightarrow{w*}\mu$  in $\Prob(\{0,1\}^C)$, then 
       \[
\Ex_{H\sim R_{B\subseteq C*}\mu_n}[f(H)]=\Ex_{H\sim \mu_n}[\widetilde f(H)]\xrightarrow{n\to \infty} \Ex_{H\sim \mu}[\widetilde f(H)]=\Ex_{H\sim R_{B\subseteq C*}\mu}[f(H)].
       \]
       Since this was true for any continuous $f$, we have that $R_{B\subseteq C*}\mu_n\xrightarrow{w*}R_{B\subseteq C*}\mu$ and we are done.
\end{proof}
\end{comment}

\begin{defn}
   Let $B\subseteq \cF$. A subset  $A\subseteq B$ is said to be a \emph{pseudo subgroup} of $B$ if there is a subgroup $H\leq \cF$ such that $R_B(H)=H\cap B=A$. Denote by $P\Sub(B)$ the collection of pseudo subgroups of $B$.\footnote{Note that this collection is a closed subset of the subsets of $B$.}
\end{defn}
\begin{claim} \label{A_in_Im(R_B)}
    Let $A\subseteq B\subseteq \cF$. Then, $A$ is a pseudo subgroup of $B$ if and only if $\langle A\rangle \cap B = A$, where $\langle A\rangle$ is the subgroup of $\cF$ generated by $A$. 

\end{claim}

\begin{proof}
    If there is a subgroup $H$ for which  $A=R_B(H)=H\cap B\subseteq H$, then in particular $\langle A\rangle\subseteq H$. Because $A$ is always contained in $\langle A\rangle \cap B$, we can conclude.
\end{proof}

\begin{cor}\label{cor:restriction_of_pseudo_subgroups}
    Let $A\subseteq B\subseteq C\subseteq \cF$. If $A$ is \textbf{not} a pseudo subgroup of $B$, and $A'\subseteq C$ restricts to $A$ (i.e., $A'\cap B=A$), then $A'$ is \textbf{not} a pseudo subgroup of $C$.
\end{cor}
\begin{proof}
    By Claim \ref{A_in_Im(R_B)}, if $A$ is not a pseudo subgroup of $B$, then there is a $w\in B$ such that $w\notin A$ while $w\in \langle A\rangle$.  Since $w\in B\subseteq C$, and $A'\cap B=A$, the same $w$ satisfies that $w\notin A'$ while $w\in \langle A\rangle\subseteq \langle A'\rangle$, proving that $A'$ is not a pseudo subgroup of $C$.
\end{proof}

\begin{defn}\label{def:P_B_Q_B}
Let $B\subseteq \cF$ be a subset.

\begin{enumerate}
\item The set $\mathcal{P}_B$ is the collection of all $\pi \in \Prob(\{0,1\}^{B})$  that are \emph{concentrated on} pseudo subgroups of $B$, namely $\pi(P\Sub(B))=1$. Equivalently, $\pi(\{0,1\}^B\setminus P\Sub(B))=0$.
\item The set $\mathcal{Q}_B$ is the subset of $\cP_B$ of distributions that are \emph{$S$-invariant when defined}. Namely, assuming  $K, L \subseteq B$ are finite subsets such that $$\forall s\in S\cup S^{-1}\ \colon\ \  sKs^{-1}, sLs^{-1} \subseteq B,$$ we have $\pi( \cC_B(K,L)) = \pi(\cC_B( sKs^{-1}, sLs^{-1}))$.  We call  probability distributions in $Q_B$ \emph{pseudo invariant random subgroups} over $B$, or pseudo IRSs in short.
\item Let $\tilde\cP_B=R_{B*}^{-1}(\cP_B)$ and $\tilde\cQ_B=R_{B*}^{-1}(\cQ_B)$.
\end{enumerate}
\end{defn}

\begin{rem}
    Note that $\cP_\cF=\tilde\cP_\cF=\Prob(\Sub(\cF))$ and $\cQ_\cF=\tilde\cQ_\cF=\IRS(\cF)$.
\end{rem}

\begin{lem}\label{Lem:upper_approx}
Let $B\subseteq C\subseteq \cF$ be subsets.
    \begin{enumerate}
       
        \item The restriction to $B$ of a distribution over pseudo subgroups of $C$, is a distribution over pseudo subgroups of $B$. Namely,
        \[
        \forall \pi\in \cP_{C}\ \colon \ \ R_{B\subseteq C*}\pi\in \cP_B.
        \] 
        Similarly, the restriction to $B$ of a pseudo IRS of $C$ is a pseudo IRS of $B$. Namely, 
        \[
        \forall \pi\in \cQ_{C}\ \colon \ \ R_{B\subseteq C*}\pi\in \cQ_B.
        \]
        \item The (Borel) probability distributions over subgroups of $\cF$ are the intersection of pull-backs of all distributions over pseudo subgroups in finite subsets. Similarly, the IRSs of $\cF$ are the intersection of pull-backs of all pseudo IRSs of finite subsets. Namely,
        \[
        \bigcap_{B\subseteq \cF, |B|<\infty} \tilde\cP_B=\Prob(\Sub(\cF)),\ \bigcap_{B\subseteq \cF, |B|<\infty} \tilde\cQ_B=\IRS(\cF).
        \]
    \end{enumerate}
\end{lem}

\begin{proof}

      Let $\pi\in \cP_C$.  By definition, $$\pi(\{0,1\}^C\setminus P\Sub(C))=0.$$ By Corollary \ref{cor:restriction_of_pseudo_subgroups}, $$R_{B\subseteq C}^{-1}(\{0,1\}^B\setminus P\Sub(B))\subseteq \{0,1\}^C\setminus P\Sub(C).$$ Hence, by the definition of the pushforward, 
      \[
      \begin{split}
R_{B\subseteq C*}\pi(\{0,1\}^B\setminus P\Sub(B))&=\pi(R_{B\subseteq C}^{-1}(\{0,1\}^B\setminus P\Sub(B)))\\
&\leq \pi(\{0,1\}^C\setminus P\Sub(C))\\
&=0,
 \end{split}
      \]
    and  so $R_{B\subseteq C*}\pi\in \cP_B$.  

        Assume further that $\pi\in \cQ_C$. By Claim \ref{Restriction_of_C(K,L)}, for every finite $K,L\subseteq B$, $$R^{-1}_{B\subseteq C}(\cC_B(K,L))=\cC_C(K,L).$$ 
        Thus  
        \[
        \begin{split}
R_{B\subseteq C*}\pi(\cC_B(K,L))&=\pi(R_{B\subseteq C}^{-1}(\cC_B(K,L)))
=\pi(\cC_C(K,L)).
\end{split}
        \]
        If $K,L\subseteq B$ such that $sKs^{-1},sLs^{-1}\subseteq B$ for every $s\in S\cup S^{-1}$, then the same applies to $C$ since it contains $B$, and hence
        \[
        \begin{split}
R_{B\subseteq C*}\pi(\cC_B(K,L))&=\pi(\cC_C(K,L))\\
&=\pi(\cC_C(sKs^{-1},sLs^{-1}))\\
&=R_{B\subseteq C*}\pi(\cC_B(sKs^{-1},sLs^{-1})).
\end{split}
        \]
        Therefore $R_{B\subseteq C*}\pi$ is in $\cQ_B$ and $(1)$ is proved.
 \\
 
  By clause $(1)$, for every $B\subseteq C\subseteq \cF$ we have $\tilde\cP_C\subseteq \tilde \cP_B$ and $\tilde \cQ_C\subseteq \tilde\cQ_B$. In particular, when $C=\cF$ we have $\Prob(\Sub(\cF))\subseteq \tilde \cP_B$ and $\IRS(\cF)\subseteq \tilde\cQ_B$. Hence  \[
        \Prob(\Sub(\cF))\subseteq  \bigcap_{B\subseteq \cF, |B|<\infty} \tilde\cP_B,\quad \IRS(\cF)\subseteq \bigcap_{B\subseteq \cF, |B|<\infty} \tilde\cQ_B.
        \]
     We are left to prove the reverse containments. Recall the open cover of the complement of the subgroups described in the proof of Claim \ref{close_compact_Sub(F)}.
     Given $\pi\in \Prob(\{0,1\}^\cF)$ such that $\pi\notin \Prob(\Sub(\cF))$, we have $$\pi\left(\left(\bigcup_{w,w'\in\cF}  C_{w,w'}\right)\cup \left(\bigcup_{w\in\cF} C_w \right) \cup  C_\Id\right)=\pi (\{0,1\}^\cF \setminus \Sub(\cF))>0.$$
    Hence, by the union bound (Boole's inequality), there is either a pair $w,w'\in\cF$ such that $\pi(C_{w,w'})>0$, or a single element $w\in \cF$ such that $\pi(C_w)>0$, or $\pi(C_\Id)>0$. Thus, if $B=\{ \Id_\cF,w,w',w\cdot w',w^{-1}\}$,  then $\pi\notin \tilde\cP_B$. In particular $\pi\notin \bigcap \tilde\cP_B$, proving the first reverse containment.
     
     Lastly, let $\pi \in \bigcap \tilde\cQ_B$.  For every finite $K,L\subseteq \cF$, there is a larger enough finite set $B$ such that $K,L\subseteq B$ and also $sKs^{-1},sLs^{-1}\subseteq B$ for every $s\in S\cup S^{-1}$. Since $\pi\in \tilde\cQ_B$,  by Claim \ref{Restriction_of_C(K,L)} applied to $C=\cF$, we have for every $s\in S\cup S^{-1}$ that
    \[
    \begin{split}
    \pi(\cC(K,L))&=\pi(R_B^{-1}(\cC_B(K,L)))\\
    &=R_{B*}\pi(\cC_B(K,L))\\
    &=R_{B*}\pi(\cC_B(sKs^{-1},sLs^{-1}))\\
    &=\pi(R_B^{-1}(\cC_B(sKs^{-1},sLs^{-1})))\\
    &=\pi(\cC(sKs^{-1},sLs^{-1})).
    \end{split}
    \]
    Hence, by Fact \ref{invariance}, $(2)$ is deduced.
\end{proof}

\begin{rem}
The system of finite subsets of $\cF$  forms a directed system induced by inclusion. The pushforwards along  restriction maps $R_{B\subseteq C*}$ thus form an inverse system of maps between the compact sets $\Prob(B)$ for $B\subseteq \cF, |B|<\infty$. In this language, Lemma \ref{Lem:upper_approx} says that 
\[
\Prob(\Sub(\cF))=\varprojlim_B\cP_B\qquad \textrm{and}\qquad \IRS(\cF)=\varprojlim_B \cQ_B.
\]
\end{rem}

\subsection{Computational aspects of pseudo invariant random subgroups}

\begin{claim} \label{Im(R_n)_is_decidable}
Given $A\subseteq B\subseteq \cF$ where $|B|<\infty$, there is an algorithm deciding whether $A$ is a pseudo subgroup of $B$. 
\end{claim}

\begin{proof}

Given a finite set of elements $A=\{ w_1,...,w_k \}\subseteq \cF$, and another element $w\in \cF$, there is an algorithm to decide whether $w\in \langle A\rangle$. Indeed, one can use the following:
\begin{enumerate}
\item Construct ${\rm Core}(A)$ the core graph of $\langle A\rangle$  via the method of Stallings' foldings  \cite{Stallings_Foldings_of_G_trees,KAPOVICH2002608}.
\item Check whether the path that begins at the basepoint of ${\rm Core}(A)$ and traverses $w$ is closed.
\item If the path is closed, then $w\in \langle A\rangle$. Otherwise $w\notin \langle A\rangle$.
\end{enumerate}

By Claim \ref{A_in_Im(R_B)}, $A$ is not a pseudo subgroup of $B$ if and only if there exists a $w\in B$ such that $w\notin A$ but $w\in \langle A \rangle$. Hence, a sufficient and necessary condition for $A$ to be a pseudo subgroup of $B$ is to be accepted by the following algorithm:
\begin{enumerate}
    \item Construct  ${\rm Core}(A)$.
    \item For every $w\in B$ satisfying $w\notin A$, \emph{reject} if $w$ is a closed path in ${\rm Core}(A)$.
    \item   \emph{Accept} if $A$ passed $(2)$ for every $w\in B$.
\end{enumerate}
\end{proof}

\begin{rem}
    Recall that a free group is LERF (locally extended residually finite), i.e., for every finite set $A=\{w_1,...,w_k\}$ and every $w\in \cF$ such that $w\notin \langle A\rangle$, there exists a finite index subgroup $H$ for which 
\[
A\subseteq H\  \textrm{and}\ w\notin H.
\]
This was originally proved by M. Hall \cite{hall_1949}, and can be proven using the method of Stallings' foldings \cite{Stallings_Foldings_of_G_trees,KAPOVICH2002608}. It follows that for  every $A$ a pseudo subgroup of a finite $B$, there exists a \emph{finite index} subgroup $H$ for which $R_B(H)=A$.
\end{rem}

\begin{lem}\label{lem:computable_polytopes}
    Let $B\subseteq \cF$ be a finite set. Then, the set $\cP_B$ of distributions over pseudo subgroups of $B$ and the set $\cQ_B$ of pseudo IRSs over $B$ are computable polytopes in $\mathbb{R}^{\{0,1\}^B}$. Namely, they are defined by (computable) integral linear equations and inequalities.
\end{lem}
\begin{proof}
         Recall that $\cQ_B\subseteq \cP_B\subseteq \Prob(\{0,1\}^B)\subseteq \RR^{\{0,1\}^B}$. Namely, all these distributions  are vectors which are indexed by subsets of $B$.     
     We describe the defining system of equations and inequalities forcing a vector $\pi\in \mathbb{R}^{\{0,1\}^B}$ to belong to $\cP_B$. First, it needs to be a probability measure, hence  
        \[
        \begin{split}
             \forall A\subseteq B\ \colon \ \ \pi(A)&\geq 0,\\
             \sum_{A\subseteq B} \pi(A)&=1.
        \end{split}
        \]
        Further, for every $A\subseteq B$, we can run the algorithm in Claim \ref{Im(R_n)_is_decidable} to decide whether $A$ is a pseudo subgroup, and if not to add the constraint $\pi(A)=0$. This is the defining system of $\cP_B$. To get the defining system of $\cQ_B$, we need to add the following:
        For every $K,L\subseteq B$, we can check whether for all $s\in S\cup S^{-1}$, the sets $sKs^{-1}$ and $sLs^{-1}$ are also contained in $B$, and if so add the constraints
        \[
            \forall s\in S\cup S^{-1}\ \colon \ \ \underbrace{\sum_{A\in \cC_B(K,L)}\pi(A)}_{\pi(\cC_B(K,L))}=\underbrace{\sum_{A\in \cC_B(sKs^{-1},sLs^{-1})}\pi(A)}_{\pi(\cC_B(sKs^{-1},sLs^{-1}))}.
        \]
        
\end{proof}

\section{Main Theorem I: Approximating the values of subgroup tests} \label{sec:Proving_the_main_thm}

Recall the definitions of challenges, tests and values from Section \ref{subsec_subgroup_tests}.
Let $\cT$ be a test over the generating set $S$, with challenges $(K_i;D_i)_{i\in Q}$ and \textbf{rational} distribution $\mu$ over $Q$. Let $K=\bigcup_{i\in Q} K_{i}$, which is a finite subset of $\cF(S)=\cF$. We can extend the definition of $\val(\cT,\cdot)$ in the following way: For every  $B$ satisfying $K\subseteq B\subseteq \cF$, and for every distribution $\pi \in \Prob(\{0,1\}^B)$,  let 
\[
\val(\cT,\pi)=\Ex_{A\sim \pi}\Ex_{i\sim \mu}[D_i(A\cap K_{i})].
\]
The next claim is formalizing the following intuitive idea: Since $D_i$ depends only on the restriction of $H$ to $K_{i}$, then replacing $\pi$ by its push-forward along any restriction $R_{B}$ will not change its value against $\cT$.

\begin{claim}\label{claim:value_under_restriction}
Under the extended definition of the value, for every $B$ and $C$ such that $K\subseteq B\subseteq C\subseteq \cF$, and for every $\pi\in \Prob(\{0,1\}^C)$, we have
\[
\val(\cT,\pi)=\val(\cT,R_{B\subseteq C*}\pi).
\]
\end{claim}

\begin{proof}
    By definition, sampling $A\sim R_{B\subseteq C*}\pi$ is the same as sampling $A'\sim \pi$ and outputting $A=A'\cap B$.
    Also, $K_i=B\cap K_{i}$ since $K_{i}\subseteq K\subseteq B$. Hence,
    \[
    \begin{split}
        \val(\cT,R_{B\subseteq C*}\pi)&=\Ex_{A\sim R_{B\subseteq C*}\pi}\Ex_{i\sim \mu}[D_i(A\cap K_{i})]\\
        &=\Ex_{A'\sim \pi}\Ex_{i\sim \mu}[D_i((A'\cap B)\cap K_{i})]\\
        &=\Ex_{A'\sim \pi}\Ex_{i\sim \mu}[D_i(A'\cap K_{i})]\\
        &=\val(\cT,\pi).
    \end{split}
    \]
\end{proof}

Recall the first clause of Theorem \ref{Main_Thm}: There is a Turing machine $M_1$ that takes as an input the encoding of a test $\cT$, and outputs an infinite sequence of non-decreasing numbers $(\alpha_t)_{t=1}^\infty \subseteq [0,1]$  such that $\lim_{t\to \infty}\alpha_t=\val_{\sof}(\cT)$.

\begin{proof}[Proof of Theorem \ref{Main_Thm} $(1)$]

Let $\sigma\colon S\to \Sym(X)$ be a finite action. The value of $\Phi(\sigma)$ against $\cT$ is 
 \[
 \begin{split}
 \val(\cT,\Phi(\sigma))&=\Ex_{H\sim \Phi(\sigma)}\Ex_{i\sim \mu}D_i(H\cap K_i)\\
 &=\sum_{x\in X}\sum_{i\in Q}\frac{\mu(i)}{|X|}\cdot D_i(\Stab(\sigma,x)\cap K_i)
 \end{split}
 \]
 which is a rational combination of  evaluations of $D_i$. Since the sets $K_i$, functions $D_i$ and coefficients $\mu(i)$ can be calculated in finite time from  the encoding of $\cT$, and $\Stab(\sigma,x)\cap K_i$ can be calculated from $\sigma$ for every $x\in X$, the value of $\Phi(\sigma)$ against $\cT$ can be computed in finite time. 

  There is a computable countable enumeration of all (equivalence classes of) finite $\cF(S)$-actions. To see that, note that a map $\sigma\colon S\to \Sym(n)$ is defined by the tuple $(n\ ;\ \{\sigma(s)\}_{s\in S})$. Since the elements of $\Sym(n)$ can be ordered (e.g., lexicographically), as well as the elements of $S$, the infinite set of tuples $(n\ ;\ \{\sigma(s)\}_{s\in S})$ inherits a countable enumeration. Denote by $\sigma_t$ the $t^{\rm th}$ action according to this enumeration.
 For $t\in \mathbb{N}$, let 
 \[
\alpha_t=\max_{s\leq t} \{\val(\cT,\Phi(\sigma_s))\}.
 \]
 By our discussion so far, the number $\alpha_t$ can be computed in finite time. Moreover, $(\alpha_t)_{t=1}^\infty$ is a non-decreasing sequence by definition. 
 Since  ${\rm Im}(\Phi)=\IRS_{\fd}(\cF)$, the sequence $(\alpha_t)_{t=1}^\infty$ tends to the supremum of $\val(\cT,\cdot)$ over $\IRS_{\fd}(\cF)$, which is, by \eqref{sofic_val_equation}, $\val_{\sof}(\cT)$.
\end{proof}

We now recall the second clause of Theorem \ref{Main_Thm}:  There is a Turing machine $M_2$ that takes as an input the encoding of a test $\cT$, and outputs an infinite sequence of non-increasing numbers $(\beta_t)_{t=1}^\infty \subseteq [0,1]$  such that $\lim_{t\to \infty}\beta_t=\val_{\erg}(\cT)$.

\begin{proof}[Proof of Theorem \ref{Main_Thm} $(2)$]
%Let $K \subset \F(S)$ be a finite set such that $D$ is $K$-local. Let $\{K_i\}_{i=1}^\infty$ and $\{L_i\}_{i=1}^\infty$ be enumerations of the collection of finite subsets of $\F(S)$.  

 Let $K=\bigcup K_{i}$ as before, and let $B\subseteq \cF$ be a finite subset containing $K$.  Recall the definition of $\cQ_B$ the pseudo invariant random subgroups over $B$ (Definition \ref{def:P_B_Q_B}), and its pre-image $\tilde \cQ_B$. Recall also  Claim \ref{claim:value_under_restriction} and the discussion before it, where we extended the notion of the value to probability distributions over  subsets of $\cF$ that contain $K$. The $B$-approximation to the ergodic value of $\cT$ is defined to be
$$\val_B(\cT) = \sup \{\val(\cT,\pi)\mid  \pi \in \cQ_B\}=\sup \{\val(\cT,\pi)\mid  \pi \in \tilde\cQ_B\}.$$

Let $\{B_t\}_{t=1}^\infty$ be a sequence of finite subsets of $\cF$, all containing $K$, such that $B_t\subseteq B_{t+1}$ and $\bigcup_{t=1}^\infty B_t=\cF$. 
We now claim the following:
\begin{enumerate}
    \item $\val_{B_{t}}(\cT)$ is computable.
    \item The sequence $\val_{B_t}(\cT)$ is non-increasing.
    \item $\lim_{t\to \infty}\val_{B_t}(\cT)=\val_{\erg}(\cT)$.
\end{enumerate}

By Lemma \ref{lem:computable_polytopes}, $\cQ_B$ is defined by a computable integral system of equations and inequalities. Thus, maximizing a linear functional $\val(\cT,\cdot)$ over it can be done using linear programming in finite time, proving $(1)$.

By Lemma \ref{Lem:upper_approx}, whenever $B\subseteq C\subseteq \cF$ we have $R_{B\subseteq C*}(\cQ_{C})\subseteq \cQ_{B}$. Hence $\tilde \cQ_{B_{t+1}}\subseteq \tilde \cQ_{B_t}$, which implies $\val_{B_{t+1}}(\cT)\leq \val_{B_t}(\cT)$, proving $(2)$.

Moreover,  $\IRS(\cF)=\tilde \cQ_{\cF}\subseteq \tilde \cQ_{B_t}$  which implies $\val_{B_t}(\cT)\geq \val_{\erg}(\cT)$.
On the other hand, by Claim \ref{Prob_is_cpct} the space $\Prob(\{0,1\}^\cF)$ is compact, and since $\tilde \cQ_B$ are closed sets they are also compact. Hence there is a distribution $\pi_t\in \tilde\cQ_{B_t}$ satisfying $\val_{B_t}(\cT)=\val(\cT,\pi_t)$. By the compactness of $\Prob(\{0,1\}^\cF)$, there is a weak limit to some subsequence of $(\pi_t)_{t=1}^\infty$, which we denote by $\pi_\infty$. Without loss of generality the subsequence is the original sequence. Since $\pi_s\in\tilde \cQ_{B_t}$ for every $s\geq t$, and since   $\tilde \cQ_B$ are closed, we deduce that $\pi_\infty\in \bigcap \tilde \cQ_{B_t}=\IRS(\cF)$. By the definition of a weak limit, since $\val(\cT,\cdot)$ is a continuous functional, we have 
\[
\lim_{t\to \infty}\val_{B_t}(\cT)=\lim_{t\to \infty}\val(\cT,\pi_t)=\val(\cT,\lim_{w^*}\pi_t)=\val(\cT,\pi_\infty)\leq \val_{\erg}(\cT),
\]
proving $(3)$. Therefore, choosing $\beta_t=\val_{B_t}(\cT)$ finishes the proof.
\end{proof}

\section{Examples of tests} \label{sec:Subgroup_Tests}

\subsection{Relations verification test}  Given a finite subset $R\subseteq \cF(S)$, the $R$-\emph{verification test} $\mathcal{V}_R$ checks whether $R$ is contained in the sampled subgroup $H$. It contains a single challenge $(R;D)$ with $D={\bf 1}_R$, namely
\[
\forall A\subseteq R\ \colon \ \ D(A)=\begin{cases}
    1 & A=R,\\
    0 & A\neq R.
\end{cases}
\]

The sofic value of $\cV_R$ is always $1$, e.g., using the Dirac measure concentrated on the group $\cF$ itself (which associates via $\Phi$ with the trivial action of $\cF$ on a singleton). Moreover, if the conjugation orbit of a finite index subgroup $H$ passes $\cV_R$ with probability $1$, then $R\subseteq \bigcap_{w\in \cF} wHw^{-1}$.

\begin{rem}
Though the value of the $R$-verification test does not give us much information about $R$, we will see  in Section \ref{sec:Robustness} that the potential robustness of this test is an important property of $R$.
Moreover, this test has variations where you sample $r\in R$ according to some distribution and check only that $r\in H$ and not that all of $R$ is in $H$. See for example \cites{CL_part1,CVY_efficient} for more on that.
\end{rem}

\subsection{Separation test}  \label{subsec_separation_test}
Given two finite subsets $R,L\subseteq \cF(S)$, the $(R,L)$-\emph{separation test} $\mathcal{V}_{R||L}$ checks whether $R$ is contained in the sampled subgroup $H$ but $L\setminus R$ is not. It contains a single challenge $(R\cup L;D)$ with $D={\bf 1}_R$, namely
\[
\forall A\subseteq R\cup L\ \colon \ \ D(A)=\begin{cases}
    1 & A=R,\\
    0 & A\neq R.
\end{cases}
\]

\begin{defn} \label{def:res_fin}
A group $\Gamma$ is \emph{residually finite} if for every finite subset $L\subseteq \Gamma$ where $\Id_\Gamma\notin L$, there is a finite index normal subgroup $N\trianglelefteq \Gamma$ such that $L\cap N=\emptyset$.
\end{defn}

\begin{claim}
    A finitely presented group $\Gamma=\langle S|R\rangle$ is residually finite if and only if for every finite $L\subseteq \cF$ such that $L\cap \langle \langle R\rangle \rangle=\emptyset$\footnote{The subgroup $\langle\langle R\rangle\rangle$ is the smallest normal subgroup containing $R$.} there is a finitely described IRS passing $\cV_{R||L}$ with probability $1$. 
\end{claim}

\begin{proof}
    Let $\Gamma=\nicefrac{\cF(S)}{\langle\langle R\rangle\rangle}$ be residually finite, and $L\subseteq \cF$ a finite set disjoint from $\langle\langle R\rangle\rangle$. By Definition \ref{def:res_fin}, there is a finite index normal subgroup $N\trianglelefteq \Gamma$ such that $L\langle\langle R\rangle\rangle\cap N=\emptyset$. Using the sequence of epimorphisms
    \[
    \cF\twoheadrightarrow \Gamma \twoheadrightarrow \Gamma/N
    \]
    we can define the left action of $\cF$ on $X=\Gamma/N$ which passes $\cV_{R||L}$ with probability $1$. 
    
    On the other hand, given a finite subset $\Id_\Gamma \notin L\subseteq \Gamma$, let  $L'\subseteq \cF$ be any set of representatives for $L$. Since $\Id_\Gamma \notin L$, we have $L'\cap \langle\langle R\rangle \rangle=\emptyset$. 
    Hence, by our assumption,  there is an  action $\rho\colon S\to \Sym(X)$ with $\val(\cV_{R||L'},\Phi(\rho))=1$. Without loss of generality, $\rho$ is transitive (otherwise, its restriction to any orbit will do).  
    The fact $\val(\cV_{R||L'},\Phi(\rho))=1$ implies in particular that $R\subseteq \Stab(\rho,x)$ for any $x\in X$. Thus, $\rho$ factors through $\Gamma$ and defines an action $\rho'\colon \Gamma\to \Sym(X)$ by $\rho'(w\langle\langle R\rangle \rangle)=\rho(w)$. By the fact $\val(\cV_{R||L'},\Phi(\rho))=1$, the permutation $\rho(w)$ has no fixed points for every $w\in L'$, and thus $\rho'(u)$ has no fixed points for every $u\in L$. Hence $N=\bigcap_{x\in X}\Stab(\rho',x)$ is a normal subgroup of $\Gamma$ satisfying $N\cap L=\emptyset$, as required to deduce $\Gamma$ is residually finite.
\end{proof}

    Let $X$ be a finite set. Given two permutations $\sigma,\tau\in \Sym(X)$, we define their \emph{normalized Hamming distance} to be
\begin{equation}\label{regular_Hamming_dist}
d_{H}(\sigma,\tau):=\frac{\left|\left\{ x\in X\mid\sigma(x)\neq\tau(x)\right\} \right|}{|X|}=\mathbb{P}_{x\in X}\left[\sigma(x)\neq\tau(x)\right].
\end{equation}

\begin{defn}\label{defn:sofic}
    A finitely presented group $\Gamma=\langle S|R\rangle$ is \emph{sofic} if for every $L\subseteq \cF(S)$ such that $L\cap \langle\langle R\rangle\rangle=\emptyset$, and for every $\eps>0$, there is a finite set $X$ and an action $\rho\colon S\to \Sym(X)$ such that 
    \[
    \begin{split}
\forall r&\in R\ \colon \ \ d_H(\rho(r),\Id_X)\leq \eps,\\
\forall w&\in L\ \colon \ \ d_H(\rho(w),\Id_X)\geq 1-\eps.
\end{split}
    \]
\end{defn}

\begin{rem}
It is straightforward to see that soficity is a relaxation of residual finiteness. 
\end{rem}

\begin{prop}
    A finitely presented group $\Gamma=\langle S|R\rangle$ is sofic if and only if for every finite $L\subseteq \cF$ such that $L\cap \langle \langle R\rangle \rangle=\emptyset$ we have $\val_{\sof}(\cV_{R||L})=1$.
\end{prop}

\begin{proof}
Assume $\Gamma$ is sofic, and let  $L\subseteq\cF$ be a finite set such that $L\cap \langle \langle R\rangle \rangle=\emptyset$. For every $\eps>0$, there is an action $\rho_\eps\colon S\to \Sym(X_\eps)$, such that 
  \[
    \begin{split}
\forall r&\in R\ \colon \ \ d_H(\rho_\eps(r),\Id_{X_\eps})\leq \eps,\\
\forall w&\in L\ \colon \ \ d_H(\rho_\eps(w),\Id_{X_\eps})\geq 1-\eps.
\end{split}
    \]
    Hence,
    \[
    \begin{split}
        \val(\cV_{R||L},\Phi(\rho_\eps))&=\Pro_{x\in X_\eps}[\forall r\in R,\forall w\in L\ \colon\  \rho_\eps(r).x=x,\ \rho_\eps(w).x\neq x]\\
        &\geq 1-\sum_{r\in R}\Pro_{x\in X_\eps}[\rho_\eps(r).x\neq x]-\sum_{w\in L}\Pro_{x\in X_\eps}[\rho_\eps(w).x=x]\\
        &\geq 1-\sum_{r\in R}d_H(\rho_\eps(r),\Id_{X_\eps})-\sum_{w\in L}\left(1-d_H(\rho_\eps(w),\Id_{X_\eps})\right)\\
        &\geq 1-|R|\eps-|L|\eps.
    \end{split}
    \]
    Since $|R|$ and $|L|$ are fixed, and $\eps>0$ is arbitrarily small, we have $\val_{\sof}(\cV_{R||L})=1$.

    On the other hand, if for every $L$ we have $\val_{\sof}(\cV_{R||L})=1$, then for every $\eps>0$ there is an action $\rho\colon S\to \Sym(X)$ with $\val(\cV_{R||L},\Phi(\rho))\geq 1-\eps$. Now, since 
    \[
      \begin{split}
\val(\cV_{R||L},\Phi(\rho))&=\Pro_{x\in X}[\forall r\in R,\forall w\in L\ \colon\  \rho(r)x=x,\ \rho(w)x\neq x]\\
&\leq \min(\min_{r\in R}(\Pro_{x\in X}[\rho(r)x=x]),\  \min_{w\in L}(\Pro_{x\in X}[\rho(w)x\neq x]))\\
&=\min(\min_{r\in R}(1-d_H(\rho(r),\Id_X)),\  \min_{w\in L}(d_H(\rho(w),\Id_X))),
\end{split}
    \]
    we deduce that
 \[
  \begin{split}
\forall r&\in R\ \colon \ \ d_H(\rho(r),\Id_X)\leq \eps,\\
\forall w&\in L\ \colon \ \ d_H(\rho(w),\Id_X)\geq 1-\eps.
\end{split}
    \]
    and $\Gamma$ is sofic.
\end{proof}
\subsection{Formula satisfaction test}
Let $\varphi(y_1,...,y_n)$ be a boolean formula. The test $\cT_\varphi$ is defined over the formal set of generators $S=\{x_1,...,x_n\}$. It has a single challenge $(K;D)$, where $K=S$. Note that the characteristic function of every subset $A\subseteq S$ is already in the form ${\bf 1}_A\colon \{x_1,...,x_n\}\to \{0,1\}$, and thus induces an assignment for the formal variables. The decision predicate $D$ accepts $A\subseteq S$ if ${\bf 1}_{A}$ induces a satisfying assignment for $\varphi$, i.e.,
\[
D(A)=\varphi({\bf 1}_{A}(x_1),...,{\bf 1}_{A}(x_n)).
\]
It is straightforward that $\val(\cT_\varphi)=1$ if and only if $\varphi$ is satisfiable. 

There is a more interesting version of this test, when the formula $\varphi$ is a 3CNF. Recall that a 3CNF is an and of clauses of the form $\varphi_t=y_i^{\eps_1} \lor y_j^{\eps_2}\lor y_k^{\eps_3}$, where $y_i,y_j,y_k$ are taken from a fixed set of formal generators (as before), $\eps_1,\eps_2,\eps_3\in\{0,1\}$, and we interpret $y^{0}=y$ and $y^{1}=\lnot y$. Now, checking that $\varphi$ is satisfied by ${\bf 1}_A$ is the same as going over all the clauses and checking that they are all satisfied. Thus, we can define a randomized version of this check by associating a challenge $(K_t;D_t)$ with each clause $\varphi_t$, by letting $K_t=\{x_i,x_j,x_k\}$ and $D_t(A)=\varphi_t({\bf 1}_A(x_i),{\bf 1}_A(x_j),{\bf 1}_A(x_k))$. Then, by choosing a uniform distribution over challenges,  the associated test $\cT_\varphi$ checks whether a uniformly chosen clause is satisfied by ${\bf 1}_A$. The famous PCP theorem \cite{PCP_thm} says that any 3CNF $\varphi$ can be (efficiently) transformed into a not much larger 3CNF $\varphi'$, such that if $\val_{\sof}(\cT_\varphi)=1$, then $\val_{\sof}(\cT_{\varphi'})=1$, and if $\val_{\sof}(\cT_\varphi)=0$, then  $\val_{\sof}(\cT_{\varphi'})\leq \nicefrac{1}{2}.$
The scaled up version of this argument shows that any $\NEXP$ problem (see \cite{BFL91}) can be reduced to approximating the sofic value of a test.\footnote{This last remark is meaningless as long as we do not specify the exact way we encode tests. The point  is that \textbf{there is} a natural way to encode tests such that $\NEXP$ can be reduced to approximating the sofic value immediately from \cite{BFL91}. In any case,  this paper is devoted to proving a much stronger conclusion, which is that approximating the sofic value is as \textbf{hard as the Halting Problem}.}

\begin{comment}
    
\end{comment}

\section{Robustness} \label{sec:Robustness}
 In this section we define  a rigidity property of tests called \emph{robustness} (see Definition \ref{defn:robustness}). Loosely, a test is robust if every almost optimal strategy against it is close to an optimal one. Though this notion is natural, one needs to clarify what is the measure of distance between strategies. 

\subsection{Edit distance}
We first define a generalized version of the normalized Hamming distance (\ref{regular_Hamming_dist}). Let $X$ and $Y$ be two finite sets, and assume $X\subseteq Y$. Given two permutations $\sigma\in \Sym(X)$ and $\tau\in \Sym(Y)$, we define their \emph{normalized Hamming distance with errors} to be
\begin{equation}\label{generalized_Hamming_dist}
   d_{H}(\sigma,\tau)=1-\frac{\left|\left\{ x\in X\mid\sigma(x)=\tau(x)\right\} \right|}{|Y|}. 
\end{equation}
If for every $y\notin X$ we let $\sigma(y)=\frak{error}$ $\notin Y,$
then we can define equivalently 
\[
d_{H}(\sigma,\tau)=\mathbb{P}_{y\in Y}\left[\sigma(y)\neq\tau(y)\right],
\]
which is the way we defined the normalized Hamming distance in (\ref{regular_Hamming_dist}). Since the generalized version is the same as the usual Hamming distance when $Y=X$, we use the same notation for both and just call them \emph{the Hamming distance} from now on.

Let $\frS \colon S\to \mathbb{R}_{>0}$ be a function, which we call the \emph{significance} function. 
Given two actions $\rho\colon S\to \Sym(X)$ and $\varphi\colon S\to \Sym(Y)$, we  define their $\frS$-\emph{weighted distance} as follows:
\[
d^
\frS(\rho,\varphi)=\sum_{s\in S}\frS(s)\cdot d_H(\rho(s),\varphi(s)).\footnote{Though the Hamming distance is always bounded by $1$, the weighted distance may be larger than $1$.}
\]
For every $\theta\in\Sym(Y)$, let $\varphi^\theta \colon S\to \Sym(Y)$ be $\varphi^\theta(s)=\theta \varphi(s) \theta^{-1}$.
Now, the $\frS$-\emph{weighted edit distance}\footnote{
There is a more graph theoretical way of viewing the edit distance, in the spirit of the  section \ref{sec:graph_theoretical_interpret}. It essentially measures how much one Schreier graph needs to be changed to get to the other graph. For more on that, see Section 4 of \cite{CL_part1} and the $\Gamma$-graph notion in \cite{BC22}.} of $\rho$ and $\varphi$ is  the following minima: 
\[
d^\frS_{\rm edit}(\rho,\varphi)=\min_{\theta\in \Sym(Y)}(d^\frS(\rho,\varphi^\theta)). 
\]
Note that the edit distance is independent of the specific embedding of $X$ in $Y$.
The edit distance induces a metric on $\IRS_{\fd}(\cF)$. For every $\pi,\pi' \in \IRS_{\fd}(\cF)$, let
\begin{equation}\label{eq:edit_dist_of_IRSs}
    d^\frS_{\rm edit}(\pi,\pi')=\inf\{ d^\frS_{\rm edit}(\rho,\varphi) \mid \rho\in \Phi^{-1}(\pi),\varphi\in \Phi^{-1}(\pi')\}\;,
\end{equation}
where recall the definition of $\Phi$ in~\eqref{correspondence_actions_IRSs}.

\subsubsection{A graph theoretic perspective} \label{sec:graph_theoretical_interpret}

\begin{defn}
 Let $\Gamma$ be a  group with $S$ a finite  generating set. Let $X$ be a set and $\sigma\colon \Gamma \to \Sym(X)$ an action of $\Gamma$ on $X$. The \emph{generalized Schreier graph}  $\Sch(\sigma,S)$ is a directed graph with vertex set $X$ and edges labeled by $S$. The edges of $\Sch(\sigma,S)$ are defined as follows: For every $x\in X$  and $s\in S$, there is a directed edge from  $\sigma(s).x$ to $x$ labeled by $s$, namely
\[
\sigma(s).x\xrightarrow[]{s} x.
\]
We may say that an edge is labeled by some inverse of $s\in S$, such as $x\xrightarrow{s^{-1}}y$. By that we mean the edge labeled by $s$ and oriented in the other direction, namely $y\xrightarrow{s}x$. This is natural as $\sigma(s^{-1})=\sigma(s)^{-1}$. Moreover, when $S$ generates $\Gamma$, it includes words with inverses, and so when one associates words in the generators to paths in a labeled graph, this interpretation is natural (cf.\ \cite{KAPOVICH2002608}). 
\end{defn}
\begin{rem}
    This  generalizes the usual notion of a Schreier coset graph $\Sch(\Gamma,H,S)$ by choosing $\sigma \colon \Gamma \to {_H\backslash^\Gamma} $ to be the natural action of $\Gamma$ on right $H$-cosets. 
\end{rem} 
The Schreier graph does not provide any new data about $\sigma$, but it  provides a graph theoretic perspective for approaching IRSs.
For example, a fact that we repeatedly use is that orbits of the action are in correspondence with the connected components of the Schreier graph. Moreover, the edit distance is better understood as a distance between graphs: Given two labeled graphs, one wants to transform one to the other. Whenever the endpoint of an edge labeled by $s$ is changed, there is an associated cost $\frS(s)$. The same cost is incurred when deleting completely or adding such an $s$-labeled edge. The goal is to minimize the cost of moving between the graphs (up  to isomorphism). The edit distance is exactly the minimal cost of such a transformation.

\subsubsection{Topology induced by the edit distance}
 It turns out that regardless of the significance function $\frS$, the topology induced by the edit distance is strictly stronger than the weak$^*$ topology:

\begin{prop}[Topology induced by the edit distance] \label{prop:top_edit_dist} We have the following:
\begin{enumerate}
\item Given two sequences $\pi_n,\pi'_n\in \IRS_{\fd}(\cF)$ such that $\pi_n$ converges  to $\pi_\infty\in \IRS(\cF)$ in the weak$^*$ topology and $\lim_{n\to \infty}d^\frS_{\rm edit}(\pi_n,\pi'_n)=0$, then $\pi'_n$ converge to the same $\pi_\infty$ in the weak$^*$ topology.
\item There are sequences $\pi_n\in \IRS_{\fd}(\cF)$ that converge in the weak$^*$ topology but are not Cauchy sequences in the edit distance topology.
\end{enumerate}
\end{prop}

\begin{rem}
    The above proposition should be contrasted with the \textbf{amenable} case, studied in \cite[Proposition 6.8]{BLT}, where the authors showed that the converse of $(1)$ is true when the free group is replaced by an amenable group. See \cite[Remark 6.9]{BLT} for a different proof of the above clause $(2)$.
\end{rem}

\begin{proof}[Proof of Proposition \ref{prop:top_edit_dist}]
%\textcolor{red}{TO DO. Lewis, can you help me with that?}

Let $\min(\frS)$ be the minimum value of $\frS$. Since $S$ is finite, $\min(\frS)>0$. So $d^\frS_{\rm edit} \ge \min(\frS)d_{\rm edit}$ where $d_{\rm edit}$ is the edit distance with significance function equal to the constant $1$. So $\lim_{n\to \infty}d_{\rm edit}(\pi_n,\pi'_n)=0$. This reduces the problem to the ordinary edit distance.

Let $K,L$ be finite subsets of $\cF$. By Remark \ref{R:weak}, we have
$$\lim_{n\to\infty} \pi_n(\cC(K,L)) = \pi_\infty(\cC(K,L)).$$
So it suffices to prove
$$\lim_{n\to\infty} |\pi_n(\cC(K,L))-\pi'_n(\cC(K,L))| = 0.$$
Let $r>0$ be a large enough radius so that $K\cup L$ is contained in the ball $B(r) \subset \cF$ of radius $r$ with respect to the word metric induced by $S$.

Because $\pi_n$ and $\pi'_n$ are finitely described, there are actions $\rho_n\colon S\to \Sym(X_n)$ and $\rho'_n\colon S\to \Sym(Y_n)$ which induce $\pi_n$ and $\pi'_n$ respectively. We may assume that for each $n$, either $X_n \subset Y_n$ or $Y_n \subset X_n$. After conjugating if necessary, we may further assume
$$d_{\rm edit}(\rho_n,\rho'_n) = d(\rho_n,\rho'_n) = \sum_{s\in S} d_H(\rho_n(s),\rho'_n(s)).$$

Given an integer $k \ge 1$, let $\Bad(k)$ be the set of all $v\in X_n \cup Y_n$ such that there exists $g$ in the ball $B(k)\subset \cF$ with $\rho_n(g)^{-1}v \ne \rho'_n(g)^{-1}v$. The definition of edit distance gives
$$|\Bad(1)| \le d_{\rm edit}(\rho_n,\rho'_n) |X_n \cup Y_n|.$$
Suppose $v \in \Bad(k) \setminus \Bad(k-1)$ for some $k\ge 2$. Then there is a $g \in B(k)$ such that $\rho_n(g)^{-1}v \ne \rho'_n(g)^{-1}v$. Let $g=s_1\cdots s_k$ with $s_i \in S \cup S^{-1}$. Since $v \notin \Bad(k-1)$, if $h=s_1\cdots s_{k-1}$, then $w:=\rho_n(h)^{-1}v = \rho'_n(h)^{-1}v$ and $\rho_n(s_k)^{-1}w \ne \rho'_n(s_k)^{-1}w$. Hence $w \in \Bad(1)$ and so $v \in B(k-1)\Bad(1)$. Therefore,
$$|\Bad(k)| \le |B(k-1)| |\Bad(1)| \le |B(k-1)|d_{\rm edit}(\rho_n,\rho'_n) |X_n \cup Y_n|.$$
Since $|B(k-1)|$ is independent of $n$, and $d_{\rm edit}(\rho_n,\rho'_n)\xrightarrow{n\to \infty} 0$, we can deduce that
$$\lim_{n\to\infty} \frac{|\Bad(k)|}{|X_n \cup Y_n|} = 0$$
for all $k$. 

Because $K \cup L \subset B(r)$, if $v \notin \Bad(r)$, then $\Stab(\rho_n, v) \in \cC(K,L)$ if and only if $\Stab(\rho'_n, v) \in \cC(K,L)$. Hence
$$|\pi_n(\cC(K,L))-\pi'_n(\cC(K,L))| \le \frac{|\Bad(k)|}{|X_n \cup Y_n|} \to 0$$
as $n\to\infty$. This finishes the proof of (1).

To prove item (2), recall that if $G=(V,E)$ is a graph then a subset $I \subset V$ is independent if no edge in $E$ has both endpoints in $I$. If $G$ is finite then the independence ratio of $G$ is $\alpha(G)=|I|/|V|$ where $I$ is an independent set of maximum cardinality. If $\rho:S \to \Sym(n)$ is an action then we let $\alpha(\rho)=\alpha(G)$ where $G$ is the associated action graph.

For any actions $\rho$ and $\rho'$,  we have $d_{edit}(\rho,\rho') \ge |\alpha(\rho)-\alpha(\rho')|$. So it suffices to show there exist action sequences $\{\rho_n\}_n$, $\{\rho'_n\}_n$ whose corresponding IRSs weak$^*$ converge to the same IRS and  $|\alpha(\rho)-\alpha(\rho')|>c$ for some constant $c>0$.

By using a first moment argument, it can be shown that if $\rho_n:S \to \Sym(n)$ is chosen uniformly at random and if $G_n$ is the associated action graph then with high probability $\alpha(G_n)<1/2 - c$ where $c>0$ is a constant depending only on $|S|$ (which we assume is at least 2). Moreover,  there are sets $\Omega_n \subset \Sym(n)^S$ with $|\Omega_n|/|\Sym(n)^S| \to 1$ as $n\to\infty$ such that if $\rho_n \in \Omega_n$ for all $n$ and $\pi_n$ are the corresponding IRSs then $\pi_n$ weak$^*$ converges to the trivial subgroup $\{\Id\}$. This can be derived by estimating the expected number of short cycles in the graph $G_n$ \cite{MR1725006}. A more precise estimate on $\Omega_n$ is obtained in \cite{MR4383230}.

Therefore, there exists a sequence $\{\rho_n\}$ of actions such that if $\pi_n$ are the corresponding IRSs then $\pi_n$ weak$^*$ converges to the trivial subgroup $\{\Id\}$ and $\alpha(G_n) < 1/2 - c$.

On the other hand, there exist $\rho'_n:S \to \Sym(n)$ whose action graphs are bi-partite such that if $\pi'_n$ is the corresponding IRS then $\pi'_n$ also weak$^*$ converges to the trivial subgroup $\{\Id\}$. Since the independence ratio in this case is 1/2, it follows that the sequence $\pi_1,\pi'_1,\pi_2,\pi'_2,\ldots$ cannot be Cauchy in the edit distance.

%Item (2) is a bit trivial. Let $(\pi_n)_n \subset \IRS_{\fd}(\cF)$ be any sequence which weak$^*$ converges to a measure $\pi_\infty \notin \IRS_{\fd}(\cF)$. Let $\rho_n\colon S\to \Sym(X_n)$ be an action which induces $\pi_n$. Since $\pi_\infty \notin \IRS_{\fd}(\cF)$, we must have $|X_n| \to \infty$ as $n\to\infty$. 

%Let $(i_n)_n$ be an increasing sequence such that $|X_{i_n}| > 2|X_n|$ for all $n$. Let $\pi'_n = \pi_{i_n}$. Then $d_{\rm edit}(\pi_n, \pi'_n) \ge 1/2$ for all $n$. So the sequence $\pi_1, \pi'_1,\pi_2,\pi'_2,\ldots$ is not Cauchy in the edit distance topology but it converges in the weak$^*$ topology to $\pi_\infty$.

%A less trivial statement is that there are sequences which weak* converge but do not converge in the ``local-global topology'' \cite{abert2011space}. 

\end{proof}

\subsection{Robustness}
Given $s\in S$, we can define a map $v_s\colon \cF\to \mathbb{N}$ that given a reduced word $w\in \cF$ counts the number of appearances of $s$ or $s^{-1}$ in $w$. Namely, if $w=\prod_{j=1}^\ell s_j^{\eps_j}$,  where $s_j\in S$ and $\eps_j\in \{\pm 1\}$, is a reduced word, then $v_s(w)=\sum_{j=1}^\ell {\bf 1}_s(s_j)$.

Let $\cT$ be 
a test with its usual associated data $S$, $(K_i;D_i)_{i\in Q}$ and $\mu$. We define $\frS_\cT\colon S\to \mathbb{R}_{>0}$, the \emph{significance function associated with $\cT$}, as follows: 
\begin{equation} \label{equation:weight_function_of_a_test}
\frS_\cT(s)=\Ex_{i\sim \mu}\left[\sum_{w\in K_{i}} v_s(w)\right].
\end{equation}

\begin{claim} \label{claim:test_weighted_edit_distance_implies_val_close}
    Let $\eps>0$, and $\cT$ be 
a test. Let $X$ and $Y$ be finite sets, and $\rho\colon S\to \Sym(X)$ and $\varphi\colon S\to \Sym(Y)$ be actions. Assume $d^{\frS_\cT}_{\rm edit}(\rho,\varphi)\leq \eps$. Then $|\val(\cT,\Phi(\rho))-\val(\cT,\Phi(\varphi))|\leq \eps$.
\end{claim}

\begin{proof}
    Assume without loss of generality that $X\subseteq Y$ and that $d^{\frS_\cT}_{\rm edit}(\rho,\varphi)=d^{\frS_\cT}(\rho,\varphi)$. Let $i\in Q$. For every $w=\prod_{j=1}^\ell s_j^{\eps_j}\in \cF$, we have
    \begin{equation} \label{eq1_Claim_test_weighted_edit_distance}
    \Pro_{y\in Y}\left[{\bf 1}_{\Stab(\rho,y)}(w)\neq{\bf 1}_{\Stab(\varphi,y)}(w)\right]\leq  \Pro_{y\in Y}\left[\rho(w).y\neq\varphi(w).y\right].
\end{equation}
Inequality (\ref{eq1_Claim_test_weighted_edit_distance}) is a consequence of the following implication: If  $\rho(w)$ acts on $y$ the same way as $\varphi(w)$, then in particular  $w$ is either in both $\Stab(\rho,y)$ and $\Stab(\varphi,y)$ or in none of them. For $1\leq k\leq \ell$, let $w_{k}$ be the suffix of $w$ from position $k$, namely $w_k=\prod_{j=k}^\ell s_j^{\eps_j}$. Moreover, let $w_{\ell+1}$ be the empty word $\Id_\cF$. Then 
\begin{equation}\label{eq2_Claim_test_weighted_edit_distance}
\Pro_{y\in Y}\left[\rho(w).y\neq\varphi(w).y\right]\leq\Pro_{y\in Y}\left[\bigvee_{1\leq k\leq \ell}\rho\left(w_k\right). y \neq \varphi\left(w_k\right). y\right].
    \end{equation}
Inequality (\ref{eq2_Claim_test_weighted_edit_distance}) is a consequence of  the following implication: If the endpoint of the path beginning at $y$ and labeled by $w^{-1}$ in $\Sch(\rho,S)$ is different than the endpoint of  the same path in $\Sch(\varphi,S)$, then the two paths diverged at \emph{some} point. Now,
   \begin{equation} \label{eq3_Claim_test_weighted_edit_distance}
   \Pro_{y\in Y}\left[\bigvee_{1\leq k\leq \ell}\rho\left(w_k\right). y \neq \varphi\left(w_k\right) .y\right]\leq\sum_{k=1}^\ell\Pro_{y\in Y}\left[\rho\left(w_k\right). y \neq \varphi\left(w_k\right). y \land  \rho\left(w_{k+1}\right). y = \varphi\left(w_{k+1}\right). y\right].
   \end{equation}
Inequality (\ref{eq3_Claim_test_weighted_edit_distance}) is a consequence of the union bound and the logic tautology $\bigvee_{j=1}^m \varphi_j=\bigvee_{j=1}^m(\varphi_j\land \lnot \varphi_{j-1})$, where $\varphi_0={\rm False}$. But, since 
   \[
\rho\left(w_k\right) .y \neq \varphi\left(w_k\right) .y \land  \rho\left(w_{k+1}\right) .y = \varphi\left(w_{k+1}\right) .y \implies  \rho(s_k^{\eps_k})\circ \rho(w_{k+1}).y\neq \varphi(s_k^{\eps_k}) \circ \rho(w_{k+1}).y,
   \]
   and since $\rho(w_{k+1}).y$ is uniformly distributed given that $y$ is uniformly distributed, we can deduce that 
    \begin{equation}
    \Pro_{y\in Y}\left[\rho\left(w_k\right) .y \neq \varphi\left(w_k\right) .y \land  \rho\left(w_{k+1}\right). y = \varphi\left(w_{k+1}\right) .y\right]
    \leq  \underbrace{\Pro_{y\in Y}\left[\rho(s_k) .y\neq \varphi(s_k) .y \right]}_{=d_H(\rho(s_k),\varphi(s_k))}.
    \end{equation}
   All in all, 
\begin{equation}\label{eq4_Claim_test_weighted_edit_distance}
       \Pro_{y\in Y}\left[{\bf 1}_{\Stab(\rho,y)}(w)\neq{\bf 1}_{\Stab(\varphi,y)}(w)\right]\leq \sum_{j=1}^\ell d_H(\rho(s_j),\varphi(s_j))=\sum_{s\in S}v_s(w)d_H(\rho(s),\varphi(s))
   \end{equation}
Hence, by the union bound and inequality (\ref{eq4_Claim_test_weighted_edit_distance}), we have
\begin{equation}\label{eq5_Claim_test_weighted_edit_distance}
\begin{split}
    \Pro_{i\sim \mu}\Pro_{y\in Y}[{\Stab(\rho,y)}\cap{K_{i}}\neq {\Stab(\varphi,y)}\cap{K_{i}}]&\leq \Ex_{i\sim \mu}\left[\sum_{w\in K_{i}}\sum_{s\in S} v_s(w)d_H(\rho(s),\varphi(s))\right]\\
    &=\sum_{s\in S}\underbrace{\Ex_{i\sim \mu}\left[\sum_{w\in K_{i}}v_s(w)\right]}_{\eqref{equation:weight_function_of_a_test}}d_H(\rho(s),\varphi(s))\\
    &=\sum_{s\in S} \frS_\cT(s)d_H(\rho(s),\varphi(s))\\
    &=d^{\frS_\cT}(\rho,\varphi)\leq\eps.
    \end{split}
\end{equation}
Lastly, by combining  
\[
\begin{split}
   |\val(\cT,\rho)-\val(\cT,\varphi)|&\leq \Pro_{i\sim \mu}\Pro_{y\in Y}[D_i(\Stab(\rho,y)\cap{K_{i}})\neq D_i(\Stab(\varphi,y)\cap{K_{i}})]\\
   &\leq \Pro_{i\sim \mu}\Pro_{y\in Y}[\Stab(\rho,y)\cap{K_{i}}\neq \Stab(\varphi,y)\cap{K_{i}}]
    \end{split}
\]
with 
 (\ref{eq5_Claim_test_weighted_edit_distance}), we deduce the claim.
\end{proof}

\begin{rem}
    For those familiar with  robustness of non-local games, $d^{\frS_\cT}_{\rm edit}$ plays a somewhat analogous role to the state-dependent distance (cf.\ Section 4 of \cite{coladangelo2017robust}, for example).
\end{rem}

Robustness is a reverse implication  to the one in Claim \ref{claim:test_weighted_edit_distance_implies_val_close}.

\begin{defn}\label{defn:robustness}
Let  $\cT$ be a test and $\delta\colon \mathbb{R}_{\geq0}\to \mathbb{R}_{\geq0}$ a function satisfying $\delta(\eps)\xrightarrow{\eps\to 0}0$. We say that $\cT$ is $\delta$-robust if: 
\begin{enumerate}
\item \emph{There exists an optimal strategy}, namely $\pi\in \IRS_{\fd}(\cF)$ such that $\val(\cT,\pi)=\val_{\sof}(\cT)$.
\item \emph{Almost optimal strategies are close to optimal ones}, namely, if  $\pi'\in \IRS_{\fd}(\cF)$ satisfies $\val(\cT,\pi')\geq\val_{\sof}(\cT)-\eps$, then there is an optimal strategy $\pi$ such that $d^{\frS_\cT}_{\rm edit}(\pi,\pi')\leq \delta(\eps)$.
\end{enumerate}
\end{defn}

\begin{defn}[See \cite{CL_part1}]
Let $\Gamma=\langle S|R\rangle$ be a finitely presented group. Then, $\Gamma$ is said to be $\delta$-homomorphism stable, where $\delta(\eps)\xrightarrow{\eps\to 0}0$, if for every $\rho\colon S\to \Sym(X)$ where  $\Ex_{r\in R}[d_H(\rho(r),\Id_X)]\leq \eps$ there is a $\Gamma$-action $\varphi\colon S\to \Sym(Y)$, where $X\subseteq Y$ and such that $\Ex_{s\in S}[d_H(\rho(s),\varphi(s))]\leq \delta(\eps).$
\end{defn}
\begin{rem}
    This notion is more commonly known as \emph{flexible pointwise group stability in permutations}. See \cites{GlebskyRivera,ArzhantsevaPaunescu,BLT,BeckerLubotzky,CL_part1,CVY_efficient} for various results in this theory.
\end{rem}
\begin{fact}
    The group $\Gamma=\langle S|R\rangle$ is $\delta$-homomorphism stable if and only if the  $R$-verification test $\cV_R$ is $\delta'$-robust. The exact relation between $\delta$ and $\delta'$ can be calculated explicitly --- they are constant multiples of one another, where the constants depend only on $S$ and $R$.
\end{fact}

\section{Tailored non-local games}\label{sec:Synch}

We commonly use both $\{0,1\}$ and $\FF_2$ for the set with $2$-modular arithmetic.

\subsection{Non-local games preliminaries}
We repeat the definitions which appeared in Section \ref{sec:intro_non-local} of the introduction. Since this paper is not focused on non-local games, we treat them somewhat technically. For formal definitions associated with the complexity class $\MIP^*$, we refer to~\cite{watrous2009quantum,vidick2016quantum}. For the connection between $\MIP^*$ and nonlocal games, a good starting point is~\cite{cleve2004consequences}.

\begin{defn}\label{defn:non-local_game}
   A (synchronous) \emph{non-local game} $\cG$ consists of a finite graph $G=(V,E)$, a length function $\ell\colon V\to \mathbb{N}$, formal sets of generators $S_x=\{\sX^{x,i}\mid 1\leq i\leq \ell(x)\}$ for every vertex $x\in V$, a distribution $\mu$ over the edges $E$, and decision functions $D_{xy}\colon \{0,1\}^{S_{xy}}\to \{0,1\}$ for every edge $xy\in E$, where $S_{xy}=S_x\cup S_y$. We denote by $S$ the set $\bigcup_{x\in V} S_x$.
\end{defn}
\begin{rem}\label{rem:standard_defn_games_and_drama}
   Anyone familiar with the definition of a non-local game will immediately notice that the standard formalism is different than the one we chose here. Usually, a (synchronous) non-local game is defined as a pair of finite sets $X,A$, commonly referred to as the question and answer sets respectively, together with a   distribution $\mu$ over pairs of questions $X\times X$ and a decision predicate $D\colon X\times X\times A\times A\to \{0,1\}$ satisfying $D(x,x,a,b)=0$ for any $x\in X$ and $a\neq b\in A$. 

   It is quite straightforward to move between the definitions.   A way one can extract the data of Definition \ref{defn:non-local_game} from the above is as follows: Let $\ell$ be the constant function $\Lambda=\lceil \log |A|\rceil$, and fix an embedding of $A$ into $\{0,1\}^\Lambda$. The vertices $V$ of the underlying graph $G$ will be $X$, and the support of $\mu$ will be the edge set $E\subseteq X\times X$. The formal generator $\sX^{x,i}$ corresponds to the $i^{\rm th}$ bit of the answer $a$ when $x\in X$ is asked as a question. 
   Then, given that $x,y$ were asked, a pair of answers $a,b$ can be encoded as a map $\gamma \colon S_x\cup S_y\to \{0,1\}$, where 
   \begin{equation}\label{eq:a_b_from_gamma}
   \gamma|_{S_x}=a\quad{\rm and}\quad \gamma|_{S_y}=b\;.
   \end{equation} 
  The  notation \eqref{eq:a_b_from_gamma} will be used repeatedly in the text. 
   Lastly, $D_{xy}(\gamma)=D(x,y,a,b)$, where $\gamma$ is the aforementioned encoding. 

   In the other direction, we can define $A$ to be all bit strings shorter than $\Lambda=\max_{x\in V} \ell(x)$, $X$ to be the vertex set $V$, and $\mu$ stays the same. Finally, $D(x,y,a,b)=D_{xy}(\gamma)$ when $a,b$ are of the correct lengths, i.e. $a=\ell(x),b=\ell(y)$, and is $0$ when either of them is of the incorrect length. Since we consider only synchronous strategies, the condition $D(x,x,a,b)$ is never checked in practice for $a\neq b$ and we can assume it was satisfied beforehand.
Because of this correspondence, we may refer to $\ell$ as the \emph{answer length} function.
\\

Non-local games are often dramatized as two-prover interactive proofs with one round \cite{ben1988multi}. This perspective  may help some readers to absorb the upcoming technicalities better. 
Two players, that can share some resources beforehand (random bits in the classical case and entangled qubits in the quantum case), are separated spatially --- e.g., they are seated in far away rooms. 
A referee  samples a pair of questions and sends one to each player. 
The players then use their shared resources to come up with answers, and send them back to the referee. 
The referee then decides, using the decision predicate, whether the players won or lost. 
The decision predicate as well as the distribution over possible questions are assumed to be known to the players beforehand.  
We demonstrate this dramatization in Example \ref{example:LCSs}.
   
\end{rem}

The following is the quantum mechanics' analogue of a probability distribution. Similar to the way probability distributions (over finite sets) are collections of non-negative real numbers that add up to $1$, the quantum analogue would be ``non-negative'' matrices that add up to the identity. 

\begin{defn}\label{defn:POVM}
    A \emph{positive operator valued measure} (POVM) of dimension $n$ with outcomes in a set $A$  is a mapping $\sP\colon A\to M_{n\times n}(\mathbb{C})$ such that $\sP_a$ is a positive (semi-definite) matrix for every $a\in A$ and $\sum_{a\in A}\sP_a=\Id_n$. It is called \emph{projective}, or a \emph{projection valued measure} (PVM), if in addition $\sP_a$ is an orthogonal projection for every $a\in A$, namely $(\sP_a)^2=\sP_a=(\sP_a)^*$, where $*$ is the conjugate transpose operation. 
Given that $\sP$ is a PVM and $A=\{0,1\}^\Lambda$,  we define for $1\leq i\leq \Lambda$ the $i^{\rm th}$ marginal of $\sP$ to be $\sP^i\colon \{0,1\}\to M_{n\times n}(\mathbb{C})$  by 
\begin{equation}\label{eq:marginalization_of_PVM}
    \sP^i_0=\sum_{\substack{a\colon \Lambda \to \{0,1\}\\ a_i=0}}\sP_a\quad{\rm and}\quad \sP^i_1=\sum_{\substack{a\colon \Lambda \to \{0,1\}\\ a_i=1}}\sP_a.
\end{equation}
Thus, $\sP_a=\prod_{i=1}^\Lambda \sP^i_{a_i}$. Furthermore, the matrix $\sU^i=\sP^i_0-\sP^i_1$ is an order $2$  unitary, which we call the $i^{\rm th}$ binary observable. All in all, a PVM can be given either as a map $\sP\colon \{0,1\}^\Lambda \to M_{n\times n}(\mathbb{N})$, or as a map $\sU\colon [\Lambda]\to U(n)$ such that the images of $\sU$ are commuting involutions --- i.e., they square to the identity.

As its name suggests, every POVM $\sP$ defines a probability distribution over its answer set $A$ as follows
\begin{equation}\label{eq:sampling_according_POVM}
    \Pro [a\ {\rm is\ sampled}]=\tau(\sP_a),
\end{equation}
where $\tau$ is the dimension normalized trace on $n\times n$ matrices. Such an answer is said to be \emph{sampled according} to $\sP$ and we denote it by $a\sim \sP$.
\end{defn}
\begin{rem}
    In the case of a PVM with outcome set $A=\{0,1\}^\Lambda$, the observables $\{\sU^i\}_{i=1}^\Lambda$ generate a unitary representation of $\FF_2^\Lambda$. The $\sP_a$'s in this case can be read from the unitary representation using the Fourier transform of the representation. Hence, for the case of $A=\{0,1\}^\Lambda$, a PVM can be given as a collection of mutually perpendicular ortohgonal projections that sum up to the identity, or as a unitary representation of $\FF_2^\Lambda$, and one can move from one perspective to the other without losing information.
\end{rem}
\begin{defn}\label{defn:strategy_for_a_game}
     A (synchronous quantum) $n$-dimensional \emph{strategy} for $\cG$ is a map $\rho\colon S\to U(n)$, where the images are involutions, and $\rho({S_x})$ commutes for every fixed $x\in V$ --- by saying that a set of matrices (or later permutations) commutes, we mean that every pair of elements in this set commute.
     Another way of viewing strategies is by saying that they associate with every vertex $x\in V$ an $n$-dimensional PVM (in observable form).
     The strategy $\rho$ is said to be \emph{commuting along edges} if $\rho(S_{xy})$ commutes for every edge $xy\in E$.

     Such a strategy defines for every edge $xy\in E$ a probability distribution over functions $\gamma\colon S_{xy}\to \{0,1\}$ as follows:
     \begin{equation}\label{eq:def-gamma-samp}
\Pro[\gamma\colon S_{xy}\to \{0,1\}\  {\rm is\ sampled}]=\tau\left(\prod_{\sX\in S_{x}}\sP^{\sX}_{\gamma(\sX)}\prod_{\sY\in S_y}\sP^{\sY}_{\gamma(\sY)}\right),
     \end{equation}
where $\tau$ is again the dimension normalized trace on $n\times n$ matrices, $\sP^{\sX}_0\colon \mathbb{C}^n\to\mathbb{C}^n $ is the projection on the $(+1)$-eigenspace of $\rho(\sX)$, and $\sP^\sX_{1}\colon \mathbb{C}^n\to\mathbb{C}^n$ is the projection on its $(-1)$-eigenspace.\footnote{Note that $\rho(\sX)=\sP^\sX_0-\sP^\sX_1$, which exactly shows how to move from this observable form of the PVMs induced by $\rho$ back to its projection based form.} We say that $\gamma$ was \emph{sampled according to} $\rho$ if it has this distribution, and we denote it by $\gamma\sim \rho$.

For $\gamma\colon S_{xy}\to \{0,1\}$, let $\sP^x_\gamma=\prod_{\sX\in S_{x}}\sP^{\sX}_{\gamma(\sX)}.$ Then, the map $\sP^{xy}\colon \{0,1\}^{S_{xy}}\to M_{n\times n}(\mathbb{C})$ defined by $\sP^{xy}_\gamma=\sP^x_\gamma\sP^y_\gamma\sP^x_\gamma$  is  a POVM. Thus, $\gamma\sim \rho$ is a special case of sampling according to a POVM \eqref{eq:sampling_according_POVM}, which in this case is $\sP^{xy}$. In the spirit of the notation in \eqref{eq:a_b_from_gamma}, it may be appropriate to write $\sP^x_\gamma=\sP^x_a$ and $\sP^y_\gamma=\sP^y_b$.
\end{defn}
\begin{rem}
    The images of $\sX$ according to $\rho$ are usually called \emph{binary observables}, and the value $\gamma(\sX)$ is usually called the \emph{measurement outcome}.
\end{rem}

\begin{claim} \label{claim:Born's_rule}
    There is a procedural (i.e., algorithmic) way of  sampling  $\gamma\colon S_{xy}\to \{0,1\}$ {according to} $\rho$: 
Since $\rho(S_x)\subseteq U(n)$ commutes, these matrices have a mutual orthonormal basis of eigenvectors $B_x$. Similarly, $\rho(S_y)$ have a mutual orthonormal basis of eigenvectors $B_y$.
 First, sample $\vec v\in B_x$ uniformly at random. 
 Then, sample $\vec u\in B_y$ according to the  squared length of its projection on $v$, namely $$\Pro[\vec u\in B_y\ {\rm is\ sampled}\mid \vec v\in B_x\ {\rm was\ sampled}]=|\langle \vec v|\vec u\rangle |^2,$$
        where $\langle \cdot| \cdot \rangle$ is the standard inner product on $\mathbb{C}^n$, i.e., $\langle \vec v|\vec u\rangle=\sum_{i=1}^n \overline v_i\cdot u_i$.
 Finally,  $\gamma\colon S_{xy}\to \{0,1\}$ is determined as follows --- for $\sX\in S_x$, $\gamma(\sX)=\begin{cases}
        0 & \rho(\sX)\vec v=\vec v,\\
        1 & \rho (\sX)\vec v =-\vec v,
    \end{cases}$, and for $ \sY\in S_y$, $\gamma(\sY)=\begin{cases}
        0 & \rho(\sY)\vec u=\vec u,\\
        1 & \rho (\sY)\vec u =-\vec u.
    \end{cases}$.
        All in all, 
        \[
\Pro[ (\vec v,\vec u)\ {\rm is\ sampled}]=\nicefrac{|\langle \vec v|\vec u\rangle|^2}{n},
        \]
        which is independent of the ordering between $x$ and $y$. Note also that the probability a specific $\gamma$ is measured is  independent of the  chosen bases $B_x$ and $B_y$.
\end{claim}

\begin{proof}
The function $\gamma\colon S_{xy}\to \FF_2$ is sampled in the above procedure if for the resulting pair $(\vec v,\vec u)$ we have $\vec v\in \Img\left( \prod_{\sX\in S_x}\sP^\sX_{\gamma(\sX)}\right)$ and $\vec u\in \Img\left(\prod_{\sY\in S_y}\sP^\sY_{\gamma(\sY)}\right)$. Denote by $\sP^x_{\gamma}$ the product $\prod_{\sX\in S_x}\sP^\sX_{\gamma(\sX)}$, and similarly $\sP^y_\gamma$. So, 
\[
\Pro [\gamma\ {\rm is\ sampled}]=\sum_{\substack{\vec v\in \Img(\sP^x_\gamma)\\ \vec u\in \Img(\sP^y_\gamma)}}\frac{|\langle \vec v|\vec u\rangle |^2}{n}.
\]
By the fact that $B_x$ is a mutual orthonormal basis of eigenvectors for $\rho(S_X)$, every $\vec v\in B_x$ is either in the image of $\sP^x_\gamma$ or in its kernel. Therefore, the above sum can be written as 
\[
\sum_{\substack{\vec v\in \Img(\sP^x_\gamma)\\ \vec u\in \Img(\sP^y_\gamma)}}\frac{|\langle \vec v|\vec u\rangle |^2}{n}=\sum_{\substack{\vec v\in B_x\\ \vec u\in B_y}}\frac{|\langle \sP^x_\gamma\vec v|\sP^y_\gamma\vec u\rangle |^2}{n}=\sum_{\substack{\vec v\in B_x\\ \vec u\in B_y}}\frac{|\langle \sP^y_\gamma\sP^x_\gamma\vec v|\vec u\rangle |^2}{n}.
\]
Since $B_y$ is an orthonormal basis and $\sP^x_\gamma,\sP^y_\gamma$ are orthogonal projections,  for every $\vec v\in B_x$ we have 
\[
\sum_{\vec u\in B_y}|\langle \sP^y_\gamma\sP^x_\gamma\vec v|\vec u\rangle |^2=\|\sP^y_\gamma\sP^x_\gamma\vec v\|^2=\langle \sP^y_\gamma\sP^x_\gamma\vec v|\sP^y_\gamma\sP^x_\gamma\vec v\rangle =\langle \sP^x_\gamma\sP^y_\gamma\sP^x_\gamma\vec v|\vec v\rangle.
\]
Thus,
\[
\sum_{\substack{\vec v\in B_x\\ \vec u\in B_y}}\frac{|\langle \sP^y_\gamma\sP^x_\gamma\vec v|\vec u\rangle |^2}{n}=\sum_{\vec v\in B_x}\frac{\langle \sP^x_\gamma\sP^y_\gamma\sP^x_\gamma\vec v|\vec v\rangle}{n}=\tau(\sP^x_\gamma\sP^y_\gamma \sP^x_\gamma)=\tau(\sP^x_\gamma\sP^y_\gamma),
\]
where the last equality uses the cyclicity of the trace and the fact $\sP^x_\gamma$ is a projection. All in all, 
\[
\Pro [\gamma\ {\rm is\ sampled}]=\tau(\sP^x_\gamma\sP^y_\gamma),
\]
as needed.
\end{proof}

 \begin{rem}\label{rem:Born_rule}
The formula~\eqref{eq:def-gamma-samp} for the probability a specific $\gamma$ is sampled is motivated by the Born measurement rule from quantum mechanics. In quantum mechanics, a (projective) measurement is specified by choosing a PVM  $\sP\colon A\to B(\mathcal{H})$ acting on a separable Hilbert space $\mathcal{H}$ (Definition \ref{defn:POVM} provided the finite dimensional case). 
Given a quantum state (unit vector) $\psi\in \mathcal{H}$, a measurement of $\psi$ returns the outcome $a$ with probability $$\Pro[a \ {\rm is\ the\  measurement\ outcome}]=\psi^* \sP_a \psi.$$ 
If there are two systems (e.g.\ two ``players''), then a Hilbert space $\mathcal{H}_A$, $\mathcal{H}_B$ is associated to each of them
respectively, such that the \textbf{joint} Hilbert space is $\mathcal{H}=\mathcal{H}_A \otimes \mathcal{H}_B$. 
The joint distribution on outcomes is then given by probabilities 
\begin{equation}\label{eq1:born_rem}
    \Pro[(a,b)\ {\rm are\ the\  measurement\ outcomes}]=\psi^* (\sP_a^x \otimes (\sP^y_b)^\dagger) \psi,
\end{equation}
where $\dagger$ is transposition with \textbf{no} complex conjugation, as opposed to $*$.\footnote{It is common to drop the transposition $\dagger$ on the $\mathcal{H}_B$ POVM in \eqref{eq1:born_rem}. But, in this case, a transposition will appear in the calculation using traces \eqref{eq2:born_rem}. So, we decided on this equivalent formulation.}
The case where $\mathcal{H}_A$ and $\mathcal{H}_B$ have the same (finite) dimension $n$, and $\psi= n^{-1/2}\sum_i e_{x,i}\otimes e_{y,i}$ for orthonormal bases $\{e_{x,i}\}$ and $\{e_{y,i}\}$ of $\mathcal{H}_A$ and $\mathcal{H}_B$ respectively, termed ``maximally entangled,'' yields 
\begin{equation}\label{eq2:born_rem}
    \Pro[(a,b)\ {\rm are\ the\  measurement\ outcomes}]=\tau(\sP^x_a \sP^y_b).
\end{equation} 
 \end{rem}

\begin{defn}\label{defn:quantum_value}
    The  \emph{value} of the strategy $\rho$ against the non-local game $\cG$ is
\[
\begin{split}
\val(\cG,\rho)&=\Ex_{xy\sim \mu}\Ex_{\gamma\sim \rho}[D_{xy}(\gamma)]\\
&=\sum_{xy\in E}\sum_{\gamma\colon S_{xy}\to \FF_2} \mu(xy)D_{xy}(\gamma)\tau\left(\prod_{\sX\in S_{x}}\sP^{\sX}_{\gamma(\sX)}\prod_{\sY\in S_y}\sP^{\sY}_{\gamma(\sY)}\right).     
\end{split}
\]
 The (synchronous) \emph{quantum value} of $\cG$, $\val^*(\cG)$, is the supremum over all possible strategies against $\cG$.
\end{defn}

The main result in \cite{MIPRE} is a reduction from the Halting Problem to approximating  the quantum value of a game:
\begin{thm}[$\MIP^*=\RE$]\label{thm:MIP*=RE}
There exists a polynomial time algorithm that takes as input (the encoding of) a Turing machine $M$ and outputs (the encoding of) a non-local game $\cG_M$ such that:
\begin{enumerate}
    \item Sampling $xy\in E$ according to  $\mu$ and evaluating  $D_{xy}(\cdot)$ from the encoding of the  game $\cG_M$  can be done in time ${\rm poly}(|M|)$, where $|M|$ is the bit-length of the encoding of $M$.
    \item If $M$ halts, then there exists a quantum strategy $\rho$ that commutes along edges, for which  $\val(\cG_M,\rho)=1$. In particular, $\val^*(\cG_M)=1$.
    \item If $M$ never halts, then $\val^*(\cG_M)\leq \nicefrac{1}{2}$.
\end{enumerate}
\end{thm}
The  goal of the next section is to show that a somewhat stronger variation of Theorem \ref{thm:MIP*=RE}, which is proved in \cite{Tailored_MIPRE},  provides a negative solution to the Aldous--Lyons Conjecture \ref{(Aldous-Lyons)}.

\subsection{Permutation strategies}\label{sec:permutation_strategies}
Our goal in  Section \ref{sec:test_associated_with_game} is to relate values of subgroup tests to values of non-local games and vice versa. The strategies for subgroup tests are finitely described IRSs, which are (induced by) maps of the form $\sigma\colon S\to \Sym(n)$, while quantum strategies for non-local games are maps of the form $\rho\colon S\to U(n)$ adhering to certain conditions. These already seem quite related, as embedding $\Sym(n)$ into $U(n)$ as permutation matrices is natural, and indeed transforms every (nice enough) finitely described IRS into a quantum strategy. But, there is a problem with doing this naively: 
Say that we want to construct a quantum strategy, such that  for a certain $\sX\in S$, regardless of anything else,  $\gamma(\sX)$ will always be $1$. Then, we can achieve it by choosing $\rho(\sX)=-\Id$. But, $-\Id$ is not a permutation matrix. Therefore, any $\sigma$ that will be naively transformed into a quantum strategy by embedding $\Sym(n)$ into $U(n)$ as permutation matrices cannot ensure a specific $\sX$ is always evaluated to $1$. 
To amend that, we use a tool developed in the study of solution groups of linear constraint system games (cf.\  \cite{slofstra2019set}): 
We add a special generator $\sJ$ that plays the role of $-\Id$. 
Since $-\Id$ is a central involution without fixed points, the image of $\sJ$ under $\sigma$ needs to behave the same. 
This forces $\sigma$ to act on an even sized set. 
Now, if we want $\sigma(\sX)=\sigma(\sJ)$ to translate into $\gamma(\sX)$ being always $1$, we will need to restrict ourselves to the $(-1)$-eigenspace of $\sigma(\sJ)$. This leads us to the following two definitions. 

\begin{defn}\label{defn:permutation_strategy}
    A \emph{permutation strategy} for a non-local game $\cG$ is a map 
$\sigma \colon S\cup \{\sJ\}\to \Sym(2n)$
with the following properties:
\begin{enumerate}
    \item \emph{The image of $\sJ$ is a central involution with no fixed points}. Namely, 
    \[
    \begin{split}
            &\sigma(\sJ)^2=\Id,\  d_H(\sigma(\sJ),\Id)=1,\\
            &\forall \sX\in S\ \colon \ \ \sigma(\sJ)\sigma(\sX)=\sigma(\sX)\sigma(\sJ).
    \end{split}
    \]
    \item \emph{The rest of the images are involutions as well}. Namely,
    \[
\forall \sX\in S\ \colon \ \ \sigma(\sX)^2=\Id.
    \]
    \item \emph{The image of $S_x$ commutes}. Namely, for every fixed vertex $x\in V$, 
    \[
\forall \sX,\sX'\in S_x\ \colon \ \ \sigma(\sX)\sigma(\sX')=\sigma(\sX')\sigma(\sX).
    \]
\end{enumerate}
\end{defn}

\begin{rem}\label{rem:flip_along_J-edges}
    Let $\sigma\colon S\cup \{\sJ\}\to \Sym(2n)$ be a permutation strategy for $\cG$. Embed $\Sym(2n)$ in $U(2n)$ as permutation matrices. Note that in this viewpoint, $U(2n)$ is acting naturally on functions $f\colon [2n]\to \mathbb{C}$, and the matrices are represented with respect to the standard basis, which consists of the indicators 
    \[
    \forall \star\in [2n]\ \colon \ \ {\bf 1}_{\star}(\diamond)=\begin{cases}
        1 & \diamond=\star,\\
        0 & \diamond\neq \star.
    \end{cases}
    \]
    Moreover, for a permutation $\zeta\in \Sym(2n)$ and function $f\colon [2n]\to \mathbb{C}$, we have
    \begin{equation}\label{eq:defn_action_of_perm_on_functions}
        \forall \star\in [2n]\ \colon \ \ (\zeta.f)(\star)=f(\zeta^{-1}.\star).
    \end{equation}
    By $(1)$ of Definition \ref{defn:permutation_strategy}, $\sigma(\sJ)$ is a fixed point free involution, and thus its $(+1)$-eigenspace $W^+$ and $(-1)$-eigenspace $W^{-}$ are of the same dimension $n$. Note that $W^+$ consists of all functions $f\colon [2n]\to \mathbb{C}$ that are constant along $\sJ$-labeled edges in the Schreier graph\footnote{Generalized Schreier graphs were defined in  Section \ref{sec:graph_theoretical_interpret}.} $\Sch(\sigma,S\cup \{\sJ\})$, while $W^{-}$ consists of all functions that \textbf{flip} their signs along $\sJ$-labeled edges. Let $\sQ=\sQ^\sJ_1\colon \mathbb{C}^{2n}\to W^{-}$ be the orthogonal  projection on $W^{-}$. Note that $\sQ^*\sQ=\sP^\sJ_1$ and $\sQ\sQ^*=\Id_{W^{-}}$.
    
    We can choose bases for $W^+$ and $W^{-}$ as follows. 
    If we denote 
    \begin{equation}\label{eq:defn_star'}
        \star'=\sigma(\sJ).\star
    \end{equation}
    for every $\star\in [2n]$, then $\star\xleftrightarrow{\sJ} \star'$ is the perfect matching induced by $\sJ$-labeled edges in  $\Sch(\sigma,S\cup \{\sJ\})$. 
    So, an orthogonal basis of $W^+$ may be taken to be  $B^+=\{{\bf 1}_\star+{\bf 1}_{\star'}\}_{\star\in [2n]}$, while for $W^{-}$ one takes  $B^{-}=\{{\bf 1}_\star-{\bf 1}_{\star'}\}_{\star\in [2n]}$ ---  $B^{-}$ contains for each vector also its opposite, so this is a basis when choosing a representative for each such pair. The union of these bases is an orthogonal basis for $\mathbb{C}^{2n}$. Furthermore, we later refer to $B^{-}$ as the \emph{standard basis} of $W^{-}\cong \mathbb{C}^n$.
\end{rem}

\begin{rem}
    Given that $\rho\colon S\cup\{\sJ\}\to \Sym(2n)$ is a permutation strategy, there are $n$ orbits of $\sigma(\sJ)$, and $\sigma$ acts on them in a well-defined manner. If we arbitrarily denote each orbit as $\{+\star,-\star\}$, instead of $\{\star,\star'\}$ a la \eqref{eq:defn_star'}, then $\sigma$ acts naturally on the set $\Omega_{\pm}=\{\pm\}\times \Omega$ in a way that commutes with the sign flip. Hence, a permutation strategy can be defined as a map into \textbf{signed} permutations, where $\sJ$ is mapped to the sign flip, and such that the variables at each vertex are mapped to commuting involutions. This point of view may be helpful for understanding some of the next claims and remarks. We decided not to use the language of signed permutations in this paper, though in the companion paper \cite{Tailored_MIPRE} it is more pronounced. 
\end{rem}

\begin{defn}\label{defn:quantum_strat_associated_with_perm_strat}
 Let $\sigma \colon S\cup \{\sJ\}\to \Sym(2n)$ be a permutation strategy for $\cG$. Recall from Remark \ref{rem:flip_along_J-edges} that $W^{-}$ is the $(-1)$-eigenspace of $\sigma(\sJ)$, and that $\sQ\colon \mathbb{C}^{2n}\to W^{-}$ is the orthogonal projection on $W^{-}$. Then,
the \emph{quantum strategy} $\rho\colon S\to \textrm{End}(W^{-})$ \emph{induced by} $\sigma$ is  the mapping to the $W^{-}$-corners
    \[
    \rho(\sX)=\sQ \sigma(\sX)\sQ^*.
    \]
\end{defn}

\begin{claim}\label{claim:quantum_strat_associated_with_perm_strat}
    Let  $\sigma$ be a permutation strategy. Then, the quantum strategy $\rho$ induced by $\sigma$, as in Definition \ref{defn:quantum_strat_associated_with_perm_strat}, is \textbf{indeed} a quantum strategy for $\cG$, as in Definition \ref{defn:strategy_for_a_game}.
\end{claim}
\begin{proof}
    The criterion for being a quantum strategy as in Definition \ref{defn:strategy_for_a_game} is to be a map into \emph{unitaries} such that the image of the generators $S_x$ associated with any specific vertex \emph{commute}, and all the images are \emph{involutions}.
First,
 \begin{equation}\label{eq:quantum_associated_strat}
     \rho(\sX)^*=(\sQ \sigma(\sX)\sQ^*)^*=\sQ \sigma(\sX)^*\sQ ^*=\sQ  \sigma(\sX)^{-1}\sQ ^*.
 \end{equation}
    Now,  the fact that $\sigma(\sJ)$ commutes with  $\sigma(\sX)$ for any $\sX\in S$ implies that $\sQ^*\sQ\sigma(\sX)=\sigma(\sX)\sQ^*\sQ$. Therefore, 
    \[
    \begin{split}
    \rho(\sX)^*\rho(\sX)&=\sQ \sigma(\sX)^{-1}\sQ^*\sQ \sigma(\sX)\sQ^*\\
    &=\sQ \underbrace{\sigma(\sX)^{-1} \sigma(\sX)}_{\Id_{2n}}\sQ^*\underbrace{\sQ\sQ^*}_{\Id_{W^{-}}}\\
    &=\sQ\sQ^*=\Id_{W^{-}},
    \end{split}
    \]
   and the image of $\rho$ is contained in $U(W^{-})$. Item $(2)$ of Definition \ref{defn:permutation_strategy} states that  $\sigma(\sX)^{-1}=\sigma(\sX)$ for all $\sX\in S$,  and combined with \eqref{eq:quantum_associated_strat} gives
   \[
\rho(\sX)^*=\sQ \sigma(\sX)^{-1}\sQ^*=\sQ\sigma(\sX)\sQ^*=\rho(\sX),
   \]
   i.e., the images of $\rho$ are involutions. Finally, since $\sigma(S_x)$ commutes for every fixed $x\in V$, we have
   \[
   \begin{split}
       \forall \sX,\sY\in S_x\ \colon \ \ \rho(\sX)\rho(\sY)&=\sQ \sigma(\sX)\sQ^*\sQ \sigma(\sY)\sQ^*\\
       &=\sQ\sQ^*\sQ  \sigma(\sX)\sigma(\sY)\sQ^*\\
       &=\sQ\sQ^*\sQ  \sigma(\sY)\sigma(\sX)\sQ^*\\
       &=\sQ\sigma(\sY)\sQ^*\sQ \sigma(\sX)\sQ^*\\
       &=\rho(\sY)\rho(\sX),
   \end{split}
   \]
   which finishes the proof.
\end{proof}

\begin{example}[Classical permutation strategies]\label{example:classical_perm_strategies}
Let $\Sym(2)=\{\Id,\zeta\}$ be the group of permutations acting on two elements. For every fixed $f\colon S\to \{0,1\}$, we can define an action $\sigma\colon S\cup\{\sJ\}\to \Sym(2)$ as follows
\[
\sigma(\sJ)=\zeta\quad;\quad \sigma(\sX)=\begin{cases}
    \Id & f(\sX)=0,\\
    \zeta & f(\sX)=1.
\end{cases}
\]
It is straightforward to see that the quantum strategy associated with $\sigma$ satisfies $\rho(\sX)=(-1)^{f(\sX)}$, and thus $\gamma\colon S_{xy}\to \{0,1\}$ that is sampled according to $\rho$ is deterministically $f|_{S_{xy}}$.
Strategies  of this form are usually called \emph{deterministic}. By taking direct sums of such actions (for potentially different $f$'s), we can get any (rational) distribution over deterministic strategies. These strategies are usually called \emph{classical}. So, every (rational) classical strategy can be obtained as a permutation strategy. 
\end{example}

\begin{claim}\label{claim:rho_id_implies_sigma_id}
    Let $\sigma\colon S\cup \{\sJ\}\to \Sym(2n)$ be a permutation strategy, and $\rho$ the associated quantum strategy. Then, if for a word $w\in \cF(S\cup \{\sJ\})$ we have $\rho(w)=\Id$, then $\sigma(w)=\Id$.
\end{claim}
\begin{proof}
    Recall from Remark \ref{rem:flip_along_J-edges} that the images of $\rho$ act naturally, and in \textbf{the same way} as $\sigma$,  on $W^{-}$, the $(-1)$-eigenspace of $\sigma(\sJ)$, which consist of functions $f\colon [2n]\to \mathbb{C}$ that flip along $\sJ$-edges. Recall also the notation $\star'=\sigma(\sJ).\star$ from \eqref{eq:defn_star'}.
    So, the fact that $\rho(w)=\Id$ means that for any  function $f\in W^{-}$, $\sigma(w).f=f$ as well.  
    This is true in particular for $B^{-}=\{{\bf 1}_\star-{\bf 1}_{\star'}\}_{\star\in [2n]}$, the basis of $W^{-}$ described in the same remark. Note that for every $\zeta\in \Sym(2n)$ we have $\zeta.{\bf 1}_\star={\bf 1}_{\zeta.\star}$.
    Hence, for every $\star\in[2n]$, we have
\[
{\bf 1}_\star-{\bf 1}_{\star'}=\sigma(w).({\bf 1}_\star-{\bf 1}_{\star'})={\bf 1}_{\sigma(w).\star}-{\bf 1}_{\sigma(w).\star'}.
\]
By evaluating the functions on both sides of the above equation on the point $\star\in [2n]$, we get
\[
1={\bf 1}_{\sigma(w).\star}(\star)-{\bf 1}_{\sigma(w).\star'}(\star),
\]
which means $\sigma(w).\star=\star$. Since $\star$ was any point in $[2n]$, $\sigma(w)=\Id$ and we are done.
\end{proof}

\begin{defn}\label{defn:game_having_perfect_permutation_strategy}
     A non-local game $\cG$ is said to have a \emph{perfect permutation strategy} if there exists a permutation strategy $\sigma$ such that the  quantum strategy $\rho$ induced by it, as in Definition \ref{defn:quantum_strat_associated_with_perm_strat}, satisfies $\val(\cG,\rho)=1$.
 \end{defn}

In the rest of this subsection we further analyze permutation strategies, and specifically discuss a procedural sampling --- similar to  Claim \ref{claim:Born's_rule}  --- of a $\gamma\colon S_{xy}\to \{0,1\}$ according to $\rho$ which is induced by a permutation strategy $\sigma$. 

\subsection{Interlude --- Fourier bases of $\FF_2^m$ actions}\label{sec:Fourier_bases}

\begin{defn}\label{defn:Fourier_basis_F_2^k}
    Let $\sigma\colon \FF_2^k\to \Sym(\FF_2^k)$ be the standard action of $\FF_2^k$ on itself, namely $\sigma(w).v=v+w$ for every $v,w\in \FF_2^k$. The \emph{Fourier basis} of $\sigma$ is the following collection of functions, denoted by $\widehat \FF_2^k$: For every $u\in \FF_2^k$, let $\widehat u\colon \FF_2^k\to \mathbb{C}$ be 
\begin{equation}\label{eq:Fourier_basis}
    \forall w\in \FF_2^k\ \colon \ \ \widehat u(w)=\frac{(-1)^{\langle u,w\rangle}}{\sqrt {2^k}},
\end{equation}
where $\langle u,w\rangle=\sum_{i=1}^k u_iw_i$ is the standard bilinear form on $\FF_2^k$.\footnote{We try to keep $\langle \cdot,\cdot\rangle$ for the bilinear form over a finite vector space, and $\langle \cdot |\cdot \rangle$ for the complex inner product.}
\end{defn}

\begin{claim}
    The collection $\widehat\FF_2^k$ is an orthonormal basis of eigenvectors for the action  $\sigma$. 
\end{claim}
\begin{proof}
    The action $\sigma$ acts as permutations on the standard basis of $\mathbb{C}^{\FF_2^k}$, and thus extends to act on the whole vector space of functions from $\FF_2^k$ to $\mathbb{C}$. Using the standard inner product $\langle \cdot|\cdot\rangle$ on  $\mathbb{C}^{\FF_2^k}$ as before, we can check that 
    \[
\forall u\in \FF_2^k\ \colon \ \ \langle \widehat u|\widehat u\rangle =\sum_{w\in \FF_2^k} \overline{\widehat u(w)}\cdot \widehat u(w)=\sum_{w\in \FF_2^k}\frac{1}{2^k}=1.
    \]
    Furthermore, 
    \[
    \begin{split}
       \forall u\neq v\in \FF_2^k\ \colon \ \ \langle \widehat v|\widehat u\rangle &=\sum_{w\in \FF_2^k} \overline{\widehat v(w)}\cdot \widehat u(w) \\
       &=\frac{1}{2^k}\sum_{w\in \FF_2^k}(-1)^{\langle v+u,w\rangle}\\
       &=0,
    \end{split}
    \]
    where the last equality is derived from the fact $v+u\neq \vec 0$. Lastly, 
    \[
    \begin{split}
\forall u,v,w\in \FF_2^k\ \colon \ \ (\sigma(w).\widehat v)(u)&=\widehat v(u+w)\\
&=\frac{(-1)^{\langle u+w,v\rangle }}{\sqrt{2^k}}\\
&=(-1)^{\langle w,v\rangle}\cdot \frac{(-1)^{\langle u,v\rangle}}{\sqrt{2^k}}\\
&=(-1)^{\langle w,v\rangle}\widehat v(u),
\end{split}
    \]
    and $\widehat\FF_2^k$ is indeed a collection of mutual eigenvectors for $\sigma$.
\end{proof}

Recall basic \emph{group actions} definitions. 
Let $G$ be a group, and $\sigma\colon G\to \Sym(X),\tau\colon G\to \Sym(Y)$ be two homomorphisms, which induce actions of $G$ on $X$ and $Y$ respectively. A map 
$f\colon X\to Y$  is a \emph{factor} (of $G$-actions)  --- also known as a \emph{morphism of $G$ actions} or a  $(\sigma,\tau)$-\emph{equivariance} --- if
for all $w\in G$ and $x\in X$, we have $f(\sigma(w).x)=\tau(w).f(x).$
Note that for every set $S\subseteq G$, such a map induces a morphism of oriented edge-labeled graphs between $\Sch(\sigma,S)$ and $\Sch(\tau,S)$.
The actions $\sigma$ and $\tau$ are said to be \emph{equivalent} if there exists a bijective factor between them, which in the graph case would be a (edge-labeled oriented) graph isomorphism between $\Sch(\sigma,S)$ and $\Sch(\tau,S)$.
Given two actions, we can define their \emph{sum}  $\sigma\oplus\tau\colon G\to \Sym(X\sqcup Y)$, where  $\sqcup$  is the  disjoint union, to be
\[
\forall w\in G, z\in X\sqcup Y\, \colon \quad \\ (\sigma\oplus\tau(w)).z=\begin{cases}
    \sigma(w). z & z\in X,\\
    \tau(w). z & z\in Y.
\end{cases}
\]
\begin{comment}
and their  \emph{product} $\sigma\otimes \tau\colon S\to \Sym(X\times Y)$ to be
\[
\forall s\in S, x\in X,y\in  Y\ \colon\ \ (\sigma \otimes \tau(s)).(x,y)=(\sigma(s).x ,\tau(s).y).
\]
\end{comment}
If $\sigma$ is a transitive action and  $x\in X$, then it is equivalent to the  action of $G$ on $\nicefrac{G}{\Stab(\sigma,x)}$, where $\Stab(\sigma,x)$ is defined (as in Section \ref{Intro_Sec_fin_desc_IRS}) to be
$\Stab(\sigma,x)=\{w\in G\mid \sigma(w).x=x\}.$
The bijective factor in this case is associating with $\sigma(w).x$ the coset $w\cdot\Stab(\sigma,x)$, and recalling that by transitivity every $y\in X$ is of the form $\sigma(w).x$ for some $w\in G$. So,  every action $\sigma\colon G\to \Sym(X)$ is equivalent to the sum of actions of $G$ on quotients of it $\nicefrac{G}{H}$, for some collection of subgroups $H\leq G$. 

Assume now that $G$ is $\FF_2^m$ for some positive integer $m$. Let $\sigma\colon G\to \Sym(X)$ be a transitive action. Since every subgroup $H$ of  $G$ is normal, this action factors through  some $\FF_2^k$ for $k\leq m$. Namely, we can find a bijective factor between $X$ and the action of $\FF_2^k$ on itself: 
 Fix some $\star\in X$; $\Stab(\sigma,\star)$ is a subspace of $\FF_2^m$, and we implicitly assumed it is isomorphic to $\FF_2^{m-k}$; for every $\diamond\in X$, there is an element $g$ of $\FF_2^m$ for which $\sigma(g).\star=\diamond$, and thus $\sigma(g').\star=\diamond$ for every $g'\in g+\Stab(\sigma,\star)$; so we can associate $g+\Stab(\sigma,\star)$ with $\diamond$ and $\nicefrac{\FF_2^{m}}{\Stab(\sigma,\star)}$ with $\FF_2^k$.  
For every $g\in \FF_2^m$, let $\tilde g\in \FF_2^k$ be its image in the bijection between the quotient space $\nicefrac{\FF_2^m}{\Stab(\sigma,\star)}$ and $\FF_2^k$. Then,  
\[
\forall v\in \FF_2^k\ \colon \ \ \sigma(g).\widehat v=(-1)^{\langle v,\tilde g\rangle}\widehat v.
\]
In particular, for $s\in S\subseteq \FF_2^m$, the action of $s$ on $\widehat v$ can be read from what happens along edges labeled by $s$ in $\Sch(\sigma,S)$ --- if $\widehat v(\star)=\widehat v(\sigma(s).\star)$, then $\widehat v$ is in the $(+1)$-eigenspace of $\sigma (s)$, and otherwise it is in its $(-1)$-eigenspace.

All in all, given a not necessarily transitive action $\sigma$ of $\FF_2^m$ on $X$, its orbits have Fourier bases as in Definition \ref{defn:Fourier_basis_F_2^k} (extended to being zero on all vertices outside of the orbit), and  the disjoint union of these will be called  \emph{the Fourier basis} of $\sigma$. 
Note that since the bijections between the orbits and $\FF_2^k$ can be chosen in many ways, the Fourier basis of $\sigma$ is determined only up to the sign of each basis vector.

\subsection{Procedural sampling according to permutation strategies}\label{sec:proc_samp_perm_strat}
Let $\sigma\colon S\cup \{\sJ\}\to \Sym(2n)$ be a permutation strategy for $\cG$, and recall Remark \ref{rem:flip_along_J-edges} and the notation therein. 
By Definition \ref{defn:permutation_strategy}, $\sigma(S_x\cup\{\sJ\})$ consists of commuting involutions. 
Therefore, the group generated by $\sigma(S_x\cup\{\sJ\})$ is a quotient of $\FF_2^{S_x\cup \{\sJ\}}$ --- this can be acquired by linearly extending the map ${\bf 1}_{\sX}\mapsto \sigma(\sX)$ for every $\sX\in S_x\cup \{\sJ\}$, where ${\bf 1}_\sX\colon {S_x\cup \{\sJ\}}\to \FF_2$ is the indicator of $\sX$. 
This means that 
\begin{equation}\label{eq:defn_sigma_x}
    \sigma_x=\sigma|_{S_x\cup\{\sJ\}}
\end{equation} 
induces an action of  $\FF_2^{S_x\cup \{\sJ\}}$ on $[2n]$. 
Choose $B_x \subset \mathbb{C}^{2n}$ to be the Fourier basis of $\sigma_x$, as was defined in the Interlude Section \ref{sec:Fourier_bases}.
Let $B_x^{-}= B_x\cap W^{-}\subset B_x$ be the collection of $n$ eigenvectors in $B_x$ on which $\sigma(\sJ)$ acts as $(-1)$, namely all $\vec v\in B_x$ such that  $\sigma(\sJ).\vec v=-\vec v$. These are the Fourier basis elements that flip along edges labeled by $\sJ$ in $\Sch(\sigma_x,S_x\cup \{\sJ\})$. Then, $B_x^{-}$ is an orthonormal basis of eigenvectors for $\rho(S_x)$.

\begin{claim}\label{claim:procedural_sampling_perm_strat}
    Let  $\sigma\colon S\cup \{\sJ\}\to \Sym(2n)$ be a permutation strategy for $\cG$, and $\rho\colon S\to U(W^{-})$ the quantum strategy associated to $\sigma$. Then, the sampling procedure of $\gamma\colon S_{xy}\to \{0,1\}$ according to $\rho$ can be described  as follows:
    \begin{itemize}
        \item Sample $\star\in [2n]$ uniformly at random. Then, it is part of an orbit $O_x\subseteq [2n]$ of $\sigma_x$ (defined in \eqref{eq:defn_sigma_x}) and $O_y\subseteq [2n]$ of $\sigma_y$.
        \item Choose a pair $\vec v\in B^{-}_x$ and $\vec u\in B^{-}_y$ such that $\vec v$ is supported on $O_x$ and $\vec u$ is supported on $O_y$, with probability 
        \[
        \frac{2\langle \vec v|\vec u\rangle ^2}{|O_x\cap O_y|}.
        \]
        \item For $\sX\in S_x$, let $\gamma(\sX)=0$ if $\rho(\sX).\vec v=\vec v$ and $1$ otherwise. Similarly, for $\sY\in S_y$, let $\gamma(\sY)=0$ if $\rho(\sY).\vec u=\vec u$ and $1$ otherwise.
    \end{itemize}
    We may use the notation $\gamma\sim \sigma$ instead of $\gamma\sim \rho$ to refer to sampling according to a permutation strategy.
\end{claim}

\begin{proof}
The probability of sampling $\star\in [2n]$ which is in $O_x\cap O_y$ is $\frac{|O_x\cap O_y|}{2n}$. Thus, the probability of sampling a specific pair $\vec v\in B_x^-,\vec u\in B_y^-$ whose intersection of supports is $O_x\cap O_y$ is
    \[
\frac{|O_x\cap O_y|}{2n}\cdot \frac{2\langle \vec v|\vec u\rangle^2}{|O_x\cap O_y|}=\frac{\langle \vec v|\vec u\rangle ^2}{n}.
    \]
    This shows that the above sampling procedure provides the same distribution as the one in Claim \ref{claim:Born's_rule}.

\end{proof}

\subsection{Tailored games}
 As permutation strategies allow us to translate finitely described strategies of subgroup tests to quantum strategies of non-local games, the soon to be defined \emph{$Z$-aligned permutation strategies} of \emph{tailored non-local games} will allow us to move in the other direction. It is harder to motivate our choices beforehand in this case, but we hope that Section \ref{sec:test_associated_with_game}  clarifies what may seem cryptic in this section. We begin by restricting the family of non-local games that we focus on --- for the reader familiar with linear constraint system games, we note that this family is a generalization of them.
 
\begin{defn}\label{defn:tailored_games}
Colloquially,  a \emph{tailored} non-local game is one where $D_{xy}$ reads \textbf{part} of $\gamma$, and decides according to this partial view which \textbf{parity checks} to apply on \textbf{the whole} of $\gamma$.\footnote{We considered calling such games \emph{controlled linear}, since it is more informative. But, since conditionally linear is a term used in \cites{MIPRE,Tailored_MIPRE}, and terms containing linear are generally overused, we decided to use a less informative notion.}

Formally,
   a tailored non-local game $\cG$ is equipped with an extra structure, described shortly, and its decision functions $D_{xy}$ behave \emph{canonically} with respect to this extra data.
    Instead of a single length function $\ell$, $\cG$ has two length functions $\ell^\frR\colon V\to \mathbb{N}$ and $\ell^{\frL}\colon V\to \mathbb{N}$, and $\ell=\ell^\frR+\ell^\frL$. Before, the length function described the size of the formal generating set at each vertex. 
    Now, the formal set of generators $S_x$ at $x\in V$ will be a disjoint union of the sets $S_x^\frR$ and $S_x^\frL$, where $S_x^\frR=\{\sX^{x,\frR,i}\mid 1\leq i\leq \ell^{\frR}(x)\}$ and $S_x^\frL=\{\sX^{x,\frL,i}\mid 1\leq i\leq \ell^{\frL}(x)\}$. The elements of  $S_x^\frR$ are  called \emph{readable} variables at $x\in V$ and the elements of $S_x^{\frL}$ the \emph{linear} or \emph{unreadable} variables at $x$. 
    Let $\frR\colon S\to \{0,1\}$ be the \emph{readability function}, namely, the indicator of whether $\sX\in S$ is  readable or not.
      In addition, $\cG$ is equipped with  a collection of \emph{controlled linear constraints} functions $L_{xy}$ that take as input a function $\gamma^\frR \colon S_x^\frR\cup S_y^\frR\to \FF_2$, and outputs a collection of subsets of $S_{xy}\cup \{\sJ\}$. Namely,
    \[
L_{xy}\colon \FF_2^{S_{x}^\frR \cup S_y^\frR}\to \FF_2^{\FF_2^{S_{xy}\cup \{\sJ\}}}.
    \]
    The image of $L_{xy}$ is interpreted as a collection of linear constraints that will be verified by the decision function.   
    Finally, the decision function $D_{xy}(\gamma)$ behaves as follows: It restricts $\gamma$ to the readable variables, namely looks at $\gamma^\frR=\gamma|_{S_x^\frR\cup S_y^\frR}\colon S_x^\frR\cup S_y^\frR\to \FF_2$, and calculates $L_{xy}(\gamma)=L_{xy}(\gamma^\frR)$. Then, it extends $\gamma$ such that $\gamma(\sJ)=1$. Finally, for every  $\alpha\in L_{xy}(\gamma)$, we have $\alpha\colon S_{xy}\cup \{\sJ\}\to \FF_2$,  and $D_{xy}$ verifies that 
    $$\langle \alpha,\gamma\rangle=\sum_{\sX\in S_{xy}\cup \{\sJ\}}\alpha(\sX) \cdot \gamma(\sX)=0.$$ 
    Namely, $L_{xy}(\gamma)$ consists of linear constraints that $\gamma$ needs to satisfy. 
\end{defn}

\begin{rem}\label{rem:naive_tailoring}
    Though being tailored  seems to be quite a restrictive form for a non-local game, every game can be tailored in a trivial manner. First, all variables are declared readable, namely $\ell^\frR=\ell$ and $\ell^\frL=0$. Then, if the decision function $D_{xy}$ decided to accept $\gamma$ according to the original game, then it lets $L_{xy}(\gamma)$ be empty (and thus all linear conditions will be satisfied regardless of what $\gamma$ is). And, if $D_{xy}$ decided to reject $\gamma$ according to the original game, 
    then it chooses $L_{xy}(\gamma)$ to contain the singleton $\{\sJ\}$ as the single subset appearing in  $L_{xy}$. Note that $\{\sJ\}$ represents the linear equation $1\cdot \gamma(\sJ)=0$, which is $1=0$, and thus cannot be satisfied. 
\end{rem}

The last remark raises the question: What have we gained by defining tailored non-local games, if any game can be tailored in a straightforward manner? 

\begin{defn}\label{defn:Z-aligned_strategy}
   A permutation strategy $\sigma$ for a tailored non-local game $\cG$ is said to be $Z$-\emph{aligned} if the assignments  $\sigma(\sX)$ for readable variables act on each point in $[2n]$ either like the identity  or like $\sigma(\sJ)$. This is equivalent to $\rho(\sX)$ being a diagonal matrix when presented according to  the \emph{standard basis} $B^{-}$ of $W^{-}$.\footnote{ The standard basis is usually called the $Z$-basis in quantum information theory, and thus the name $Z$-aligned strategies.}  
\end{defn}

\begin{rem}
    The classical permutation strategies described in Example \ref{example:classical_perm_strategies} are $Z$-aligned. But, one can construct permutation strategies that induce a classical strategy in the usual sense (i.e., whose images are \textbf{all} commuting) without being $Z$-aligned. 
\end{rem}
It is clearer now why the way one tailors a non-local game matters: The existence of a perfect $Z$-aligned permutation strategy for the game  depends on it. Let us demonstrate this with binary linear constraint system (LCS) games, and specifically with the magic square game. For  a full description of  the magic square game we refer to~\cite{aravind2002simple} or~\cite[Example 2.19]{Tailored_MIPRE}, and for a general  introduction to LCS games see~\cite{cleve2017perfect}. 

\begin{example}[Linear constraint system games]\label{example:LCSs}
   Let $A$ be an $m\times n$ matrix with $\FF_2$ coefficients, and let $\vec b$ be a column vector in $\FF_2^m$. Classically, such a pair defines a system of linear equations $A\vec x=\vec b$ over $\FF_2$.  It also defines a certain non-local game $\cG(A,\vec b)$ which is the quantum counterpart of this classical system of equations. Let us provide a dramatization of this game, in the spirit of Remark \ref{rem:standard_defn_games_and_drama}. In  $\cG(A,\vec b)$, the referee chooses a random linear constraint in $A\vec x=\vec b$ (i.e., a row) and sends it to one of the players, and chooses  a random variable that appears  in the chosen constraint (i.e., a column whose intersection with the chosen row is non-zero) to the other player. It then expects the row player to respond with an assignment to the variables appearing in this constraint, and from the column player to respond with an assignment to its single variable. The referee declares this to be a winning condition if the row player's assignment satisfied the constraint, and both players agree on the value of the variable given to the column player. 
   The column player is thought to hold a global assignment and answer according to it, while the row player is assumed to have a satisfying assignment to each constraint --- intuitively, by  checking consistency between the players, the referee verifies that the global assignment of the column player indeed satisfies all the constraints. 
   
   Let us define this game rigorously. 
   The vertices in the underlying graph of $\cG(A,\vec b)$ will be indexed by the rows and columns of the matrix $A$, namely $\{r_i\mid i\in [m]\}$ and $\{c_j\mid j\in [n]\}$. There is an edge between a row $r_i$ and a column $c_j$ if and only if $A_{ij}=1$. The length of every column vertex is $1$, and we denote by $\sX_j$ the variable associated with the $j^{\rm th}$ column $c_j$. The length of each row vertex is the number of $1$'s in the row, and we associate variables $S_{r_i}=\{\sR_{ij'}\mid A_{ij'}=1\}$ to every row $r_i$. The decision function $D_{r_ic_j}$ gets as input an assignment $\gamma$ to $\sX_j$ and $\{\sR_{ij'}\mid A_{ij'}=1\}$, and accepts if and only if 
\begin{equation}\label{eq:LCS_check}
    \sum_{j'\colon A_{ij'}=1}\gamma(\sR_{ij'})=b_i \quad\textrm{and}\quad  \gamma(\sX_j)=\gamma(\sR_{ij}).
\end{equation}
Though for our discussion the distribution $\mu$ over edges in this game is not important, one can consider the following standard sampling scheme: 1) Choose a row uniformly at random. 2) Choose a uniform column out of the support of the chosen row.

Let us describe a non-trivial tailoring of  $\cG(A,\vec b)$. First, all variables are \emph{unreadable}, namely $\ell^\frR=0$ and $\ell^\frL=\ell$. Given that the edge $r_ic_j$ was sampled, the system $L_{r_ic_j}$ will consist of two checks, which are derived from 
\eqref{eq:LCS_check}:\footnote{note that since there are no readable variables, $L_{r_ic_J}$ is constant.} 
\[
\begin{split}
    \alpha_{ {\rm consistency}}(\sX)&=\begin{cases}
    0 & \sX\neq \sX_j,\sR_{ij}\\
    1 & \textrm{otherwise}
\end{cases}\\
\alpha_{ {\rm linear}}(\sX)&=\begin{cases}
    0 & \sX=\sX_j\\
    1 & \sX=\sR_{ij'}\in S_{r_i}\\
    b_i & \sX=\sJ
\end{cases}
\end{split}
\]
Then, $\alpha_{{\rm consistency}}$ will force $D_{r_ic_j}$ to check that $\gamma(\sX_j)=\gamma(\sR_{ij})$, and $\alpha_{{\rm linear}}$ will force it to check that $\displaystyle{\sum_{j'\colon A_{ij'}=1}}\gamma(\sR_{ij'})=b_i$, and thus it acts the same way as before.

   The difference between the above tailored form of $\cG(A,\vec b)$ and the one suggested in Remark \ref{rem:naive_tailoring} may seem technical. But, here all the variables are unreadable, and in the version of Remark~\ref{rem:naive_tailoring} version all variables are readable. 
   If all variables of a tailored game are readable, for it to have a perfect  $Z$-aligned permutation strategy is the same as having a perfect classical strategy --- which in turn is the same as for the linear system $A\vec x=\vec b$ to have a solution. 
   But, when all the variables are unreadable, there could be a perfect $Z$-aligned permutation strategy without $A\vec x=\vec b$ having a solution. 
   For example, the perfect strategy for the Mermin--Peres magic square game (see for example~\cite{aravind2002simple}) can be derived from a permutation strategy (see our companion paper \cite[Example 2.19]{Tailored_MIPRE}), and since there are no readable variables in this case,  it is automatically $Z$-aligned.
   Furthermore, the magic square game has \textbf{no} perfect classical strategy. This shows that if we care about perfect $Z$-aligned permutation strategies, the way we tailor a game \textbf{matters}.
\end{example}

\section{Main Theorem II: Tests associated with tailored games}
\label{sec:test_associated_with_game}
The goal of this section is to define a mapping from tailored non-local games to subgroup tests such that: 
\begin{enumerate}
    \item \textbf{Completeness}: Perfect $Z$-aligned permutation strategies that commute along edges for the tailored game are translated to perfect finitely described strategies of the subgroup test.
    \item \textbf{Soundness}: Almost perfect finitely described strategies for the subgroup test are close to almost perfect $Z$-aligned permutation strategies for the tailored non-local game. 
\end{enumerate} 

Throughout this section, $\cG$ is a tailored non-local game with underlying graph $G=(V,E)$,  length functions $\ell^\frR,\ell^\frL\colon V\to \NN$ with $\ell=\ell^\frR+\ell^\frL$ and $\Lambda =\displaystyle{\max_{x\in V}}(\ell(x))$, formal sets of generators $S_x=S_x^\frR\cup S_x^\frL$ for every vertex $x\in V$, where $S_x^\frR=\{\sX^{x,\frR,i}\mid 1\leq i\leq \ell^\frR(x)\}$ and $S_x^\frL=\{\sX^{x,\frL,i}\mid 1\leq i\leq \ell^\frL(x)\}$, a readability function $\frR\colon S\to \{0,1\}$ where $S=\bigcup_{x\in V} S_x$, a distribution $\mu$ over  edges $E$, and decision functions $D_{xy}\colon \FF_2^{S_{xy}}\to \FF_2$ for every edge $xy\in E$ with a controlled linear constraints map $L_{xy}\colon \FF_2^{S_x^\frR\cup S_y^\frR}\to \FF_2^{\FF_2^{S_{xy}\cup \{\sJ\}}}$.

\begin{defn}\label{defn:associated_test}
    The \emph{synchronous subgroup test} $\tilde \cT=\tilde \cT(\cG)$ \emph{associated} with the tailored non-local game $\cG$ is defined as follows: The set of generators $\tilde S$ in the test $\tilde \cT$ is $S\cup \{\sJ\}$. The set  $\tilde Q$ which parametrizes the challenges of $\tilde \cT$ is the edge set $E$ of $\cG$, and the distribution $\tilde \mu$ over challenges is $\mu$ of $\cG$. The decision function $\tilde D_{xy}$ works as follows (and the subsets $\tilde K_{xy}$ can be derived from this description): Given a subgroup $H$ of $\cF(\tilde S)$, 
    \begin{enumerate}[label=\textcolor{black}{Check \arabic*.}, ref=Check \arabic*]
        \item \label{clause:Check_1_in_associated_test} It verifies that
    \[
\sJ\notin H\quad;\quad \sJ^2\in H\quad;\quad \forall \sX\in S_{xy}\ \colon \ \ [\sJ,\sX]\in H.
    \]
    \item \label{clause:Check_2_in_associated_test} It verifies that
    \[
    \begin{split}
        \forall \sX\in S_{xy}\ &\colon \ \ \sX^2\in H,\\
        \forall \sX,\sX'\in S_x\ &\colon \ \ [\sX,\sX']\in H,\\
        \forall \sY,\sY'\in S_y\ &\colon \ \ [\sY,\sY']\in H.
    \end{split}
    \]
        \item \label{clause:Check_3_in_associated_test} 
    It verifies that 
    \[
\forall \sX\in S_{x}^\frR\cup S_y^\frR \ \colon \ \ \sX\in H\quad\textrm{or}\quad \sJ\sX\in H.
    \]
        \item \label{clause:Check_4_in_associated_test} If $H$ passed all the previous checks, then the following function $\tilde\gamma^\frR\colon S_{x}^\frR\cup S_y^\frR\to\FF_2$ can be calculated: $\tilde\gamma^\frR(\sX)=0$ if $\sX\in H$ and $\tilde\gamma^\frR(\sX)=1$ if $\sJ\sX\in H$. After recovering this $\tilde\gamma^\frR$, the subset 
        \begin{equation}\label{eq:defn_tilde_L_xy}
           \tilde L_{xy}(H)=L_{xy}(\tilde\gamma^\frR) 
        \end{equation}
        can  be calculated.
        So, the decision function verifies that 
        \[
\forall \alpha\in \tilde L_{xy}(H)\ \colon \ \ \sJ^{\alpha(\sJ)}\cdot\prod_{\sX\in S_x}\sX^{\alpha(\sX)}\cdot\prod_{\sY\in S_y}\sY^{\alpha(\sY)}\in H,
        \]
          where the product is according to some pre-fixed ordering of $S_x$ and $S_y$.
    \end{enumerate}
  
    If any of the above checks did not go through, $\tilde D_{xy}$ rejects $H$, and otherwise it accepts it.
\end{defn}

\begin{rem}
    Let us motivate the checks that the decision function $\tilde D_{xy}$ is applying on $H$. Let $\sigma\colon \tilde S\to \Sym(X)$ be a finitely described strategy. If $\sigma$ always passes \labelcref{clause:Check_1_in_associated_test}, it means that $\sJ$ acts as a fixed point free central involution on $X$, which is part of the requirement of a permutation strategy  for $\cG$ (see Definition \ref{defn:permutation_strategy}). If $\sigma$ always passes \labelcref{clause:Check_2_in_associated_test}, then the images of all other variables in $\tilde S$ are involutions, and for any fixed $x\in V$, $\sigma(S_x)$ commutes, which are again requirements for $\sigma$ to be a permutation strategy of $\cG$. If $\sigma$ always passes \labelcref{clause:Check_3_in_associated_test}, then the readable variables always act as either $\Id$ or $\sJ$, which implies $\sigma$ is a $Z$-aligned permutation strategy for $\cG$ (see Definition \ref{defn:Z-aligned_strategy}). So, the goal of the first three Checks is to force any finitely described strategy of $\tilde \cT$ to be a permutation $Z$-aligned strategy for $\cG$. Lastly, \labelcref{clause:Check_4_in_associated_test} verifies ``the same''  linear relations as $\cG$, which means that the winning probability of $\sigma$  against  $\tilde \cT$ and $\cG$ are similar.
\end{rem}

We are ready to formulate our main theorem:
\begin{thm}\label{thm:values_of_associated_tests}
    Let $\cG$ be a tailored non-local game with $\displaystyle{\max_{x\in V}} (\ell(x))=\Lambda>0$, and let $\tilde \cT=\tilde \cT(\cG)$ be its associated synchronous subgroup test, as in Definition \ref{defn:associated_test}. 
    \begin{enumerate}[label=\textcolor{black}{(\arabic*)}, ref= (\arabic*)]
        \item \label{clause:1_in_values_associated_test} \textbf{Completeness}: If $\cG$ has a perfect $Z$-aligned permutation strategy that commutes along edges, then $\tilde\cT$ has a perfect finitely described strategy.
        \item \label{clause:2_in_values_associated_test} \textbf{Soundness}: If $\tilde\cT$ has a finitely described strategy $\sigma$ with $\val(\tilde\cT,\sigma)\geq 1-\eps$, then there exists a quantum strategy $\rho$ such that $\val(\cG,\rho)\geq 1-C\Lambda ^4\cdot 2^{6\Lambda}\eps$, where $C>0$ is a universal constant.
    \end{enumerate}
\end{thm}
The next section is devoted to the proof of Theorem \ref{thm:values_of_associated_tests}.  

\begin{thm}[$\TMIP^*=\RE$, see \cite{Tailored_MIPRE}]\label{thm:tailored_MIP*=RE}
    There exists a polynomial time algorithm that takes as input (the encoding of) a Turing machine $M$ and outputs (the encoding of) a \textbf{tailored} non-local game $\cG_M$ such that:
\begin{enumerate}[label=\textcolor{black}{(\arabic*)}, ref= (\arabic*)]
    \item \label{clause:1_tailored_MIP*} Sampling $xy\in E$ according to  $\mu$ and evaluating  $D_{xy}(\cdot)$ from the encoding of the  game $\cG_M$  can be done in time ${\rm poly}(|M|)$, where $|M|$ is the bit-length of the encoding of $M$.
    \item \label{clause:2_tailored_MIP*} If $M$ halts, then $\cG_M$ has a perfect  \textbf{$Z$-aligned permutation} strategy that commutes along edges.
    \item  \label{clause:3_tailored_MIP*} If $M$ never halts, then $\val^*(\cG_M)\leq \nicefrac{1}{2}$.
\end{enumerate}
\end{thm}
Theorem \ref{thm:tailored_MIP*=RE} is proved in \cite{Tailored_MIPRE}.
\begin{cor}\label{cor:AL_is_false}
   The Aldous--Lyons Conjecture \ref{(Aldous-Lyons)} has a negative solution.
\end{cor}
\begin{proof}
The idea is as follows. Given a Turing machine $M$, we can define  the tailored non-local game  $\cG_M$ calculated in Theorem \ref{thm:tailored_MIP*=RE}, and then calculate the associated subgroup test $\tilde\cT(\cG_M)$  from Definition  \ref{defn:associated_test}. In the next paragraph, the following two facts are proved:
\begin{itemize}
    \item If $M$ halts, then $\tilde\cT(\cG_M)$ has sofic value $1$.
    \item On the other hand, if $M$ does not halt, then the sofic value of $\tilde\cT(\cG_M)$ is smaller than $1-\lambda(M)$, where $\lambda(M)=2^{-c|M|^c}$ for some universal constant $c>0$ independent of $M$. 
\end{itemize} 
 Recall that by Theorem \ref{Main_Thm}, there is a computable sequence approaching $\val_{\erg}(\tilde \cT(\cG_M))$ from above and another computable sequence approaching $\val_{\sof}(\tilde \cT(\cG_M))$ from below. If the Aldous--Lyons Conjecture \ref{(Aldous-Lyons)} has a positive solution, then these values are the same, and as Corollary \ref{cor:Aldous_Lyons_implies_SOFVAL_decidable} states, one can approximate $\val_{\sof}(\tilde \cT(\cG_M))$ up to any predetermined additive error $\theta$. As $\lambda(M)$ can be computed directly from $|M|$, one can choose $\theta=\frac{\lambda(M)}{2}$, and thus deduce whether $\val_{\sof}(\tilde \cT(\cG_M))<1$ or $\val_{\sof}(\tilde \cT(\cG_M))>1-\lambda(M)$. This in turn allows one to decide whether $M$ halts or not. Since the Halting Problem is undecidable, this is a contradiction, which implies that Conjecture  \ref{(Aldous-Lyons)} must have a negative solution.
\\

Let us prove the two bullets above. Since the running time of the verification procedure of $\cG_M$ bounds $\Lambda=\displaystyle{\max_{x\in V}}(\ell(x))$, we can use \labelcref{clause:1_tailored_MIP*} of Theorem \ref{thm:tailored_MIP*=RE} to deduce that $\Lambda=\textrm{poly}(|M|)$. 
    If $M$ halts, then by \labelcref{clause:2_tailored_MIP*} of Theorem \ref{thm:tailored_MIP*=RE} and \labelcref{clause:1_in_values_associated_test} of Theorem \ref{thm:values_of_associated_tests}, we have $\val_{\sof}(\tilde \cT(\cG_M))=1$.
    If $M$ does not halt, then by \labelcref{clause:3_tailored_MIP*} of Theorem \ref{thm:tailored_MIP*=RE}, we have $\val^*(\cG_M)\leq  \frac{1}{2}$. Hence, by \labelcref{clause:2_in_values_associated_test} of Theorem \ref{thm:values_of_associated_tests},
    we have
    \[
\frac{1}{2}\geq 1-C\Lambda^4\cdot 2^{6\Lambda}(1-\val_{\sof}(\tilde \cT(\cG_M))).
    \]
    Rearranging this inequality, we get 
    \[
\val_{\sof}(\tilde \cT(\cG_M))\leq 1-\frac{1}{2C\Lambda^4\cdot 2^{6\Lambda}}.
    \]
    The quantity $\frac{1}{2C\Lambda^4\cdot 2^{6\Lambda}}$ can be bounded from below by $\lambda(M)=2^{-c|M|^c}$, where $c>0$ is a universal constant that depends on the constants implicit in the notation  $\Lambda=\textrm{poly}(|M|)$ (guaranteed by Theorem \ref{thm:tailored_MIP*=RE}), as well as the constant $C>0$ appearing in Theorem \ref{thm:values_of_associated_tests}. This finishes the proof.
\end{proof}

\section{Proving Theorem \ref{thm:values_of_associated_tests}}\label{sec:proof_of_value_preserving_transformation}

\subsection{Proving the completeness in Theorem \ref{thm:values_of_associated_tests}}\label{sec:proof_of_completeness}

This is the easy direction, we have done most of the work towards it, and here we use the notation that we have established before. Recall that $\tilde S=S\cup \{\sJ\}$. Let $\sigma\colon S\cup \{\sJ\}\to \Sym(2n)$ be a perfect $Z$-aligned permutation strategy that commutes along edges for the tailored non-local game $\cG$. Further recall that the finitely described IRS $\Phi(\sigma)$ samples a subgroup $H$ of $\cF$ as follows --- it takes a uniform $\star\in [2n]$ and outputs $H=\Stab(\sigma,\star)$. By Definition \ref{defn:permutation_strategy},
\[
d_H(\sigma(\sJ),\Id)=1\quad ; \quad \sigma(\sJ)^2=\Id\quad;\quad \forall \sX\in S\ \colon \ \ \sigma(\sJ)\sigma(\sX)=\sigma(\sX)\sigma(\sJ).
\]
Therefore, for every $\star\in [2n]$, 
\[
\sJ\notin \Stab(\sigma,\star)\quad ;\quad \sJ^2\in \Stab(\sigma,\star)\quad;\quad \forall \sX\in S\ \colon \ \ [\sX,\sJ]\in \Stab(\sigma,\star),
\]
and when $\Phi(\sigma)$ runs against $\tilde \cT$ it  always passes \labelcref{clause:Check_1_in_associated_test}. Also by  Definition \ref{defn:permutation_strategy},
\[
\forall \sX\in S\ \colon \sigma(\sX)^2=\Id\quad;\quad \forall x\in V,\ \forall\sX,\sX'\in S_x\ \colon \ \ \sigma(\sX)\sigma(\sX')=\sigma(\sX')\sigma(\sX).
\]
Therefore, for every $\star\in [2n]$,
\[
\forall \sX\in S\ \colon \sX^2\in \Stab(\sigma,\star)\quad;\quad \forall x\in V,\ \forall\sX,\sX'\in S_x\ \colon \ \ [\sX,\sX']\in \Stab(\sigma,\star),
\]
and when $\Phi(\sigma)$ runs against $\tilde \cT$ it always passes \labelcref{clause:Check_2_in_associated_test}.
By Definition \ref{defn:Z-aligned_strategy}, for every $\star\in [2n]$ we have 
\[
\forall \sX\in \frR^{-1}(1)\ \colon \ \ \sigma(\sX).\star=\star\quad\textrm{or}\quad \sigma(\sX).\star=\sigma(\sJ).\star.
\]
Namely, for every $\star\in [2n]$,
\[
\forall \sX\in \frR^{-1}(1)\ \colon \ \ \sX\in \Stab(\sigma,\star)\quad\textrm{or}\quad \sJ\sX\in \Stab(\sigma,\star),
\]
and when $\Phi(\sigma)$ runs against $\tilde \cT$ it always passes \labelcref{clause:Check_3_in_associated_test}.

We are left to show that $\Phi(\sigma)$ always passes \labelcref{clause:Check_4_in_associated_test}.
To that end, we need a couple of claims.
    The first is a well known and commonly used observation in quantum information theory. For the convenience of the reader, we provide a proof.
\begin{claim}\label{claim:perfect_3Lin_commuting_observables}
    Let $S$ be a finite set, $\rho\colon S\to U(n)$ such that $\rho(S)$ commutes\footnote{This claim can be generalized so that a weaker assumption is used. See Section 3 of \cite{Tailored_MIPRE}.} and all the images are involutions.  Then, given a fixed $\alpha\colon S\to \FF_2$, we have $\sum_{\sX\in S}\alpha(\sX)\gamma(\sX)=0$  with probability $1$ when  $\gamma\sim \rho$ if and only if $\prod_{\sX\in S}\rho(\sX)^{\alpha(\sX)}=\Id$. Namely,
\[
\Pro_{\gamma\sim \rho}[\langle \alpha,\gamma\rangle=0]=1 \iff \prod_{\sX\in S}\rho(\sX)^{\alpha(\sX)}=\Id.
\]
\end{claim}
\begin{proof}
Assume without loss of generality that $\alpha$ is fully supported. Otherwise, we can focus on $\rho|_{\textrm{Supp}(\alpha)}$ and proceed in the same manner. 
    As before, we can write $\rho(\sX)=\sP_0^\sX-\sP_1^\sX$, where $\sP^\sX_0$ is the projection on the $(+1)$-eigenspace of $\rho(\sX)$, and $\sP_1^\sX$ is the projection on its $(-1)$-eigenspace.  By Definition \ref{defn:strategy_for_a_game}, 
\[
\begin{split}
    \Pro_{\gamma\sim\rho}\left[\sum_{\sX\in S}\gamma(\sX)=0\right]&=\sum_{\substack{\gamma\colon S\to \FF_2\\ \sum \gamma(\sX)=0}} \tau\left(\prod_{\sX\in S} \sP_{\gamma(\sX)}^\sX\right).
\end{split}
\]
Note that $\sP^\sX_0+\sP^\sX_1=\Id$, and thus 
$\displaystyle{\sum_{\gamma\colon S\to \FF_2}} \tau\left(\prod_{\sX\in S} \sP_{\gamma(\sX)}^\sX\right)$ is always $1$. 
Therefore, 
\[
\begin{split}
\tau\left(\prod_{\sX\in S}\rho(\sX)\right)&=\tau\left(\prod_{\sX\in S}(\sP^\sX_0-\sP^\sX_1)\right)\\    
&=\sum_{\substack{\gamma\colon S\to \FF_2\\  \sum \gamma(\sX)=0} }\tau\left(\prod\sP_{\gamma(\sX)}^\sX\right)\ -\sum_{\substack{\gamma\colon S\to \FF_2\\  \sum \gamma(\sX)=1} }\tau\left(\prod\sP_{\gamma(\sX)}^\sX\right)\\
&=2\cdot\sum_{\substack{\gamma\colon S\to \FF_2\\  \sum \gamma(\sX)=0} }\tau\left(\prod\sP_{\gamma(\sX)}^\sX\right)-1\\
&=2\cdot \Pro_{\gamma\sim\rho}\left[\sum_{\sX\in S}\gamma(\sX)=0\right]-1.
\end{split}
\]
Thus,  $\prod \rho(\sX)=\Id$ if and only if $\tau\left(\prod\rho(\sX)\right)=1$ if and only if  $\Pro_{\gamma\sim\rho}\left[\sum_{\sX\in S}\gamma(\sX)=0\right]=1$, which finishes the proof.
\end{proof}

\begin{claim}\label{claim:L_depends_on_vertex_only}
    Let $\sigma\colon S\cup \{\sJ\}\to \Sym(2n)$ be a $Z$-aligned permutation strategy for a tailored game $\cG$, and let $\rho$ be the quantum strategy associated with $\sigma$ (as in Definition \ref{defn:quantum_strat_associated_with_perm_strat}). 
    Then, the restriction of  $\gamma\sim \rho$ to the readable variables in $S_{xy}$ (which we denoted by $\gamma^\frR$) depends only on the vertex $\star\in [2n]$ sampled in the beginning of the procedure outlined in Claim \ref{claim:procedural_sampling_perm_strat}.
    Furthermore, if we denote the sampled $L_{xy}(\gamma)$ by $L_{xy}(\star)$, emphasizing its dependence on the sampled vertex, then $\tilde L_{xy}(\Stab(\sigma,\star))=L_{xy}(\star)$, where $\tilde L_{xy}$ was defined in \eqref{eq:defn_tilde_L_xy}.
\end{claim}
\begin{proof}
    Let $\star\in [2n]$ be some fixed vertex, $O_x$ be the $\sigma_x$ orbit of $\star$, and $\sX\in S_x$ be a readable variable. Since $\sigma$ is $Z$-aligned, either $\sigma(\sX).\star=\star$ or $\sigma(\sX).\star=\sigma(\sJ).\star$. Now, as discussed in Section \ref{sec:proc_samp_perm_strat}, $\sigma$ acting on  $O_x$ is equivalent to an action of $\FF_2^k$ on itself. Furthermore, checking whether  $\vec v\in B^-_x$ that is supported on $O_x$ is in the $(+1)$ or $(-1)$ eigenspace of $\sigma(\sX)$ depends only on whether $\vec v$ flips along an $\sX$-labeled edge in $\Sch(\sigma,S)$ for \textbf{some} vertex in $O_x$. Since $\vec v\in B^{-}_x$, it must flip along $\sJ$-labeled edges. So, if $\sigma(\sX).\star=\star$ then $\sigma(\sX)\vec v=\vec v$ and $\gamma(\sX)=0$, and if $\sigma(\sX).\star=\sigma(\sJ).\star$ then $\sigma(\sX)\vec v=-\vec v$ and $\gamma(\sX)=1$. 
    Since this was \textbf{independent} of the specific $\vec v$ that we considered, the first part of the claim is deduced.  
   This proves that $\tilde \gamma^\frR(\sX)=\gamma^\frR(\sX)$, where $\tilde \gamma^\frR(\sX)$ was defined in \labelcref{clause:Check_4_in_associated_test} of Definition \ref{defn:associated_test}, and therefore $\tilde L_{xy}(\Stab(\sigma,\star))=L_{xy}(\tilde \gamma^\frR)=L_{xy}(\gamma^\frR)=L_{xy}(\star)$ which is the second part of the claim.
\end{proof}

We can now come back to our proof. Let us, in the spirit of previous notations, denote  $\sigma_{xy}=\sigma|_{S_{xy}\cup \{\sJ\}}$. Let $\star\in [2n]$, and let $O_{xy}$ be the orbit of $\star$ under $\sigma_{xy}$. Since $\sigma$ commutes along edges (see Definition \ref{defn:strategy_for_a_game}), $\sigma_{xy}$ induces an action of $\FF_2^{S_{xy}\cup \{\sJ\}}$. Since $\Img(\sigma_{xy})$ is commuting, the action of readable variables in $O_{xy}$ is constant, namely if a readable variable acts as $\Id$ on some $\star\in O_{xy}$, it will act that way on all $\star\in O_{xy}$, and similarly for acting as $\sJ$. Furthermore, by Claim \ref{claim:L_depends_on_vertex_only}, the linear constraints sampled in the associated game $\tilde L_{xy}$ at a specific vertex $\star\in O_{xy}$ are the same as the ones sampled by the tailored non-local game.
Joining these observations, there is a set of constraints $L$ such that for every $\star\in O_{xy}$ we have  $\tilde L_{xy}(\Stab(\sigma,\star))=L_{xy}(\star)=L$.  So, we can focus on the orbit $O_{xy}$, and study the corner of $\rho$, which was the quantum strategy induced by $\sigma$, with respect to the projection $\sQ^{O_{xy}}\colon \mathbb{C}^{2n}\to \mathbb{C}^{O_{xy}}$. 
Since $\rho$ is a perfect strategy, for every $\gamma$ that was sampled (according to Claim \ref{claim:procedural_sampling_perm_strat}) by first choosing $\star\in O_{xy}$, and for every $\alpha\in L$, we have $\sum\alpha(\sX)\gamma(\sX)=0$. 
By Claim \ref{claim:perfect_3Lin_commuting_observables},  this means that $\prod\rho(\sX)^{\alpha(\sX)}$ is the identity when restricted to functions supported on $O_{xy}$ that flip along $\sJ$-labeled edges. By Claim \ref{claim:rho_id_implies_sigma_id}, we can deduce that $\prod\sigma(\sX)^{\alpha(\sX)}$ acts as the identity on $O_{xy}$. Since this was true for every $\alpha$ in $L$, and $L=\tilde L_{xy}(\Stab(\sigma,\star))$ for every $\star\in O_{xy}$, we can conclude that $\sigma$ always passes \labelcref{clause:Check_4_in_associated_test} when run against $\tilde \cT$. All in all, $\sigma$ is a perfect strategy for $\tilde \cT$, proving the completeness clause of Theorem \ref{thm:values_of_associated_tests}.

\subsection{Proving the soundness in Theorem \ref{thm:values_of_associated_tests}}

The plan is as follows:
\begin{itemize}
    \item Recall that $\tilde S= S\cup \{\sJ\}$.  Given a finitely described strategy $\sigma\colon \tilde S\to \Sym(n)$ for $\tilde \cT$, with $\val(\tilde \cT,\sigma)\geq 1-\eps$, we  perturb it bit by bit until it passes with probability $1$ \labelcref{clause:Check_1_in_associated_test}, \labelcref{clause:Check_2_in_associated_test} and \labelcref{clause:Check_3_in_associated_test} from the definition of $\tilde \cT$ (while potentially worsening the probability of passing  \labelcref{clause:Check_4_in_associated_test}). 
    \item  Then, we can bound from below the value of this perturbed strategy using the mechanisms of Section \ref{sec:Robustness}. 
    \item  Since the resulting perturbed strategy satisfies the first three checks of $\tilde \cT$, it is a $Z$-aligned permutation strategy for $\cG$. For such strategies, we prove that their value against $\tilde \cT$ induces a lower-bound on their value against $\cG$. This will conclude the proof.
\end{itemize}   

The following three propositions are the manifestation of the above plan. We abuse notations and denote by $\mu$ the marginal distribution on vertices of $\mathcal{G}$, namely
\[
\forall x\in V\ \colon \ \ \mu(x)=\frac{1}{2}\sum_{xy\in E}\mu(xy).
\]
Also, recall the notion of the significance function $\frS_{\cT}$  associated with a test $\cT$, defined in \eqref{equation:weight_function_of_a_test}, and the notion of the edit distance $d^\frS_{\rm edit}$ \eqref{eq:edit_dist_of_IRSs} between finitely described IRSs. Finally, recall that  $\Lambda=\displaystyle{\max_{x\in V}}(\ell(x))$, and assume $\Lambda>0$.
\begin{prop}\label{prop:perturbing_to_satisfy_1to3}
   Let $\sigma\colon \tilde S\to \Sym(n)$ be a finitely described strategy  for $\tilde \cT$, with $\val(\tilde \cT,\sigma)\geq 1-\eps$. Denote by $\eps_x$ the losing probability in $\tilde \cT$ given that an edge containing $x\in V$ was sampled as the challenge --- note that $\Ex_{x\sim \mu}[\eps_x]= \eps$. 
   Then, there is a strategy $\sigma'\colon \tilde S\to \Sym(2m)$, where $m= \lceil\frac{n}{2}\rceil$, that passes with probability $1$ \labelcref{clause:Check_1_in_associated_test}, \labelcref{clause:Check_2_in_associated_test} and \labelcref{clause:Check_3_in_associated_test} in the definition of the associated test $\tilde \cT$, and which is \emph{close} to $\sigma$ in the following sense:
   \[
   \begin{split}
       &d_H(\sigma(\sJ),\sigma'(\sJ))\leq  C_0\eps\\  
       \forall x\in V,\   \sX\in S_x\ \colon \ \ &d_H(\sigma(\sX),\sigma'(\sX))\leq  C_0\cdot 2^{2\Lambda}\cdot \Lambda^3(\eps+\eps_x).
   \end{split}
\]
    for some universal constant $C_0>0$.
\end{prop}

\begin{prop}\label{prop:associated_quantum_strat_has_value_at_least_as_original}
     Every  finitely described strategy 
 $\sigma\colon \tilde S\to \Sym(2n)$   for $\tilde \cT$ that passes with probability $1$ \labelcref{clause:Check_1_in_associated_test}, \labelcref{clause:Check_2_in_associated_test} and \labelcref{clause:Check_3_in_associated_test} in the definition of the associated test $\tilde \cT$ is a $Z$-aligned permutation strategy for $\cG$ and vice versa. Moreover, there is a universal constant $C_1>0$ such that  if  $\val(\tilde \cT,\sigma)\geq 1-\eps$, then $\val(\cG,\rho)\geq 1-C_1\cdot 2^{2\Lambda}\eps$ where $\rho$ is the quantum strategy associated with $\sigma$ (Definition \ref{defn:quantum_strat_associated_with_perm_strat}). 
\end{prop}
\begin{prop}\label{prop:significance_of_associated_test}
    Let $\tilde \cT$ be the test associated with the tailored non-local game $\cG$.  Then, there is a universal constant $C_2>0$, such that the significance $\frS_{\tilde \tau}$ associated with the test $\tilde \cT$  satisfies
    \begin{align}
        \frS_{\tilde \cT}(\sJ)&\leq C_2\cdot 2^{2\Lambda},\\
        \forall x\in V,\ \sX\in S_x\ \colon \ \ \frS_{\tilde \cT}(\sX)&\leq \mu(x)\cdot C_2\cdot 2^{2\Lambda}.
    \end{align}
\end{prop}
\begin{rem}
    We did not try to optimise the parameters in the three preceding propositions. Actually, we decided to choose a simpler version of the associated synchronous test (Definition \ref{defn:associated_test}) that provides an $\textrm{exp}(\Lambda)$ deterioration in the value, as seen in Proposition \ref{prop:perturbing_to_satisfy_1to3}. We have variations of tailored games and the associated synchronous test which allow for only a $\textrm{poly}(\Lambda)$ deterioration in the value, but we could not find any reduction that is independent of $\Lambda$.
\end{rem}

Now, we can deduce the soundness clause of Theorem \ref{thm:values_of_associated_tests} as a corollary of the preceding propositions.

\begin{proof}[Proof of \labelcref{clause:2_in_values_associated_test} in Theorem \ref{thm:values_of_associated_tests}]
    Let $\sigma\colon \tilde S\to \Sym(n)$ be the given finitely described strategy with value $\ge 1-\eps$. By Proposition \ref{prop:perturbing_to_satisfy_1to3} and Proposition  \ref{prop:significance_of_associated_test}, there is a strategy $\sigma'\colon \tilde S\to \Sym(2\lceil\frac{n}{2}\rceil)$ passing \labelcref{clause:Check_1_in_associated_test}, \labelcref{clause:Check_2_in_associated_test} and \labelcref{clause:Check_3_in_associated_test} with probability $1$, and such that the significance weighted distance between $\sigma$ and $\sigma'$ is
    \[
    \begin{split}
         d^{\frS_{\tilde \cT}}(\sigma,\sigma')&=\sum_{s\in \tilde S}\frS_{\tilde \cT}(s)\cdot d_H(\sigma(s),\sigma'(s))\\
         &=\frS_{\tilde \cT}(\sJ)d_H(\sigma(\sJ),\sigma'(\sJ))+\sum_{x\in V}\sum_{\sX\in S_x}\frS_{\tilde \cT}(\sX)d_H(\sigma(\sX),\sigma'(\sX))\\
         &\leq C_2\cdot 2^{2\Lambda}d_H(\sigma(\sJ),\sigma'(\sJ))+\sum_{x\in V}\sum_{\sX\in S_x}\mu(x)\cdot C_2\cdot 2^{2\Lambda} d_H(\sigma(\sX),\sigma'(\sX))\\
         &\leq C_0C_2\cdot2^{2\Lambda} \eps+\sum_{x\in V}|S_x|\cdot\mu(x)\cdot C_2\cdot 2^{2\Lambda} \cdot C_0\cdot 2^{2\Lambda}\cdot \Lambda^3(\eps+\eps_x)\\
         &\leq C_0C_2\cdot2^{2\Lambda} \eps+ C_0C_2\Lambda ^4 \cdot 2^{4\Lambda}\underbrace{\left(\sum_{x\in V}\mu(x)(\eps+\eps_x)\right)}_{\Ex_{x\sim \mu}[\eps+\eps_x]=2\eps}\\
         &\leq 3C_0C_2\Lambda ^4\cdot 2^{4\Lambda}\eps.
    \end{split}
    \]
    Therefore, by Claim \ref{claim:test_weighted_edit_distance_implies_val_close}, $$\val(\tilde \cT,\sigma')\geq 1-\eps-3C_0C_2\Lambda ^4\cdot 2^{4\Lambda}\eps\geq 1-4C_0C_2\Lambda ^4\cdot 2^{4\Lambda}\eps.$$ By applying Proposition \ref{prop:associated_quantum_strat_has_value_at_least_as_original} on $\sigma'$,  its associated quantum strategy $\rho'$ satisfies  
    \[
    \val(\cG,\rho')\geq 1-4C_0C_1C_2\Lambda ^4\cdot 2^{6\Lambda}\eps.
    \]
By    choosing $C=4C_0C_1C_2$, we deduce the claim.
\end{proof}

We are left to prove Propositions \ref{prop:perturbing_to_satisfy_1to3}, \ref{prop:associated_quantum_strat_has_value_at_least_as_original} and 
 \ref{prop:significance_of_associated_test}. This is done in a reverse manner.

\begin{proof}[Proof of Proposition \ref{prop:significance_of_associated_test}]
   Note that $\sJ$ participates in most of  the checks in $\tilde \cT$ regardless of the specific challenge. 
   In \labelcref{clause:Check_1_in_associated_test}, it appears $3+2|S_{xy}|$ times, and since $|S_{xy}|\leq 2\Lambda$, this is bounded by $3+4\Lambda$. 
   In \labelcref{clause:Check_2_in_associated_test} it is not used. 
   In \labelcref{clause:Check_3_in_associated_test} it is used at most $|S_{xy}|$ times, which is bounded from above  by $2\Lambda$. 
   Lastly, the number of times $\sJ$  appears in $L_{xy}(\gamma)$ is bounded by the number of subsets of $S_{xy}\cup \{\sJ\}$ that contain it, which is $2^{|S_{xy}|}\leq 2^{2\Lambda}$.   All in all, the significance of $\sJ$ is at most $(3+4\Lambda)+2\Lambda+2^{2\Lambda}\leq 4\cdot 2^{2\Lambda}$. 
    
    Now, a variable $\sX$ in $ S_x$ appears only if we sampled an edge in $E$ with $x$ as one of its endpoints. If this happened, it will appear twice in \labelcref{clause:Check_1_in_associated_test}, $2+2|S_x|$ times in \labelcref{clause:Check_2_in_associated_test}, at most twice in \labelcref{clause:Check_3_in_associated_test}, and at most $2^{|S_{xy}\cup \{\sJ\}|-1}$ times in \labelcref{clause:Check_4_in_associated_test}. Again, $|S_x|\leq \Lambda$ and $|S_{xy}|\leq 2\Lambda$, and thus $\sX$ appears at most $2+2+2\Lambda+2+2^{2\Lambda}\leq 3\cdot 2^{2\Lambda}$ times when an edge containing $x$ is sampled. This means that the significance of $\sX$ is at most $\mu(x)\cdot 3\cdot 2^{2\Lambda}$. Choosing $C_2=4$ completes the proof.
\end{proof}

\begin{proof}[{Proof of Proposition \ref{prop:associated_quantum_strat_has_value_at_least_as_original}}]
Let $\sigma\colon \tilde S\to \Sym(2n)$ be a finitely described strategy that passes \labelcref{clause:Check_1_in_associated_test}, \labelcref{clause:Check_2_in_associated_test} and \labelcref{clause:Check_3_in_associated_test} of $\tilde \cT$ with probability $1$. Then, 
from perfectly passing \labelcref{clause:Check_1_in_associated_test} and  \labelcref{clause:Check_2_in_associated_test}, $\sigma$ is a permutation strategy, and by perfectly passing \labelcref{clause:Check_3_in_associated_test}  it is $Z$-aligned. The other direction was proved in the completeness analysis in Section \ref{sec:proof_of_completeness}. We are going to prove that 
\begin{equation}\label{eq:sufficient_condition_Prop_8.4}
    \val(\cG,\rho)\geq 1- 2\cdot 2^{2\Lambda}(1-\val(\tilde \cT,\sigma)),
\end{equation} and thus $C_1=2$ satisfies the claim.

We start by giving expressions for the success probability of $\sigma$ in $\tilde\cT$, and of the associated $\rho$ in $\mathcal{G}$ respectively. 
Fix an edge $xy\in E$, and denote by $\sigma_x=\sigma|_{S_x\cup \{\sJ\}}$ and $\sigma_y=\sigma|_{S_y\cup \{\sJ\}}$ the appropriate restrictions of the permutation strategy $\sigma$.  The intersection of orbits $O_x\cap O_y$, where $O_x$ is an orbit of $\sigma_x$ and $O_y$ is an orbit of $\sigma_y$, partitions $[2n]$ into disjoint sets. Fixing a pair of orbits $O_x,O_y$ such that $O_x\cap O_y\neq \emptyset$,  by Claim \ref{claim:L_depends_on_vertex_only}, there is a set $L_{xy}^{O_x,O_y}$  of subsets of  $S_{xy}\cup \{\sJ\}$ such that 
\[
\forall \star\in O_x\cap O_y\ \colon \ \ L^{O_x,O_y}_{xy}=L_{xy}(\star)=\tilde L_{xy}(\Stab(\sigma,\star)).
\] 
For $\alpha \in \mathbb{F}_2^{S_{xy}\cup\{\sJ\}}$, define
\begin{equation}\label{eq:def_sX^alpha,sY^alpha}
\sX^\alpha=\prod_{\sX\in S_x\cup \{\sJ\}}\sX^{\alpha(\sX)}\quad \textrm{ and  }\quad \sY^\alpha=\prod_{\sY\in S_y}\sY^{\alpha(\sY)}.
\end{equation}
Let $\tilde \Upsilon^{O_x,O_y}(\alpha)$ be the set of all vertices $\star\in O_x\cap O_y$ such that $\sigma(\sX^{\alpha}).\star\neq \sigma(\sY^\alpha).\star$, namely all vertices in the intersection of orbits that do not satisfy the subgroup test constraint induced by $\alpha$.  Let 
\begin{equation}\label{eq:def_upsilon_xy_local}
    \tilde\Upsilon^{O_x,O_y}_{xy} = \bigcup_{\alpha \in L^{O_x,O_y}_{xy}} \tilde\Upsilon^{O_x,O_y}(\alpha),
\end{equation}
and let 
\[
\tilde\Upsilon_{xy}=\bigcup_{O_x,O_y}\tilde\Upsilon^{O_x,O_y}_{xy},
\]
where the union is over all orbits of $\sigma_x$ and $\sigma_y$ that have a non-empty intersection.
Then $\tilde\Upsilon_{xy}$ is the set of all  $\star\in [2n]$ such that $\tilde D_{xy}( \Stab(\sigma,\star))=0$, hence
\begin{equation}\label{eq:tg-2}
 \val(\tilde\cT,\sigma)=1- \displaystyle{\Ex_{xy\sim \mu}}  \frac{|\tilde\Upsilon_{xy}|}{2n} =1-\displaystyle{\Ex_{xy\sim \mu}\Ex_{\star\in [2n]}} \eps_{xy}(\star)  \;,
 \end{equation}
 where
 \begin{equation}\label{eq:eps_xy(star)}
 \eps_{xy}(\star)=\frac{|\tilde\Upsilon_{xy}^{O_x,O_y}|}{|O_x\cap O_y|}
 \end{equation}
 and $O_x$ and $O_y$ are the unique $\sigma_x$ and $\sigma_y$ orbits for which $\star\in O_x\cap O_y$.
 
Recall that according to the procedural sampling described in  Claim \ref{claim:procedural_sampling_perm_strat}, any pair $\vec v\in B_x^{-}, \vec u\in B^{-}_y$ of Fourier basis elements supported on $\sigma_x$-orbit $O_x$ and $\sigma_y$-orbit $O_y$ respectively, is sampled with probability $\frac{2\langle \vec v|\vec u\rangle^2}{|O_x\cap O_y|}$ (conditional on the sampled $\star$ being in $O_x\cap O_y$), and it induces an ``assignment'' $\gamma \colon S_{xy}\cup \{\sJ\}\to \FF_2$ for which $\gamma(\sJ)=1$.
 Let $\Upsilon^{O_x,O_y}(\alpha)$ be the set of all such $(\vec v,\vec u)$ for which 
  $\gamma$ does not satisfy a specific linear constraint $\alpha\in \FF_2^{S_{xy}\cup \{\sJ\}}$, namely $\sum_{\sX\in S_{xy}\cup\{\sJ\}} \gamma(\sX)\alpha(\sX)\neq 0$. 
  Let  
 \[
 \Upsilon^{O_x,O_y}_{xy} = \bigcup_{\alpha\in L^{O_x,O_y}_{xy}} \Upsilon^{O_x,O_y}(\alpha)
 \]
 be the collection of Fourier basis pairs $(\vec v,\vec u)\in B_x^-\times B_y^-$ supported on $O_x,O_y$, such that the assignment $\gamma$ induced by the pair is rejected by $D_{xy}$. 
 Finally let 
 \[
\Upsilon_{xy}=\bigcup_{O_x,O_y}\Upsilon^{O_x,O_y}_{xy}\subseteq B^-_x\times B^-_y,
 \]
 which is exactly the collection of possible sampled pairs by the quantum strategy $\rho$ induced by $\sigma$, given that the edge $xy$ was asked,  that produce an assignment $\gamma$ satisfying $D_{xy}(\gamma)=0$. 
 Hence,
\begin{equation}\label{eq:tg-1}
\val(\cG,\rho)=1-\displaystyle{\Ex_{xy\sim \mu}\sum_{(\vec v,\vec u)\in \Upsilon_{xy}}\frac{\langle \vec v|\vec u\rangle^2}{n}}= 1-\displaystyle{\Ex_{xy\sim \mu}\Ex_{\star\in [2n]}} \delta_{xy}(\star)\;,
\end{equation}
where 
\begin{equation}\label{eq:delta_xy(star)}
\delta_{xy}(\star) =\sum_{(\vec v,\vec u)\in \Upsilon^{O_x,O_y}_{xy}}\frac{2\langle \vec v|\vec u\rangle ^2}{|O_x\cap O_y|}
\end{equation}
and $O_x,O_y$ are the unique orbits for which $\star\in O_x\cap O_y$.
Observe the following claim.
\begin{claim}\label{claim:tg}
    For every $\alpha \in \mathbb{F}_2^{S_{xy}\cup\{\sJ\}}$,  $\sigma_x$-orbit $O_x$ and $\sigma_y$-orbit $O_y$ such that $O_x\cap O_y\neq \emptyset$, we have
\begin{equation}\label{eq:main_bound_prop_z-alinged_test_and_game}
    \sum_{(\vec v,\vec u)\in \Upsilon^{O_x,O_y}(\alpha)}\langle \vec v|\vec u\rangle ^2\,\leq\, \frac{1}{2} \big|\tilde{\Upsilon}^{O_x,O_y}(\alpha)\big| \;.
\end{equation}
\end{claim}
Let us show how to deduce our conclusion assuming the  validity of this claim, and then we will address the claim itself.
Given $\star\in O_x\cap O_y$, we have
\[ 
\begin{split}
    \delta_{xy}(\star)&\quad\overset{\eqref{eq:delta_xy(star)}}{=}\sum_{(\vec v,\vec u)\in \Upsilon^{O_x,O_y}_{xy}}\frac{2\langle \vec v|\vec u\rangle ^2}{|O_x\cap O_y|}\\
    &\quad\  \leq\, \sum_{\alpha \in L_{xy}^{O_x,O_y}}\ \sum_{(\vec v,\vec u)\in \Upsilon^{O_x,O_y}(\alpha)}\frac{2\langle \vec v|\vec u\rangle ^2}{|O_x\cap O_y|}\\
    &\overset{{\rm Claim}\ \ref{claim:tg}}{\leq} \sum_{\alpha \in L_{xy}^{O_x,O_y}}\ \frac{|\tilde \Upsilon^{O_x,O_y}(\alpha)|}{|O_x\cap O_y|}\ .
\end{split}\]
Now, for every $\alpha\in L_{xy}^{O_x,O_y}$,  we have by \eqref{eq:def_upsilon_xy_local} that $\tilde\Upsilon^{O_x,O_y}(\alpha)\subseteq \tilde\Upsilon^{O_x,O_y}_{xy}$, which using the above analysis means 
\[
\delta_{xy}(\star)\leq \sum_{\alpha \in L_{xy}^{O_x,O_y}}\ \frac{|\tilde \Upsilon^{O_x,O_y}_{xy}|}{|O_x\cap O_y|}\leq 2^{2\Lambda+1}\cdot \eps_{xy}(\star),
\]
where the last inequality uses \eqref{eq:eps_xy(star)} and  the fact $|L^{O_x,O_y}_{xy}|\leq|\FF_2^{S_{xy}\cup \{\sJ\}}|\leq  2^{2\Lambda+1}$. 
Therefore,
\[
\begin{split}
1-\val(\mathcal{G},\rho)\overset{\eqref{eq:tg-1}}{=}    \Ex_{xy\sim \mu}\Ex_{\star\in [2n]}[\delta_{xy}(\star)]\leq 2\cdot 2^{2\Lambda}\Ex_{xy\sim \mu}\Ex_{\star\in [2n]}[\eps_{xy}(\star)]\overset{\eqref{eq:tg-2}}{=}2\cdot 2^{2\Lambda}(1-\val( \tilde{\mathcal{T}},\sigma)),
\end{split}
\]
proving \eqref{eq:sufficient_condition_Prop_8.4} and thus the proposition. 
\end{proof}

\begin{proof}[Proof of Claim~\ref{claim:tg}]
Let $\tilde{\Gamma}^{O_x,O_y}(\alpha)$ be the complement of $\Upsilon^{O_x,O_y}(\alpha)$ in $O_x\cap O_y$, namely 
\[
\tilde{\Gamma}^{O_x,O_y}(\alpha)= (O_x\cap O_y) \setminus \tilde\Upsilon^{O_x,O_y}(\alpha).
\]
Recall the notation $\sX^\alpha$ and $\sY^\alpha$ from \eqref{eq:def_sX^alpha,sY^alpha}.
For $\star\in O_x\cap O_y$, we have $\sigma(\sX^{\alpha}).\star\in O_x$ and $\sigma(\sY^\alpha).\star\in O_y$ by the fact $O_x$ is a $\sigma_x$-orbit and $O_y$ is a $\sigma_y$-orbit. If $\star\in \tilde\Gamma^{O_x,O_y}(\alpha)$, then by its definition we have $\diamond:=\sigma(\sX^\alpha).\star=\sigma(\sY^\alpha).\star$, and $\diamond\in O_x\cap O_y$.  As $\sigma(\sX^\alpha)$ and $\sigma(\sY^\alpha)$ are involutions, we have $\sigma(\sX^\alpha).\diamond=\star=\sigma(\sY^\alpha).\diamond$, which implies $\diamond\in \tilde\Gamma^{O_x,O_y}(\alpha)$.  All in all, $\tilde{\Gamma}^{O_x,O_y}(\alpha)$ is invariant under the actions of $\sX^\alpha$ and $\sY^\alpha$ --- graphically, if we add $\sX^\alpha$ and $\sY^\alpha$ labels to the Schrier graph of the action $\sigma$ on $[2n]$, then they agree on every $\star\in \tilde \Gamma^{O_x,O_y}(\alpha)$ and are contained in it.

For $(\vec v,\vec u) \in \Upsilon^{O_x,O_y}(\alpha)$  we know that  $\sigma(\sX^\alpha)\vec v =\pm\vec v$ while $\sigma(\sY^\alpha)\vec u=\mp\vec u$. 
This means that for every $\sY^\alpha$-labeled edge between vertices in $\tilde \Gamma^{O_x,O_y}(\alpha)$, $\vec v$ and $\vec u$ agree on one of its endpoints and evaluate to opposites on the other endpoint. 
This implies
\[
\begin{split}
    \langle \vec v|\vec u\rangle &=\sum_{\star\in O_x\cap O_y}\vec v(\star)\vec u(\star)\\
    &=\sum_{\star\in \tilde \Upsilon^{O_x,O_y} (\alpha)}\vec v(\star)\vec u(\star)+\sum_{\star\in \tilde\Gamma^{O_x,O_y}(\alpha)}\vec v(\star)\vec u(\star)\\
    &=\sum_{\star\in \tilde \Upsilon^{O_x,O_y} (\alpha)}\vec v(\star)\vec u(\star)+\frac{1}{2}\sum_{\star\in \tilde\Gamma^{O_x,O_y}(\alpha)}\underbrace{\vec v(\star)\vec u(\star)+\vec v(\sigma(\sY^{\alpha}).\star)\vec u(\sigma(\sY^\alpha).\star)}_{=0}\\
    &=\langle \vec v\cdot {\bf 1}_{\tilde \Upsilon^{O_x,O_y} (\alpha)}|\vec u \rangle,
\end{split}
\]
where  ${\bf 1}_{\tilde \Upsilon^{O_x,O_y}}$ is the indicator function of the set $\tilde \Upsilon^{O_x,O_y}$, and $\cdot$ is the pointwise product of functions.
Thus,
\[
\begin{split}
    \sum_{(\vec v,\vec u)\in \Upsilon^{O_x,O_y}(\alpha)} \langle \vec v|\vec u\rangle^2&=\sum_{\vec v\colon \textrm{Supp}(\vec v)=O_x}\left(\sum_{\substack{\vec u \colon\\  (\vec v,\vec u)\in \Upsilon^{O_x,O_y}(\alpha)}} \langle \vec v\cdot {\bf 1}_{\tilde \Upsilon^{O_x,O_y} (\alpha)} |\vec u\rangle^2\right)\\
    &\leq \sum_{\vec v\colon \textrm{Supp}(\vec v)=O_x}\left(\sum_{\vec u \colon \textrm{Supp}(\vec u)=O_y} \langle \vec v\cdot {\bf 1}_{\tilde \Upsilon^{O_x,O_y} (\alpha)} |\vec u\rangle^2\right)\\
    &=(*),
\end{split}
\]
and since ${\bf 1}_{\tilde\Upsilon^{O_x,O_y}(\alpha)}$ is supported on $O_y$, we have 
\[
\sum_{\vec u \colon \textrm{Supp}(\vec u)=O_y} \langle \vec v\cdot {\bf 1}_{\tilde \Upsilon ^{O_x,O_y}(\alpha)} |\vec u\rangle^2=\|\vec v\cdot {\bf 1}_{\tilde \Upsilon^{O_x,O_y} (\alpha)} \|^2=\frac{|\tilde \Upsilon^{O_x,O_y} (\alpha)|}{|O_x|}\;.
\]
Since there are $\nicefrac{|O_x|}{2}$ many $\vec v\in B^{-}_x$ which are supported on $O_x$, we deduce that
$(*)\leq \nicefrac{|\tilde \Upsilon^{O_x,O_y}(\alpha)|}{2}$,
as desired.
\end{proof}

A short detour into (quite elementary) group stability results is needed for proving  Proposition \ref{prop:perturbing_to_satisfy_1to3}.
\subsection{Interlude --- Group stability}\label{sec:interlude_stability}

\begin{claim} [Almost involutions are close to involutions]\label{claim:fixing_to_involution}
\label{claim:2-torsion-Sym} Let $\zeta\in\Sym\left(X\right)$. Then,
there is $\tau\in\Sym\left(X\right)$ such that $\tau^{2}=\Id$ and
$d_{H}\left(\zeta,\tau\right)= d_{H}\left(\zeta^{2},\Id\right)$. 
\end{claim}
\begin{proof}
Let $W=\left\{ x\in X \mid\zeta^{2} .x=x\right\} $.
Note that the restriction $\zeta|_{W}$ is an involution $W\rightarrow W$.
Define 
\[
\tau\left(x\right)=\begin{cases}
\zeta\left(x\right) & x\in W,\\
x & x\notin W.
\end{cases}
\]
Then $\tau^{2}=\Id$ and $d_{H}\left(\zeta,\tau\right)=1-\frac{|W|}{|X|}=d_{H}\left(\zeta^{2},\Id\right)$.
\end{proof}

\begin{claim}[Almost fixed point free are close to fixed point free]\label{claim:fixing_fixed_point_free}
        Let $\zeta\in \Sym(X)$ be an involution. Assume $d_H(\zeta,\Id)\geq 1-\eps$. Then, there is an involution with no fixed points $\tau\in \Sym(Y)$, where $|Y|=2\cdot\lceil\nicefrac{|X|}{2}\rceil$, such that $d_H(\zeta,\tau)\leq 2\eps$.
\end{claim}
\begin{proof}
    Let $W=\{x\in X\mid \zeta.x=x\}$. Since $d_H(\zeta,\Id)=1-\nicefrac{|W|}{|X|}$, we can deduce that $|W|\leq \eps|X|$. If $W$ is odd, add a vertex to it and make it even. Now that $W$ is even sized, we can choose any perfect matching on it and define $\tau$ to be the involution induced by this perfect matching. Extend $\tau$ to the rest of $X$ by letting it act as $\zeta$ on the vertices out of $W$. Then the resulting $\tau$ is an involution with no fixed points, and 
    \[
d_H(\zeta,\tau)\leq \frac{|W|+1}{|X|}\leq 2\eps.
    \]
\end{proof}

\begin{claim}[Almost commuting involutions are close to commuting involutions]\label{claim:fixing_involutions_to_commute}
    Given two involutions $\zeta,\tau\in \Sym(X)$, there is an involution $\tau'\in \Sym(X)$ such that
    \[
[\tau',\zeta]=\Id\quad {\rm and}\quad d_H(\tau,\tau')\leq d_H([\tau,\zeta],\Id). 
    \]
    \end{claim}
    \begin{proof}
   Let $\delta=d_H([\tau,\zeta],\Id)$, and let $W=\{x\in X \mid \zeta \tau .x=\tau \zeta .x\}$. Then $|W|= (1-\delta)|X|$. Moreover, for every $x
    \in W$, 
    \[
    \begin{split}
\zeta \tau.(\zeta. x)&=\zeta .(\tau \zeta .x)=\zeta^2 \tau .x=\tau .x=\tau \zeta .(\zeta .x)\implies \zeta .x\in W;\\
\tau \zeta. (\tau .x)&=\tau .(\zeta \tau .x)= \tau^2 \zeta .x=\zeta.x=\zeta \tau .(\tau .x)\implies \tau .x\in W.
\end{split}
    \]
    Hence, $W$ is invariant under the actions of $\zeta$ and $\tau$. Thus, the following is well defined
    \[
        \tau'(x)=\begin{cases}
            \tau(x) & x\in W,\\
            x & x\notin W.
        \end{cases}
    \]
    We therefore can conclude that $d_H(\tau,\tau')\leq 1-\frac{|W|}{|X|}=\delta$, and 
    \[\begin{split}
        \forall x\in W\ \colon\ \  \tau'\underbrace{\zeta.x}_{\in W}&=\tau \zeta. x=\zeta \tau .x=\zeta \tau'. x;\\
        \forall x\notin W\ \colon\ \ \tau'\underbrace{\zeta. x}_{\notin W}&= \zeta .x=\zeta \tau'.x,
    \end{split}
    \]
    which implies $[\tau',\zeta]=\Id$ as required.
 \end{proof}
 \begin{rem}
     Claim \ref{claim:fixing_involutions_to_commute} can be proved even without assuming that the permutations are involutions (while allowing both permutations to change and not only $\tau$). This was originally proved in \cite{ArzhantsevaPaunescu} and was later reproved with effective bounds in \cite{BeckerMosheiff}.
 \end{rem}

The following is a special case of item $1$ of Theorem 2 in \cite{GlebskyRivera}.
\begin{claim}[Almost actions of finite groups are close to actual actions]\label{claim:Glebsky_Rivera}
    Let $G$ be a finite group, and let $f\colon G\to \Sym(X)$ be a function. Assume that for every $g,h\in G$, $d_H(f(gh),f(g)f(h))\leq \eps$. Then, there is a homomorphism $\varphi\colon G\to \Sym(X)$ such that 
    \[
\forall g\in G\ \colon \ \ d_H(f(g),\varphi(g))\leq |G|^2\eps.
    \]
\end{claim}
\begin{proof}
    Let $x\in X$ be a vertex such that 
    \begin{equation}\label{eq:assumption_on_x}
        \forall g,h\in G\ \colon \ \ f(gh).x=f(g)f(h).x.
    \end{equation}
    Then, if $O\subseteq X$ is the orbit of $x$ with respect to the action of (the group generated by) $\Img(f)$, then for every $y\in O$ we have $f(gh).y=f(g)f(h).y$. To prove that, it is enough to prove it for a single step, namely for  $y=f(k).x$ for some $k\in G$ (the rest is by induction). So, for every $g,h\in G$, we have
    \[
f(gh).y=f(gh)f(k).x=f(ghk).x=f(g)f(hk).x=f(g)f(h)f(k).x=f(g)f(h).y,
    \]
    where all equalities are either $f(k).x=y$ or \eqref{eq:assumption_on_x}. This means that the restriction of $f$ to its action on $O$ induces a homomorphism from $G$ to $\Sym(O)$.
    
    We assumed for every $g,h\in G$, that $d_H(f(gh),f(g)f(h))\leq \eps$. Therefore, there are at most $|G|^2\eps$ many vertices $z$ in $X$ for which there is some $g,h\in G$ such that $f(gh).z\neq f(g)f(h).z$. So, the union of orbits as $O$ consists of at least $(1-|G|^2\eps)|X|$ of the vertices in $X$. So, if we define $\varphi(g).x=f(g).x$ whenever $x$ satisfies \eqref{eq:assumption_on_x}, and $\varphi(g).x=x$ otherwise, then $\varphi$ is an action of $G$ it satisfies for every $g\in G$, $d_H(f(g),\varphi(g))\leq |G|^2\eps$.
\end{proof}

\subsection{Proof of Proposition \ref{prop:perturbing_to_satisfy_1to3}}
Throughout this proof, we use the term \emph{triangle inequality} for \textbf{iterative applications} of the Hamming distance's triangle inequality,  as was used  in the proof of Claim \ref{claim:test_weighted_edit_distance_implies_val_close}.

   Let $\sigma\colon \tilde S\to \Sym(n)$ be a finitely described strategy  for $\tilde \cT$, with $\val(\tilde \cT,\sigma)\geq 1-\eps$. By \labelcref{clause:Check_1_in_associated_test}, $d_H(\sigma(\sJ)^2,\Id)\leq \eps$. By Claim \ref{claim:fixing_to_involution}, there is a function $\sigma_1\colon \tilde S\to \Sym(n)$ that agrees with $\sigma$ on all $\sX\in S$, but $\sigma_1(\sJ)$ is an involution and $d_H(\sigma(\sJ),\sigma_1(\sJ))\leq \eps$. Again by \labelcref{clause:Check_1_in_associated_test}, $d_H(\sigma(\sJ),\Id)\geq 1-\eps$, and by the triangle inequality $d_H(\sigma_1(\sJ),\Id)\geq 1-2\eps$. Let $m=\lceil \frac{n}{2}\rceil$. Then, by Claim \ref{claim:fixing_fixed_point_free}, there is a function $\sigma_2\colon \tilde S\to \Sym(2m)$ such that for every $\star\in [n]\subseteq [2m]$, we have $\sigma_2(\sX).\star=\sigma_1(\sX).\star=\sigma(\sX).\star$, and $\sigma_2(\sJ)$ is a fixed point free involution with $d_H(\sigma_1(\sJ),\sigma_2(\sJ))\leq 4\eps$. All in all, for every $\sX\in S\cup \{\sJ\}$ we have $d_H(\sigma_2(\sX),\sigma(\sX))\leq 5\eps$.

   Now, let $\eps_x$ be the losing probability when an edge containing $x$ is sampled in $\tilde \cT$. Then, $\Ex_{x\sim \mu}[\eps_x]=\eps$. Furthermore, by \labelcref{clause:Check_2_in_associated_test}, for every $\sX\in S_x$ we have $d_H(\sigma(\sX)^2,\Id)\leq \eps_x$. Therefore, by the triangle inequality, $d_H(\sigma_2(\sX)^2,\Id)\leq d_H(\sigma_2(\sX)^2,\sigma(\sX)^2)+d_H(\sigma(\sX)^2,\Id)\leq 10\eps+\eps_x$. By Claim \ref{claim:fixing_to_involution}, there is a function $\sigma_3\colon \tilde S\to \Sym(2m)$  such that  $\sigma_3(\sJ)=\sigma_2(\sJ)$, and for every $x\in V$ and $\sX\in S_x$ both $\sigma_3(\sX)^2=\Id$ and $d_H(\sigma_3(\sX),\sigma_2(\sX))\leq 10\eps+\eps_x$. Now, by \labelcref{clause:Check_1_in_associated_test}, for every $\sX\in S_x$ we have $d_H(\sigma(\sX)\sigma(\sJ),\sigma(\sJ)\sigma(\sX))\leq \eps_x$. Thus, by the triangle inequality, $d_H(\sigma_3(\sX)\sigma_3(\sJ),\sigma_3(\sJ)\sigma_3(\sX))\leq 30\eps+2\eps_x$. Since $\sigma_3(\sX)$ and $\sigma_3(\sJ)$ are involutions, we may apply Claim \ref{claim:fixing_involutions_to_commute} and find a function $\sigma_4\colon \tilde S\to \Sym(2m)$ such that $\sigma_4(\sJ)=\sigma_3(\sJ)$, $\sigma_4(\sX)\sigma_4(\sJ)=\sigma_4(\sJ)\sigma_4(\sX)$ and $d_H(\sigma_4(\sX),\sigma_3(\sX))\leq 30\eps+2\eps_x$. Again by the triangle inequality, we deduce that $d_H(\sigma(\sJ),\sigma_4(\sJ))\leq 5\eps$, and that for $\sX\in S_x$ we have $d_H(\sigma(\sX),\sigma_4(\sX))\leq 40\eps +3\eps_x.$

   Recall that $S_x=\{\sX^{x,i}\mid i\leq \ell(x)\}$ (ignoring readability). We extend $\sigma_4$ to $\FF_2^{S_x}$ in the following way: For $\xi\colon S_x\to \FF_2$, let
   \[
\sigma_4(\xi)=\prod_{i=1}^{\ell(x)}\sigma_4(\sX)^{\xi(\sX)},
   \]
   where the product is ordered according to the index $i$. By \labelcref{clause:Check_2_in_associated_test}, for every $\sX,\sX'\in S_x$ we have $$d_H(\sigma(\sX)\sigma(\sX'),\sigma(\sX')\sigma(\sX))\leq \eps_x.$$
   Thus, by the triangle inequality, $d_H(\sigma_4(\sX)\sigma_4(\sX'),\sigma_4(\sX')\sigma_4(\sX))\leq 160\eps+ 13\eps_x$. Note also that the images of $\sigma_4$ are involutions. Therefore, for every $\xi_1,\xi_2\colon S_x\to \FF_2$ we have 
   \[
d_H(\sigma_4(\xi_1+\xi_2),\sigma_4(\xi_1)\sigma_4(\xi_2))\leq \ell(x)^2(160\eps+ 13\eps_x)\leq \Lambda^2(160\eps+ 13\eps_x).
   \]
   By Claim \ref{claim:Glebsky_Rivera} applied to each $\FF_2^{S_x}$ individually, we have a function $\sigma_5\colon \tilde S\to \Sym(2m)$ such that $\sigma_5(S_x)$ commutes and consists only of involutions, and $$d_H(\sigma_4(\sX),\sigma_5(\sX))\leq |\FF_2^{S_x}|^2 \cdot\Lambda^2(160\eps+13\eps_x)\leq 2^{2\Lambda}\cdot\Lambda^2(160\eps+13\eps_x).$$ Furthermore, because of the way $\sigma_5$ is constructed in the proof of Claim \ref{claim:Glebsky_Rivera}, by letting $\sigma_5(\sJ)=\sigma_4(\sJ)$, the image of $\sJ$ is still a central involution with no fixed points. So, $\sigma_5$ always passes \labelcref{clause:Check_1_in_associated_test} and \labelcref{clause:Check_2_in_associated_test}. 

   Since $\sigma$ passes \labelcref{clause:Check_3_in_associated_test} with probability $\eps_x$ when an edge with $x\in V$  is sampled, we deduce that every readable variable $\sX\in S_x^\frR$ satisfies
\[
\Pro_{\star\in [n]}[\sigma(\sX).\star=\star\ \textrm{or}\ \sigma(\sX).\star=\sigma(\sJ).\star]\geq 1-\eps_x.
\]
Hence, by the triangle inequality, 
   \[
   \begin{split}
       \Pro_{\star\in [2m]}[\sigma_5(\sX).\star=\star\ \textrm{or}\ \sigma_5(\sX).\star=\sigma_5(\sJ).\star]&\geq 1-\eps_x-d_H(\sigma_5(\sJ),\sigma(\sJ))-d_H(\sigma_5(\sX),\sigma(\sX))\\
       &\geq 1-\eps_x-5\eps-2^{2\Lambda} \cdot \Lambda^2(160\eps+13\eps_x)\\
       &\geq 1-2^{2\Lambda} \cdot \Lambda^2(165\eps+14\eps_x).
   \end{split}
   \]
   By the fact that $\sigma_5(S_x)$ commutes for every $x\in V$,  the property $\sigma_5(\sX).\star=\star$ implies $\sigma_5(\sX).\diamond=\diamond$ for every $\diamond\in O_x$, where $O_x$ is the orbit of $\star$ induced by $\sigma_5(S_x)$ (which we usually denoted as $(\sigma_5)_x$ beforehand). Since $\sigma_5(\sJ)$ commutes with all images, the same is true for the property $\sigma_5(\sX).\star=\sigma_5(\sJ).\star$.  
   Hence, in each orbit of $\sigma_5(S_x)$, either all vertices $\star$ satisfy $\sigma_5(\sX).\star=\star\ \textrm{or}\ \sigma_5(\sX).\star=\sigma_5(\sJ).\star$ or none of them. 
   We can thus define a function $\sigma_6$ that agrees with $\sigma_5$ on the orbits in which all readable variables $\sX\in S_x^\frR$ act appropriately, and on the other orbits we just define for every $\sX\in S_x$ that $\sigma_6(\sX).\star=\star$ (note that we keep $\sJ$ acting the same no matter what). So, we have 
   \[
   \begin{split}
       \forall x\in V,\ \sX\in S_x \ \colon \ \ d_H(\sigma_6(\sX),\sigma_5(\sX))&\leq |S_x^\frR|\cdot 2^{2\Lambda} \cdot \Lambda^2(165\eps+14\eps_x)\\
       &\leq 2^{2\Lambda} \cdot \Lambda^3(165\eps+14\eps_x),
   \end{split}
\]
and now $\sigma_6$ always passes \labelcref{clause:Check_1_in_associated_test}, \labelcref{clause:Check_2_in_associated_test} and \labelcref{clause:Check_3_in_associated_test}. By the triangle inequality, $d_H(\sigma_6(\sJ),\sigma(\sJ))\leq 5\eps$ and for every $\sX\in S_x$ we have
\[
\begin{split}
    d_H(\sigma_6(\sX),\sigma(\sX))&\leq 2^{2\Lambda} \cdot \Lambda^3(165\eps+14\eps_x)+2^{2\Lambda}\cdot\Lambda^2(160\eps+13\eps_x)\\
    &+30\eps+2\eps_x+10\eps+\eps_x+5\eps\\
    &\leq 2^{2\Lambda}\Lambda ^3(370\eps+30\eps_x).
\end{split}
\]
Choosing $\sigma'=\sigma_6$ and $C_0=370$, we deduce the claim.

\begin{rem}\label{rem:better_params}
    How better can the parameters in Propositions \ref{prop:perturbing_to_satisfy_1to3}, \ref{prop:associated_quantum_strat_has_value_at_least_as_original} and \ref{prop:significance_of_associated_test} be? We know how to  make the parameters in Proposition \ref{prop:perturbing_to_satisfy_1to3}  polynomial in $\Lambda$ by implementing a Hadmamard code in every vertex $x\in V$ and using \cite{BC22} or \cite{GowersHatami} in the analysis instead of \cite{GlebskyRivera} --- since we did not see any theoretical gain by doing that, and it was a lengthier proof, we decided not to implement this parameter improvement. We are also not sure whether the dependence in Proposition \ref{prop:associated_quantum_strat_has_value_at_least_as_original} is needed or not. Finally, the significance analysis performed in Section \ref{sec:Robustness} is very crude. Can a better analysis allow to remove the dependence on $\Lambda$ altogether, at least for some rich enough family of games?
\end{rem}

\bibliographystyle{plain}
\bibliography{Bibliography}

\end{document}